\documentclass{amsart}

\usepackage[foot]{amsaddr}
\title[Ribbon structures]{Ribbon structures of the Drinfeld center of a finite tensor category}
\author[K.~Shimizu]{Kenichi Shimizu}
\email{kshimizu@shibaura-it.ac.jp}
\address{Department of Mathematical Sciences \\
  Shibaura Institute of Technology \\
  307 Fukasaku, Minuma-ku, Saitama-shi, Saitama 337-8570, Japan.}
\date{}

\usepackage{amsmath,amssymb,amscd}
\usepackage{dsfont}
\usepackage{upgreek}
\usepackage{stmaryrd}
\usepackage{url}
\usepackage{mathrsfs}
\usepackage{hyperref}
\usepackage{adjustbox}
\usepackage{tikz}
\usepackage{tikz-cd}
\usetikzlibrary{calc}
\usepackage{amsthm}
\numberwithin{equation}{section}
\newtheorem{counter}{}[section]
\theoremstyle{definition}
\newtheorem{definition}         [counter]{Definition}

\newtheorem*{notation*}         {Notation}
\theoremstyle{plain}
\newtheorem{lemma}              [counter]{Lemma}

\newtheorem{theorem}            [counter]{Theorem}
\newtheorem{corollary}          [counter]{Corollary}

\newtheorem*{theorem*}          {Theorem}
\theoremstyle{remark}
\newtheorem{remark}             [counter]{Remark}

\newcommand{\xarrlen}{2em}
\newcommand{\xarr}[1]{\xrightarrow{\makebox[\xarrlen]{$\scriptstyle #1$}}}
\newcommand{\xarq}[1]{\xrightarrow{\quad #1 \quad}}
\newcommand{\Kappa}{\mathord{\mathrm{K}}}
\newcommand{\bfk}{\mathds{k}} 
\newcommand{\din}{\mathtt{i}} 
\newcommand{\id}{\mathrm{id}}
\newcommand{\op}{\mathrm{op}}
\newcommand{\Hom}{\mathrm{Hom}}

\newcommand{\Nat}{\mathrm{Nat}}
\newcommand{\radj}{\mathrm{ra}}
\newcommand{\rradj}{\mathrm{rra}}

\newcommand{\ladj}{\mathrm{la}}
\newcommand{\lladj}{\mathrm{lla}}

\newcommand{\Fun}{\mathrm{Fun}}
\newcommand{\Lex}{\mathrm{Lex}}

\newcommand{\lmod}[1]{{#1}\mbox{-}\mathrm{mod}}

\newcommand{\bimod}[2]{{#1}\mbox{-}\mathrm{mod}\mbox{-}{#2}}
\newcommand{\fdVec}{\mathbf{Vec}}
\newcommand{\unitobj}{\mathord{\mathds{1}}}
\newcommand{\eval}{\mathrm{ev}}
\newcommand{\coev}{\mathrm{coev}}
\newcommand{\rev}{\mathrm{rev}}
\newcommand{\env}{\mathrm{env}}
\newcommand{\catactl}{\mathbin{\ogreaterthan}}
\newcommand{\catactr}{\mathbin{\olessthan}}
\newcommand{\iHom}{\underline{\mathrm{Hom}}}
\newcommand{\Inv}{\mathrm{Inv}}
\newcommand{\Nak}{\mathds{N}}
\newcommand{\relmod}{\mu}
\newcommand{\modobj}{D}
\newcommand{\relmodb}{\boldsymbol{\relmod}}
\newcommand{\modobjb}{\boldsymbol{\modobj}}
\begin{document}

\maketitle

\begin{abstract}
  We classify the ribbon structures of the Drinfeld center $\mathcal{Z}(\mathcal{C})$ of a finite tensor category $\mathcal{C}$. Our result generalizes Kauffman and Radford's classification result of the ribbon elements of the Drinfeld double of a finite-dimensional Hopf algebra. Our result implies that $\mathcal{Z}(\mathcal{C})$ is a modular tensor category in the sense of Lyubashenko if $\mathcal{C}$ is a spherical finite tensor category in the sense of Douglas, Schommer-Pries and Snyder.
\end{abstract}

\section{Introduction}

A braided monoidal category is a monoidal category $\mathcal{B}$ equipped with an isomorphism $\sigma_{X,Y}: X \otimes Y \to Y \otimes X$ satisfying the hexagon axiom, and a ribbon category is a braided rigid monoidal category $\mathcal{B}$ equipped with a {\em ribbon structure} (also called a {\em twist}), that is, a natural isomorphism $\theta: \id_{\mathcal{B}} \to \id_{\mathcal{B}}$ satisfying
\begin{equation*}
  \theta_{X \otimes Y} = (\theta_X \otimes \theta_Y) \circ \sigma_{Y, X} \circ \sigma_{X,Y}
  \quad \text{and} \quad
  (\theta_X)^* = \theta_{X^*}
\end{equation*}
for all $X, Y \in \mathcal{B}$, where $(-)^*$ is the duality functor; see, {\it e.g.}, \cite{MR3242743}. These notions are used, for example, to formulate and construct several kinds of topological invariants or, more generally, topological quantum field theory.

Given a rigid monoidal category $\mathcal{C}$, we have a braided rigid monoidal category $\mathcal{Z}(\mathcal{C})$ called the {\em Drinfeld center} of $\mathcal{C}$. In this paper, we classify the ribbon structures of the Drinfeld center of a {\em finite tensor category} in the sense of \cite{MR2119143}. A typical example of a finite tensor category is the category $H\mbox{-mod}$ of finite-dimensional left modules over a finite-dimensional Hopf algebra $H$. As is well-known, the Drinfeld center of $H\mbox{-mod}$ is identified with the category of finite-dimensional modules over the Drinfeld double of $H$. Our result can be thought of as a categorical generalization of Kauffman and Radford's classification result of the ribbon elements of the Drinfeld double of a finite-dimensional Hopf algebra \cite{MR1231205}.

Etingof, Nikshych and Ostrik \cite{MR2097289} have introduced the {\em distinguished invertible object} $\modobj$ of a finite tensor category $\mathcal{C}$. Following \cite{MR2097289}, there is a natural isomorphism $\delta_X: \modobj \otimes X \to X^{****} \otimes \modobj$ ($X \in \mathcal{C}$). As this result generalizes the celebrated Radford $S^4$-formula, we call $\delta$ the {\em Radford isomorphism}. Our classification result claims that the ribbon structures of $\mathcal{Z}(\mathcal{C})$ are parametrized by `square roots' of the Radford isomorphism ({\it cf}. \cite[Theorem 3]{MR1231205}).

Our result yields a new example of `non-semisimple' modular tensor categories in the sense of Lyubashenko \cite{MR1324034,MR1352517,MR1354257,MR1862634}. If $\mathcal{B}$ is a braided finite tensor category, then the coend $F = \int^{X \in \mathcal{B}} X \otimes X^*$ has a canonical paring $\omega: F \otimes F \to \unitobj$ defined in terms of the braiding. We say that $\mathcal{B}$ is {\em non-degenerate} if $\omega$ is. A {\em modular tensor category} \cite{MR1862634} is a non-degenerate ribbon finite tensor category. The braided finite tensor category $\mathcal{Z}(\mathcal{C})$ is always non-degenerate by \cite{MR3996323} and \cite[Proposition 8.6.3]{MR3242743}, but it does not have a ribbon structure in general. Our result determines when $\mathcal{Z}(\mathcal{C})$ admits a ribbon structure, and hence is modular. For example, $\mathcal{Z}(\mathcal{C})$ is a modular tensor category if $\mathcal{C}$ is spherical in the sense of Douglas, Schommer-Pries and Snyder \cite[Definition 4.5.2]{2013arXiv1312.7188D}. This answers Open Problem (7) of \cite[Section 6]{MR2681261}.

\subsection*{Organization of this paper}

This paper is organized as follows: In Section~\ref{sec:preliminaries}, we collect basics on monoidal categories and their modules from \cite{MR1712872,MR3242743,MR3934626} and fix notations used throughout in this paper.

In Section~\ref{sec:bimodule-cat-center}, for two tensor functors $F, G: \mathcal{B} \to \mathcal{C}$ between finite tensor categories $\mathcal{B}$ and $\mathcal{C}$, the category $\mathcal{Z}(F, G)$ is introduced as a special case of the center of a bimodule category. Spelling out the definition, an object of this category is a pair $(V, \sigma)$ consisting of an object $V \in \mathcal{C}$ and a natural transformation
\begin{equation*}
  \sigma_X: V \otimes F(X) \to G(X) \otimes V \quad (X \in \mathcal{B})
\end{equation*}
satisfying certain conditions. The Drinfeld center is the case where $F$ and $G$ are the identity functor. Unlike the Drinfeld center, the category $\mathcal{Z}(F, G)$ does not have a tensor product. Though, for three tensor functors $F, G, H: \mathcal{B} \to \mathcal{C}$, one can define the tensor product $\otimes: \mathcal{Z}(G, H) \times \mathcal{Z}(F, G) \to \mathcal{Z}(F, H)$. These categories, as well as this tensor product, are useful to formulate our classification result.

The main result of Section~\ref{sec:bimodule-cat-center} is a monadic description of $\mathcal{Z}(F, G)$. Given tensor functors $F, G: \mathcal{B} \to \mathcal{C}$, one can define an algebra $A_{F,G} \in \mathcal{C} \boxtimes \mathcal{C}^{\rev}$ as a coend of a certain functor. The category $\mathcal{C} \boxtimes \mathcal{C}^{\rev}$ acts on $\mathcal{C}$ from the left in a natural way, and hence the algebra $A_{F,G}$ defines a monad on $\mathcal{C}$. We see that the Eilenberg-Moore category of this monad can be identified with $\mathcal{Z}(F, G)$.

In Section~\ref{sec:relative-modulus}, we first reformulate a result of \cite{MR2097289} with the techniques introduced in Section \ref{sec:bimodule-cat-center}. Let $\mathcal{C}$ be a finite tensor category, and let $\iHom$ be the internal Hom functor of the $\mathcal{C} \boxtimes \mathcal{C}^{\rev}$-module category $\mathcal{C}$.
We consider the algebra $A = \iHom(\unitobj,\unitobj)$ in $\mathcal{C} \boxtimes \mathcal{C}^{\rev}$.
It is known that there is an equivalence
\begin{equation*}
  \Kappa: \mathcal{C} \xrightarrow{\quad \approx \quad}
  (\text{the category of $A$-modules in $\mathcal{C} \boxtimes \mathcal{C}^{\env}$}),
  \quad V \mapsto (V \boxtimes \unitobj) \otimes A
\end{equation*}
of left module categories over $\mathcal{C} \boxtimes \mathcal{C}^{\rev}$. By the results of Section~\ref{sec:bimodule-cat-center}, we see that this equivalence induces an equivalence
\begin{equation*}
  \mathcal{Z}(\id_{\mathcal{C}}, S^4) \approx  (\text{the category of $A^{**}$-$A$-bimodules in $\mathcal{C} \boxtimes \mathcal{C}^{\env}$})
\end{equation*}
of categories, where $S = (-)^*$ is the left duality functor of $\mathcal{C}$. We define the {\em Radford object} to be the object $\modobjb_{\mathcal{C}} \in \mathcal{Z}(\id_{\mathcal{C}}, S^4)$ corresponding to the $A^{**}$-$A$-bimodule $A^*$ (Definition~\ref{def:rad-obj}). This object capsules the categorical analogue of the Radford $S^4$-formula given in \cite{MR2097289}; see Appendix~\ref{sec:def-rad-obj}.

For a tensor functor $F: \mathcal{B} \to \mathcal{C}$ whose right adjoint is exact, the relative modular object $\relmod_F \in \mathcal{C}$ is defined by $\relmod_F = F^{\rradj}(\unitobj)$, where $F^{\rradj}$ is a double right adjoint of $F$ \cite{MR3569179}.
As noted in \cite{MR3569179}, there is a canonical isomorphism $\gamma_F$ such that, in our notation, the pair $\relmodb_F := (\chi_F, \gamma_F)$ belongs to the category $\mathcal{Z}(F, F)$. The main result of Section~\ref{sec:relative-modulus} is that there is an isomorphism $F(\modobjb_{\mathcal{B}}) \cong \modobjb_{\mathcal{C}} \otimes \relmodb_F$ in $\mathcal{Z}(F, F S^4)$ (see Theorem~\ref{thm:relative-modulus} for the precise meaning). This result is established by utilizing the Nakayama functor introduced in \cite{MR4042867}.

The main result of this paper is stated and proved in Section~\ref{sec:main-result}.
We say that a pivotal structure $j$ of $\mathcal{Z}(\mathcal{C})$ is {\em ribbon} (Definition~\ref{def:ribbon-pivot}) if $u^{-1} j$, where $u$ is the Drinfeld isomorphism, is a ribbon structure of $\mathcal{Z}(\mathcal{C})$. Since every ribbon structure of $\mathcal{Z}(\mathcal{C})$ is obtained from a ribbon pivotal structure of $\mathcal{Z}(\mathcal{C})$ in this way, we actually classify ribbon pivotal structures of $\mathcal{Z}(\mathcal{C})$ in this paper.

Pivotal structures of $\mathcal{Z}(\mathcal{C})$ were classified in \cite{2016arXiv160805905S}.
Let $\mathrm{Aut}(\mathcal{C})$ denote the class of tensor autoequivalences of $\mathcal{C}$.
Given $F \in \mathrm{Aut}(\mathcal{C})$, we denote by $\widetilde{F}$ the braided tensor autoequivalence on $\mathcal{Z}(\mathcal{C})$ induced by $F$. In \cite{2016arXiv160805905S}, a bijection
\begin{equation*}
  \Theta: \Inv(\mathcal{Z}(F, G)) \to \Nat_{\otimes}(\widetilde{F}, \widetilde{G})
\end{equation*}
is constructed for $F, G \in \mathrm{Aut}(\mathcal{C})$ (see Theorem~\ref{thm:nat-tensor-bij}), where $\Inv(\mathcal{Z}(F, G))$ is the set of isomorphism classes of invertible objects of $\mathcal{Z}(F, G)$. Pivotal structures of $\mathcal{Z}(\mathcal{C})$ are classified by this bijection for $F = \id_{\mathcal{C}}$ and $G = S^2$.

The object $\modobjb_{\mathcal{Z}(\mathcal{C})}$ can be expressed by the bijection $\Theta$. We consider the relative modular object $\relmodb_U$ of the forgetful functor $U: \mathcal{Z}(\mathcal{C}) \to \mathcal{C}$. Theorem~\ref{thm:relative-modulus} expresses $\relmodb_U$ in terms $\modobjb_{\mathcal{Z}(\mathcal{C})}$ and $\modobjb_{\mathcal{C}}$. On the other hand, $\relmodb_U$ is expressed in terms of the half-braiding (Corollary~\ref{cor:rel-mod-obj-Dri}). By comparing these two expressions, we find that $\modobjb_{\mathcal{Z}(\mathcal{C})}$ is given by $\modobjb_{\mathcal{Z}(\mathcal{C})} = (\unitobj_{\mathcal{Z}(\mathcal{C})}, \Theta(\modobjb_{\mathcal{C}}))$ (Theorem~\ref{thm:Rad-obj-Dri-cen}).

The condition for a pivotal structure of $\mathcal{Z}(\mathcal{C})$ to be ribbon can be written in terms of $\modobjb_{\mathcal{Z}(\mathcal{C})}$ (see Lemma~\ref{lem:ribbon-pivot-finite}). By using the functorial property of the map $\Theta$, we see that the bijection $\Theta$ restricts to a bijection between the set
\begin{equation*}
  \{ [\boldsymbol{\beta}] \in \Inv(\mathcal{Z}(\id_{\mathcal{C}}, S_{\mathcal{C}}^2)) \mid S_{\mathcal{C}}^2(\boldsymbol{\beta}) \otimes \boldsymbol{\beta} \cong \modobjb_{\mathcal{C}} \}
\end{equation*}
of `square roots' of $\modobjb_{\mathcal{C}}$ and the set of ribbon pivotal structures of $\mathcal{Z}(\mathcal{C})$ (Theorem~\ref{thm:classification}). This generalizes a result of Kauffman and Radford to the setting of finite tensor categories \cite{MR1231205}.
Some applications are discussed at the end of this paper.

\subsection*{Acknowledgment}

The author is supported by JSPS KAKENHI Grant Number 16K17568 and JP20K03520.

\section{Preliminaries: Monoidal categories and their modules}
\label{sec:preliminaries}

\subsection{Adjunctions, ends and coends}

Our standard reference for category theory is \cite{MR1712872}.
The opposite category of a category $\mathcal{C}$ is denoted by $\mathcal{C}^{\op}$.
Given a functor $F$, we denote by $F^{\ladj}$ and $F^{\radj}$ a (fixed) left and right adjoint of $F$, respectively, provided that they exist. The following result is well-known:

\begin{lemma}
  Let $\mathcal{V}$, $\mathcal{X}$ and $\mathcal{Y}$ be categories. If a functor $F: \mathcal{X} \to \mathcal{Y}$ admits a left adjoint, then there are natural isomorphisms
  \begin{align}
    \label{eq:nat-adjunctions-1}
    \Nat(A, F B) & \cong \Nat(F^{\ladj} A, B), \\
    \label{eq:nat-adjunctions-2}
    \Nat(C F, D) & \cong \Nat(C, D F^{\ladj})
  \end{align}
  for functors $A : \mathcal{V} \to \mathcal{Y}$, $B : \mathcal{V} \to \mathcal{X}$,
  $C : \mathcal{Y} \to \mathcal{V}$ and $D : \mathcal{X} \to \mathcal{V}$.
\end{lemma}

The isomorphism \eqref{eq:nat-adjunctions-1} is induced by the adjunction isomorphism, and the latter one is constructed as follows: Let $\eta$ and $\varepsilon$ be the unit and the counit of the adjunction $F \dashv F^{\ladj}$, respectively. Given a natural transformation $\xi : C F \to D$, we define the natural transformation $\xi^{\sharp} : C \to D F^{\ladj}$ by
\begin{equation*}
  \xi^{\sharp}_Y = \Big(C(Y)
  \xrightarrow{\quad C(\eta_Y) \quad}
  C(F F^{\ladj}(Y))
  \xrightarrow{\quad \xi_{F^{\ladj}(Y)} \quad} D F^{\ladj}(Y) \Big)
\end{equation*}
for $Y \in \mathcal{Y}$. Given $\zeta \in \Nat(C, D F^{\ladj})$, we define $\zeta^{\flat} : C F \to D$ by
\begin{equation*}
  \zeta^{\flat}_X = \Big(
  C F(X) \xrightarrow{\quad \zeta_{F(X)} \quad} D F^{\ladj} F(X)
  \xrightarrow{\quad D(\varepsilon_X) \quad} D(X)
  \Big)
\end{equation*}
for $X \in \mathcal{X}$. The isomorphism \eqref{eq:nat-adjunctions-2} is given by the maps $\xi \mapsto \xi^{\sharp}$ and $\zeta \mapsto \zeta^{\flat}$.

The reader is assumed to be familiar with basic results and the standard integral notation for (co)ends \cite[IX]{MR1712872}. In this paper, the following formula for (co)ends will be important:

\begin{lemma}
  \label{lem:coend-adj}
  Let $T : \mathcal{A}^{\op} \times \mathcal{B} \to \mathcal{C}$ and $F : \mathcal{A} \to \mathcal{B}$ be functors, where $\mathcal{A}$, $\mathcal{B}$ and $\mathcal{C}$ be categories. We suppose that $F$ admits a left adjoint. Then we have
  \begin{equation}
    \label{eq:adj-coend-1}
    \int^{X \in \mathcal{A}} T(X, F(X))
    \cong \int^{Y \in \mathcal{B}} T(F^{\ladj}(Y), Y)
  \end{equation}
  meaning that the coend on the left-hand side exists if and only if that on the right-hand side exists and, if either of them exists, then they are canonically isomorphic. Similarly, we have a canonical isomorphism
  \begin{equation}
    \label{eq:adj-coend-2}
    \int_{X \in \mathcal{A}} T(F(X), X)
    \cong \int_{Y \in \mathcal{B}} T(Y, F^{\ladj}(Y)),
  \end{equation}
  in an analogous meaning as the case of coends.
\end{lemma}

The isomorphism \eqref{eq:adj-coend-1} is noted in \cite[Lemma 3.9]{MR2869176}, and the isomorphism \eqref{eq:adj-coend-2} is obtained by the dual argument.
For reader's convenience, we note how to obtain \eqref{eq:adj-coend-1}. Let $C_1$ and $C_2$ be coends of the left and the right-hand side of \eqref{eq:adj-coend-1}, respectively. We denote their universal dinatural transformation by
\begin{equation*}
  i^{(1)}_X : T(X, F(X)) \to C_1
  \quad \text{and} \quad
  i^{(2)}_Y : T(F^{\ladj}(Y), Y) \to C_2
\end{equation*}
for $X \in \mathcal{A}$ and $Y \in \mathcal{B}$. The isomorphism $\phi : C_1 \to C_2$ of the above lemma and its inverse are characterized by the equations
\begin{gather*}
  \phi \circ i^{(1)}_X = i^{(2)}_{F(X)} \circ T(\varepsilon_X, \id_{F(X)})
  \quad \text{and} \quad
  \phi^{-1} \circ i^{(2)}_Y = i^{(1)}_{F^{\ladj}(Y)} \circ T(\id_{F^{\ladj}(Y)}, \eta_Y^{})
\end{gather*}
for $X \in \mathcal{A}$ and $Y \in \mathcal{B}$, respectively.

\subsection{Monoidal categories}
\label{subsec:mon-cat}

A {\em monoidal category} \cite[VII.1]{MR1712872} is a category $\mathcal{C}$ endowed with a functor $\otimes: \mathcal{C} \times \mathcal{C} \to \mathcal{C}$ (called the {\em tensor product}), an object $\unitobj \in \mathcal{C}$ (called the {\em unit object}), and natural isomorphisms
\begin{equation*}
  (X \otimes Y) \otimes Z \cong X \otimes (Y \otimes Z)
  \quad \text{and} \quad
  \unitobj \otimes X \cong X \cong X \otimes \unitobj
  \quad (X, Y, Z \in \mathcal{C})
\end{equation*}
satisfying the pentagon and the triangle axioms. If these natural isomorphisms are identities, then $\mathcal{C}$ is said to be {\em strict}. In view of the Mac Lane coherence theorem, we may assume that all monoidal categories are strict.

\subsubsection{Monoidal functors}

Let $\mathcal{C}$ and $\mathcal{D}$ be monoidal categories. A {\em monoidal functor} \cite[XI.2]{MR1712872} from $\mathcal{C}$ to $\mathcal{D}$ is a functor $F: \mathcal{C} \to \mathcal{D}$ endowed with a morphism $F_0: \unitobj \to F(\unitobj)$ in $\mathcal{C}$ and a natural transformation
\begin{equation*}
  F_2(X, Y): F(X) \otimes F(Y) \to F(X \otimes Y) \quad (X, Y \in \mathcal{C})
\end{equation*}
satisfying certain axioms. A monoidal functor $F$ is said to be {\em strong} if $F_2$ and $F_0$ are invertible, and said to be {\em strict} if $F_2$ and $F_0$ are identities.

Given two monoidal functors $F$ and $G$ from $\mathcal{C}$ to $\mathcal{D}$, we denote by $\Nat_{\otimes}(F, G)$ the set of monoidal natural transformations from $F$ to $G$ (see \cite[XI.2]{MR1712872} for the definition of monoidal natural transformations).

\subsubsection{Convention on duals}

We fix our convention for dual objects in a monoidal category. Let $L$ and $R$ be objects of a monoidal category $\mathcal{C}$, and let $\varepsilon: L \otimes R \to \unitobj$ and $\eta: \unitobj \to R \otimes L$ be morphisms of $\mathcal{C}$. We say that $(L, \varepsilon, \eta)$ is a {\em left dual object} of $\mathcal{C}$ and $(R, \varepsilon, \eta)$ is a {\em right dual object of $L$} if the equations
\begin{equation*}
  (\varepsilon \otimes \id_L) \circ (\id_L \otimes \eta) = \id_L
  \quad \text{and} \quad
  (\id_R \otimes \varepsilon) \circ (\eta \otimes \id_R) = \id_R
\end{equation*}
are satisfied. A monoidal category is said to be {\em rigid} if every its object has a left dual object and a right dual object.

Now we assume that $\mathcal{C}$ is a rigid monoidal category. For each $X \in \mathcal{C}$, we fix a left dual object $(X^*, \eval_X, \coev_X)$ and a right dual object $({}^*\!X, \eval'_X, \coev'_X)$ of $X$. Then the assignment $X \mapsto X^*$ extends to a strong monoidal functor $(-)^*: \mathcal{C}^{\op} \to \mathcal{C}^{\rev}$, which we call the {\em left duality functor} of $\mathcal{C}$. Here, $\mathcal{C}^{\rev}$ is the monoidal category obtained from $\mathcal{C}$ by reversing the order of the tensor product. The {\em right duality functor} ${}^*(-): \mathcal{C}^{\op} \to \mathcal{C}^{\rev}$ of $\mathcal{C}$ is defined analogously.

It is known that ${}^*(-)$ is a quasi-inverse of $(-)^*$. In this paper, we assume that $(-)^*$ and ${}^*(-)$ are strict monoidal functors and mutually inverse to each other (see, {\it e.g.}, \cite[Lemma 5.4]{MR3314297} for a discussion). Thus we have
$(X \otimes Y)^* = Y^* \otimes X^*$,
${}^*(X^*) = X = ({}^*\!X)^*$, etc.

\subsubsection{Duality transformation}

Let $F : \mathcal{C} \to \mathcal{D}$ be a strong monoidal functor between rigid monoidal categories, and let $X \in \mathcal{C}$ be an object. The object $F(X^*)$ is a left dual object of $F(X)$ in a natural way, and hence there is a canonical isomorphism $\zeta_X : F(X)^* \to F(X^*)$ by the uniqueness of a left dual object.
The family $\zeta = \{ \zeta_X \}$ is an isomorphism of monoidal functors, which we call the {\em duality transformation} of $F$ ({\it cf}. \cite[Section 1]{MR2381536}).

\subsection{Modules over a monoidal category}

Let $\mathcal{C}$ be a monoidal category. A {\em left $\mathcal{C}$-module category} is a category $\mathcal{M}$ endowed with a functor $\ogreaterthan: \mathcal{C} \times \mathcal{M} \to \mathcal{M}$ (called the {\em action}) and natural isomorphisms
\begin{equation*}
  (X \otimes Y) \ogreaterthan M \to X \ogreaterthan (Y \ogreaterthan M)
  \quad \text{and} \quad
  \unitobj \ogreaterthan M \to M
  \quad (X, Y \in \mathcal{C}, M \in \mathcal{M})
\end{equation*}
satisfying certain equations similar to the axioms for monoidal categories.
A right $\mathcal{C}$-module category, a $\mathcal{C}$-bimodule category and related notions are defined analogously; see \cite[Chapter 7]{MR3242743} and \cite{MR3934626}. Here we only review basic results on left $\mathcal{C}$-module categories.

Let $\mathcal{M}$ and $\mathcal{N}$ be left $\mathcal{C}$-module categories.
A {\em lax left $\mathcal{C}$-module functor} from $\mathcal{M}$ to $\mathcal{N}$ is a functor $F: \mathcal{M} \to \mathcal{N}$ equipped with a natural transformation
\begin{equation*}
  \xi_{X,M}: X \catactl F(M) \to F(X \catactl M) \quad (X \in \mathcal{C}, M \in \mathcal{M})
\end{equation*}
compatible with the associativity and the unit isomorphisms. We also introduce the dual notion: An {\em oplax left $\mathcal{C}$-module functor} from $\mathcal{M}$ to $\mathcal{N}$ is a functor $G : \mathcal{M} \to \mathcal{N}$ equipped with a natural transformation
\begin{equation*}
  \zeta_{X,M} : G(X \catactl M) \to X \catactl G(M) \quad (X \in \mathcal{C}, M \in \mathcal{M})
\end{equation*}
such that, in a word, $G^{\op} : \mathcal{M}^{\op} \to \mathcal{N}^{\op}$ is a lax left $\mathcal{C}^{\op}$-module functor.
In this paper, adjoints of (op)lax module functors are considered.
A left adjoint of a lax left $\mathcal{C}$-module functor is not such a functor in general, but an oplax one. More precisely, we have:

\begin{lemma}[{\cite[Lemma 2.11]{MR3934626}}]
  \label{lem:mod-func-adj}
  Let $\mathcal{M}$ and $\mathcal{N}$ be left $\mathcal{C}$-module categories, let $R: \mathcal{M} \to \mathcal{N}$ be a functor, and let $L: \mathcal{N} \to \mathcal{M}$ be a left adjoint of $R$ with unit $\eta: \id_{\mathcal{N}} \to R L$ and counit $\varepsilon: L R \to \id_{\mathcal{M}}$.
  \begin{itemize}
  \item [(1)] If $(R, \xi)$ is a lax left $\mathcal{C}$-module functor, then $L$ is an oplax left $\mathcal{C}$-module functor with the structure morphism $\zeta$ given by
    \begin{equation*}
      \zeta_{X,N} = \varepsilon_{X \catactl L(N)} \circ L(\xi_{X, L(N)}) \circ L(\id_X \catactl \eta_{N})
      \quad (X \in \mathcal{C}, N \in \mathcal{N}).
    \end{equation*}
  \item [(2)] If $(L, \zeta)$ is an oplax left $\mathcal{C}$-module functor, then $R$ is a lax left $\mathcal{C}$-module functor with the structure morphism $\xi$ given by
    \begin{equation*}
      \xi_{X,M} = R(\id_X \catactl \varepsilon_{M}) \circ R(\zeta_{X, R(M)}) \circ \eta_{X \catactl R(M)}
      \quad (X \in \mathcal{C}, M \in \mathcal{M}).
    \end{equation*}
  \end{itemize}
  The above two constructions give a one-to-one correspondence between
  the set of lax left $\mathcal{C}$-module functor structures on $R$ and
  the set of oplax left $\mathcal{C}$-module functor structures on $L$.
\end{lemma}

We note that the natural transformations $\xi_{X,-}$ and $\zeta_{X,-}$ for $X \in \mathcal{C}$ appearing in the above lemma correspond to each other via the bijection
\begin{equation}
  \label{eq:mod-fun-adj-1}
  \Nat(T_X R, R T'_X)
  \mathop{\cong}^{\eqref{eq:nat-adjunctions-1}} \Nat(L T_X R, T'_X)
  \mathop{\cong}^{\eqref{eq:nat-adjunctions-2}} \Nat(L T_X, T'_X L),
\end{equation}
where $T_X = X \catactl \id_{\mathcal{N}}$ and $T'_X = X \catactl \id_{\mathcal{M}}$.

We say that an (op)lax $\mathcal{C}$-module functor is {\em strong} if its structure morphism is invertible. The following lemma is also important:

\begin{lemma}[{\cite[Lemma 2.10]{MR3934626}}]
  \label{lem:mod-func-strong}
  Suppose that $\mathcal{C}$ is rigid. Then all lax $\mathcal{C}$-module functors and all oplax $\mathcal{C}$-module functors are strong.
\end{lemma}

Thus, when $\mathcal{C}$ is rigid, lax $\mathcal{C}$-module functors and oplax $\mathcal{C}$-module functors may be simply called $\mathcal{C}$-module functors (but sometimes the adjective `op(lax)' is used to specify the direction of the structure morphism). Lemmas~\ref{lem:mod-func-adj} and \ref{lem:mod-func-strong} imply that the class of $\mathcal{C}$-module functors is closed under taking an adjoint.

Now we suppose that $\mathcal{C}$ is rigid. Let $F : \mathcal{M} \to \mathcal{N}$ be a lax left $\mathcal{C}$-module functor with structure morphism $\xi$.
We note how the inverse of $\xi$ is given:
First, we remark that there is a natural isomorphism
\begin{equation}
  \label{eq:adj-X-action}
  \Hom_{\mathcal{M}}(M, X \catactl M')
  \cong \Hom_{\mathcal{M}}(X^* \catactl M, M')
  \quad (X \in \mathcal{C}, M, M' \in \mathcal{M}),
\end{equation}
that is, $T_{X^*}$ is left adjoint to $T_X$. By the proof of \cite[Lemma 2.10]{MR3934626}, we see that the inverse of $\xi_{X,-} : T_X F \to F T_X$ corresponds to $\xi_{X^*,-}$ via the bijection
\begin{equation}
  \label{eq:mod-fun-adj-2}
  \Nat(F T_{X}, T_{X} F)
  \xrightarrow[\cong]{\ \eqref{eq:nat-adjunctions-1}, \eqref{eq:nat-adjunctions-2} \ }
  \Nat((T_{X})^{\ladj} F, F (T_{X})^{\ladj})
  \mathop{=}^{\eqref{eq:adj-X-action}}
  \Nat(T_{X^*} F, F T_{X^*}).
\end{equation}

If $F$ has a left adjoint, then $F^{\ladj}$ is an oplax left $\mathcal{C}$-module functor by the structure morphism $\zeta$ given by Lemma~\ref{lem:mod-func-adj}. By the above discussion, the inverse of $\xi_{X,-}$ corresponds to $\zeta_{X^*,-}$ via the bijection
\begin{equation*}
  \Nat(F T_{X}, T_{X} F)
  \mathop{\cong}^{\eqref{eq:mod-fun-adj-2}}
  \Nat(T_{X^*} F, F T_{X^*})
  \mathop{\cong}^{\eqref{eq:mod-fun-adj-1}}
  \Nat(F^{\ladj} T_{X^*}, T_{X^*} F^{\ladj}).
\end{equation*}
In general, a morphism $\nu \in \Nat(F T_{X}, T_{X} F)$ is sent to
\begin{equation*}
  F^{\ladj} T_{X^*}
  \xrightarrow{\quad \cong \quad}
  (T_{X} F)^{\ladj}
  \xrightarrow{\quad \nu^{\ladj} \quad}
  (F T_{X})^{\ladj}
  \xrightarrow{\quad \cong \quad}
  T_{X^*} F^{\ladj}.
\end{equation*}
via the above bijection. Here, the first and the third arrows represent the canonical isomorphism $(P \circ Q)^{\ladj} \cong Q^{\ladj} \circ P^{\ladj}$ for composable functors $P$ and $Q$ admitting left adjoints. By the above discussion, we obtain:

\begin{lemma}
  Under the above assumptions, the following diagram commutes:
  \begin{equation*}
    \begin{tikzcd}[column sep = 3em]
      F^{\ladj} \circ T_{X^*}
      \arrow[r, equal, "\eqref{eq:adj-X-action}"]
      \arrow[d, "\zeta_{X^*,-}"']
      & F^{\ladj} \circ (T_{X})^{\ladj}
      \arrow[r, "\cong"]
      & (T_{X} \circ F)^{\ladj}
      \arrow[d, leftarrow, "(\xi_{X,-})^{\ladj}"] \\
      T_{X^*} \circ F^{\ladj}
      \arrow[r, equal, "\eqref{eq:adj-X-action}"]
      & (T_{X})^{\ladj} \circ F^{\ladj}
      \arrow[r, "\cong"]
      & (F \circ T_X)^{\ladj}
    \end{tikzcd}
  \end{equation*}
\textsf{}\end{lemma}

By using the above lemma twice, we obtain:

\begin{lemma}
  \label{lem:mod-func-double-adj}
  Let $F: \mathcal{M} \to \mathcal{N}$ be a lax left $\mathcal{C}$-module functor with structure morphism $\xi$.
  Suppose that a double left adjoint $F^{\lladj} := (F^{\ladj})^{\ladj}$ exists, and let
  \begin{equation*}
    \omega_{X,M} : X \catactl F^{\lladj}(M) \to F^{\lladj}(X \catactl M)
    \quad (X \in \mathcal{C}, M \in \mathcal{M})
  \end{equation*}
  be the structure morphism of $F^{\lladj}: \mathcal{M} \to \mathcal{N}$ as a lax left $\mathcal{C}$-module functor. Then, for every object $X \in \mathcal{C}$, the following diagram commutes:
  \begin{equation*}
    \begin{tikzcd}[column sep = 3em]
      T_{X^{**}} \circ F^{\lladj}
      \arrow[d, "\omega_{X^{**}, -}"']
      \arrow[r, equal, "\eqref{eq:adj-X-action}"]
      & (T_{X})^{\lladj} \circ F^{\lladj}
      \arrow[r, "\cong"]
      & (T_{X} \circ F)^{\lladj}
      \arrow[d, "(\xi_{X,-})^{\lladj}"] \\
      F^{\lladj} \circ T_{X^{**}}
      \arrow[r, equal, "\eqref{eq:adj-X-action}"]
      & F^{\lladj} \circ (T_X)^{\lladj}
      \arrow[r, "\cong"]
      & (F \circ T_X)^{\lladj}
    \end{tikzcd}
  \end{equation*}
\end{lemma}

\subsection{Finite abelian categories}
\label{subsec:fin-ab-cat}

Throughout this paper, we work over an algebraically closed field $\bfk$ of arbitrary characteristic. By an algebra over $\bfk$, we always mean an associative and unital algebra over $\bfk$. Given an algebra $R$ over $\bfk$, we denote by $R\mbox{-mod}$ the category of finite-dimensional left $R$-modules. A {\em finite abelian category} over $\bfk$ is a $\bfk$-linear category that is equivalent to $A\mbox{-mod}$ for some finite-dimensional algebra $A$ over $\bfk$.

Given two finite abelian categories $\mathcal{M}$ and $\mathcal{N}$, we denote by $\Lex(\mathcal{M}, \mathcal{N})$ the category of $\bfk$-linear left exact functors from $\mathcal{M}$ to $\mathcal{N}$.
Let $A$ and $B$ be two finite-dimensional algebras, and let $\bimod{A}{B}$ denote the category of finite-dimensional $A$-$B$-bimodules.
By the Eilenberg-Watts theorem and the tensor-Hom adjunction, we see that the functor
\begin{equation}
  (\bimod{A}{B})^{\op} \to \Lex(\lmod{A}, \lmod{B}),
  \quad M \mapsto \Hom_{A}(M, -)
\end{equation}
is an equivalence of categories. Let $F: \mathcal{M} \to \mathcal{N}$ be a $\bfk$-linear functor between finite abelian categories $\mathcal{M}$ and $\mathcal{N}$. The above equivalence implies that $F$ has a left adjoint if and only if $F$ is left exact. By applying this argument to $F^{\op}$, we see that $F$ has a right adjoint if and only if it is right exact.

The above equivalence also implies that a $\bfk$-linear functor $\mathcal{M} \to \fdVec := \lmod{\bfk}$ is representable if and only if it is left exact. Thus, given objects $V \in \fdVec$ and $M \in \mathcal{M}$, an object $V \otimes_{\bfk} M \in \mathcal{M}$ can be defined by the natural isomorphism
\begin{equation}
  \label{eq:Vect-action-def}
  \Hom_{\mathcal{M}}(V \otimes_{\bfk} M, M') \cong \Hom_{\bfk}(V, \Hom_{\mathcal{M}}(M, M'))
\end{equation}
for $M' \in \mathcal{M}$. Every finite abelian category $\mathcal{M}$ is a left $\fdVec$-module category by the operation $(V, M) \mapsto V \otimes_{\bfk} M$, and every $\bfk$-linear functor between finite abelian categories is a $\fdVec$-module functor in a natural way.

\subsection{Finite tensor categories and their modules}
\label{subsec:FTCs}

A {\em finite tensor category} \cite{MR2119143} is a rigid monoidal category $\mathcal{C}$ such that $\mathcal{C}$ is a finite abelian category, the tensor product functor $\otimes: \mathcal{C} \times \mathcal{C} \to \mathcal{C}$ is $\bfk$-linear in each variable, and the unit object $\unitobj \in \mathcal{C}$ is a simple object.

Let $\mathcal{C}$ be a finite tensor category. A {\em finite left $\mathcal{C}$-module category} is a left $\mathcal{C}$-module category $\mathcal{M}$ such that $\mathcal{M}$ is a finite abelian category over $\bfk$ and the action of $\mathcal{C}$ on $\mathcal{M}$ is $\bfk$-linear and right exact in each variable (this condition implies that the action is {\em exact} in each variable \cite[Corollary 2.26]{MR3934626}). A finite right $\mathcal{C}$-module category and a finite $\mathcal{C}$-bimodule category are defined analogously.

Now let $\mathcal{M}$ be a left $\mathcal{C}$-module category. An algebra $A$ in $\mathcal{C}$ ($=$ a monoid object in $\mathcal{C}$ \cite{MR1712872}) defines a monad $A \catactl \id_{\mathcal{M}}$ on $\mathcal{M}$. We define the category ${}_A \mathcal{M}$ of {\em left $A$-modules in $\mathcal{M}$} to be the Eilenberg-Moore category of this monad.
It is well-known that ${}_A \mathcal{M}$ is a finite abelian category if $\mathcal{M}$ is a finite left $\mathcal{C}$-module category.

If $\mathcal{N}$ is a right $\mathcal{C}$-module category, then the category $\mathcal{N}_A$ of right $A$-modules in $\mathcal{N}$ is defined analogously. If $B$ is an algebra in $\mathcal{C}$ and $\mathcal{L}$ is a $\mathcal{C}$-bimodule category, then the category ${}_A \mathcal{L}_B$ of $A$-$B$-bimodules in $\mathcal{C}$ is defined. The categories $\mathcal{N}_A$ and ${}_A \mathcal{L}_B$ are finite abelian categories provided that $\mathcal{N}$ and $\mathcal{L}$ are finite.

\subsection{Internal Hom functors}

Let $\mathcal{C}$ be a finite tensor category, and let $\mathcal{M}$ be a finite left $\mathcal{C}$-module category. We fix an object $M \in \mathcal{M}$. Then the functor $\id_{\mathcal{C}} \catactl M$ from $\mathcal{C}$ to $\mathcal{M}$ is $\bfk$-linear and exact, and hence it has a right adjoint. We denote a right adjoint of $\id_{\mathcal{C}} \catactl M$ by $\iHom(M, -)$. Namely, there is an isomorphism
\begin{equation*}
  \Hom_{\mathcal{C}}(V, \iHom(M, N)) \cong \Hom_{\mathcal{M}}(V \catactl M, N)
\end{equation*}
natural in $V \in \mathcal{C}$ and $N \in \mathcal{N}$. The assignment $(M, N) \mapsto \iHom(M, N)$ extends to a functor $\mathcal{M}^{\op} \times \mathcal{M} \to \mathcal{C}$, which we call the {\em internal Hom functor} of $\mathcal{M}$.

Since the functor $\id_{\mathcal{C}} \catactl M$ is a left $\mathcal{C}$-module functor in an obvious way, its right adjoint is also a left $\mathcal{C}$-module functor. Namely, there is a natural isomorphism
\begin{align}
  \label{eq:int-Hom-iso-1}
  \iHom(M, X \catactl N) & \cong X \otimes \iHom(M, N)
\end{align}
for $X \in \mathcal{C}$ and $N \in \mathcal{M}$.

\section{The Drinfeld center and its variants}
\label{sec:bimodule-cat-center}

\subsection{The Drinfeld center and its variants}

Let $\mathcal{C}$ be a finite tensor category, and let $\mathcal{M}$ be a finite $\mathcal{C}$-bimodule category.
The {\em center} of $\mathcal{M}$, denoted by $\mathcal{Z}(\mathcal{M})$, is the category defined as follows:
An object of this category is a pair $(M, \sigma_M)$ consisting of an object $M \in \mathcal{M}$ and a natural transformation
\begin{equation*}
  \sigma_M(X) : M \catactr X \to X \catactl M
  \quad (X \in \mathcal{C})
\end{equation*}
such that the diagrams
\begin{equation*}
  \begin{tikzcd}[column sep = 5em]
    M \catactr (X \otimes Y)
    \arrow[r, "\sigma_{M}(X \otimes Y)"]
    \arrow[d, "\cong"']
    & (X \otimes Y) \catactl M \\
    (M \catactr X) \catactr Y
    \arrow[d, "\sigma_M(X) \catactr \id"']
    & X \catactl (Y \catactl M)
    \arrow[u, "\cong"'] \\
    (X \catactl M) \catactr Y
    \arrow[r, "\cong"]
    & X \catactl (M \catactr Y)
    \arrow[u, "\id \catactl \sigma_M(Y)"']
  \end{tikzcd}
  \qquad
  \begin{tikzcd}
    \unitobj \catactl M
    \arrow[dd, "\sigma_M(\unitobj)"']
    \arrow[rd, "\cong"] \\
    & M \\
    M \catactr \unitobj
    \arrow[ru, "\cong"']
  \end{tikzcd}
\end{equation*}
commute for all objects $X, Y \in \mathcal{C}$, where $\cong$'s mean the associativity or the unit isomorphism for the $\mathcal{C}$-bimodule category $\mathcal{M}$. Given two objects $\mathbf{M} = (M, \sigma_M)$ and $\mathbf{N} = (N, \sigma_N)$ of $\mathcal{Z}(\mathcal{M})$, a morphism $f : \mathbf{M} \to \mathbf{N}$ in $\mathcal{Z}(\mathcal{M})$ is a morphism $f : M \to N$ in $\mathcal{M}$ such that the equation
\begin{equation*}
  (\id_X \catactl f) \circ \sigma_M(X) = \sigma_N(X) \circ (f \catactr \id_X)
\end{equation*}
holds for all $X \in \mathcal{C}$.

Given two finite $\mathcal{C}$-bimodule categories $\mathcal{M}$ and $\mathcal{N}$, we denote by $\Fun_{\mathcal{C}|\mathcal{C}}(\mathcal{M}, \mathcal{N})$ the category of $\bfk$-linear $\mathcal{C}$-bimodule functors from $\mathcal{M}$ to $\mathcal{N}$. The following lemma is well-known:

\begin{lemma}
  \label{lem:center-as-bimod-func}
  For a finite $\mathcal{C}$-bimodule category $\mathcal{M}$, there is an equivalence
  \begin{equation}
    \label{eq:center-as-bimod-func}
    \Fun_{\mathcal{C}|\mathcal{C}}(\mathcal{C}, \mathcal{M})
    \approx \mathcal{Z}(\mathcal{M}),
    \quad T \mapsto T(\unitobj)
  \end{equation}
  of $\bfk$-linear categories.
\end{lemma}

Specifically speaking, if $T \in \Fun_{\mathcal{C}|\mathcal{C}}(\mathcal{C}, \mathcal{M})$, then the object $T(\unitobj) \in \mathcal{M}$ becomes an object of $\mathcal{Z}(\mathcal{M})$ together with the natural transformation defined by
\begin{equation*}
  T(\unitobj) \catactr X
  \xrightarrow{\quad \xi^{(r)}_{X,\unitobj} \quad}
  T(\unitobj \otimes X)
  \xrightarrow{\quad = \quad}
  T(X \otimes \unitobj)
  \xrightarrow{\quad (\xi^{(\ell)}_{X,\unitobj})^{-1} \quad}
  X \catactl T(\unitobj)
\end{equation*}
for $X \in \mathcal{C}$, where $\xi^{(\ell)}$ and $\xi^{(r)}$ are the left and the right $\mathcal{C}$-module structure of the bimodule functor $T$ (we note that a structure morphism of a $\mathcal{C}$-bimodule functor is invertible by Lemma~\ref{lem:mod-func-strong}).

A quasi-inverse of \eqref{eq:center-as-bimod-func} is constructed as follows: Given an object $(M, \sigma_M)$ of $\mathcal{Z}(\mathcal{M})$, we define $T_{M} : \mathcal{C} \to \mathcal{M}$ by $T_M(X) = X \catactl M$ for $X \in \mathcal{C}$. The functor $T_M$ becomes a $\mathcal{C}$-bimodule functor with the left $\mathcal{C}$-module structure given by the associativity isomorphism and the right $\mathcal{C}$-module structure given by
\begin{gather*}
  T_M(X) \catactr Y
  = (X \catactl M) \catactr Y
  \xrightarrow{\quad \cong \quad} X \catactl (M \catactr Y) \\
  \xrightarrow{\quad \id \catactl \sigma_M(Y) \quad}
  X \catactl (Y \catactr M)
  \xrightarrow{\quad \cong \quad} (X \otimes Y) \catactr M
  = T_M(X \otimes Y)
\end{gather*}
for $X, Y \in \mathcal{C}$. The assignment $(M, \sigma_M) \mapsto T_M$ gives a quasi-inverse of \eqref{eq:center-as-bimod-func}.

The structure morphism $\sigma_M : M \catactr \id_{\mathcal{C}} \to \id_{\mathcal{C}} \catactl M$ is used to define a (lax) right $\mathcal{C}$-module structure of $T_M$. Since $\mathcal{C}$ is rigid, every lax $\mathcal{C}$-module functor is strong by Lemma~\ref{lem:mod-func-strong}. Thus we have:

\begin{lemma}
  If $(M, \sigma_M)$ is an object of $\mathcal{Z}(\mathcal{M})$, then $\sigma_M$ is invertible.
\end{lemma}

The center construction is functorial in the following sense:

\begin{lemma}
  Let $\mathcal{M}$ and $\mathcal{N}$ be finite $\mathcal{C}$-bimodule categories, and let $F: \mathcal{M} \to \mathcal{N}$ be a $\bfk$-linear $\mathcal{C}$-bimodule functor with structure morphisms
  \begin{equation*}
    \xi^{(\ell)}_{X,M}: X \catactl F(M) \to F(X \catactl M)
    \quad \text{and} \quad
    \xi^{(r)}_{M,X}: F(M) \catactr X \to F(M \catactr X).
  \end{equation*}
  Then $F$ induces a $\bfk$-linear functor
  \begin{equation*}
    F: \mathcal{Z}(\mathcal{M}) \to \mathcal{Z}(\mathcal{N}),
    \quad (M, \sigma_M) \mapsto \Big( F(M), (\xi^{(\ell)}_{X,M})^{-1}F(\sigma_M)\xi^{(r)}_{M,X} \Big).
  \end{equation*}
\end{lemma}
\begin{proof}
  Straightforward.
\end{proof}

By a {\em tensor functor}, we mean a $\bfk$-linear exact strong monoidal functor.
Let $\mathcal{C}$ and $\mathcal{D}$ be finite tensor categories. Given a tensor functor $F : \mathcal{C} \to \mathcal{D}$ and a left $\mathcal{D}$-module category $\mathcal{M}$ with action $\catactl$, we denote by ${}_{\langle F \rangle}\mathcal{M}$ the category $\mathcal{M}$ viewed as a left $\mathcal{C}$-module category by the action $\catactl_F$ defined by
\begin{equation*}
  X \catactl_F M := F(X) \catactl M \quad (X \in \mathcal{C}, M \in \mathcal{M}).
\end{equation*}
Analogous notations will be used for right module categories and bimodule categories. In this paper, we mainly deal with the following special case of the center construction:

\begin{definition}
  For two tensor functors $F, G : \mathcal{C} \to \mathcal{D}$, we set
  \begin{equation*}
    \mathcal{Z}(F, G) := \mathcal{Z}({}_{\langle G \rangle} \mathcal{D}_{\langle F \rangle}).
  \end{equation*}
\end{definition}

Some important categories are obtained as special cases:
For example, $\mathcal{Z}(F, F)$ is the category that has been introduced by Majid \cite{MR1151906} under the name of the dual of the functored category $(\mathcal{C}, F)$. In particular, the category $\mathcal{Z}(\mathcal{C}) := \mathcal{Z}(\id_{\mathcal{C}}, \id_{\mathcal{C}})$ is the {\em Drinfeld center} of $\mathcal{C}$. The category $\mathcal{Z}(\id_{\mathcal{C}}, S^2)$, where $S = (-)^*$ is the left duality functor on $\mathcal{C}$, is called the {\em twisted Drinfeld center} of $\mathcal{C}$ in \cite{MR3638361}.

\subsection{Tensor products, duals and invertible objects}

\newcommand{\TF}{\mathscr{T\!F}}

Let $\mathcal{C}$ and $\mathcal{D}$ be finite tensor categories, and let $F : \mathcal{C} \to \mathcal{D}$ be a tensor functor.
Following Majid \cite[Theorem 3.3]{MR1151906}, the category $\mathcal{Z}(F, F)$ is a rigid monoidal category.
In the case where $F \ne G$, the category $\mathcal{Z}(F, G)$ does not seem to have a reasonable structure of a monoidal category in general.
Though, for three tensor functors $F, G, H: \mathcal{C} \to \mathcal{D}$, the tensor product $\otimes: \mathcal{Z}(G, H) \times \mathcal{Z}(F, G) \to \mathcal{Z}(F, H)$ is defined by
\begin{equation*}
  (V, \sigma_V) \otimes (W, \sigma_W) = (V \otimes W, \rho)
\end{equation*}
for $(V, \sigma_V) \in \mathcal{Z}(G, H)$ and $(W, \sigma_W) \in \mathcal{Z}(F, G)$, where $\rho$ is defined by
\begin{equation*}
  \rho(X) = (\sigma_V(X) \otimes \id_W) \circ (\id_V \otimes \sigma_W(X))
  : V \otimes W \otimes F(X) \to H(X) \otimes V \otimes W
\end{equation*}
for $X \in \mathcal{C}$. The class of tensor functors from $\mathcal{C}$ to $\mathcal{D}$ form a bicategory, which we denote by  $\TF(\mathcal{C}, \mathcal{D})$, with respect to this tensor product.

As we have mentioned in the above, $\mathcal{Z}(F, F)$ is a rigid monoidal category.
This result is generalized to the bicategory $\TF(\mathcal{C}, \mathcal{D})$ as follows: Given an object $\mathbf{V} = (V, \sigma_V)$ of $\mathcal{Z}(F, G)$, we define $\mathbf{V}^* \in \mathcal{Z}(G, F)$ by $\mathbf{V}^* = (V^*, \sigma_{V^*})$, where
\begin{equation*}
  \sigma_{V^*}(X) =
  \begin{gathered}[t]
    (\eval_{V} \otimes \id_{F(X)} \otimes \id_{V^*})
    \circ (\id_{V^*} \otimes \sigma_{V}(X)^{-1} \otimes \id_{V^*}) \\
    \circ (\id_{V^*} \otimes \id_{G(X)} \otimes \coev_{V})
  \end{gathered}
\end{equation*}
for $X \in \mathcal{C}$.
Let $\xi_X : F(X)^* \to F(X^*)$ and $\zeta_X : G(X)^* \to G(X^*)$ be the duality transformations for $F$ and $G$, respectively.
We also define ${}^*\mathbf{V} \in \mathcal{Z}(G, F)$ by
${}^*\mathbf{V} = ({}^* V, \sigma_{{}^*V})$, where $\sigma_{{}^*V}$ is the natural transformation given by
\begin{equation*}
  \sigma_{{}^* V}(X)
  = ({}^*\xi_X \otimes \id_{{}^*V}) \circ {}^* \sigma_{V}(X^*) \circ (\id_{{}^*V} \otimes {}^*\zeta_X^{-1})
\end{equation*}
for $X \in \mathcal{C}$.
Then $(\mathbf{V}^*, \eval_V, \coev_V)$ and $({}^*\mathbf{V}, \eval_V', \coev_V')$ are a left and a right dual object of $\mathbf{V}$, respectively, in the bicategory $\TF(\mathcal{C}, \mathcal{D})$.

A 1-cell in a bicategory is {\em invertible} if it is dualizable and the evaluation and the coevaluation morphisms are isomorphisms. A 1-cell $\mathbf{V} = (V, \sigma_V)$ of $\TF(\mathcal{C}, \mathcal{D})$ is invertible if and only if its underlying object $V$ is an invertible object of $\mathcal{D}$.

\subsection{The center as the category of modules}
\label{subsec:center-as-modules}

We go back to discuss the center of general finite bimodule categories.
Let $\mathcal{C}$ be a finite tensor category, and let $\mathcal{M}$ be a finite $\mathcal{C}$-bimodule category.
We consider the coend
\begin{equation*}
  A = \int^{X \in \mathcal{C}} X \boxtimes {}^*\!X
\end{equation*}
in the `enveloping' tensor category $\mathcal{C}^{\env} := \mathcal{C} \boxtimes \mathcal{C}^{\rev}$ (see \cite{MR3632104} for the existence of this coend). We denote by $\din_X : X \boxtimes {}^*\!X \to A$ the universal dinatural transformation for the coend $A$.

\begin{remark}
  \label{rem:coend-dual-shift}
  Let $n$ be an integer. Since the left duality functor $S := (-)^*$ of $\mathcal{C}$ is an anti-equivalence, the dinatural transformation
  \begin{equation*}
    \din_{S^n(X)}: S^{n}(X) \boxtimes S^{n-1}(X) \to A \quad (X \in \mathcal{C})
  \end{equation*}
  is universal. Hence we may identify $A = \int^{X \in \mathcal{C}} S^{n}(X) \boxtimes S^{n-1}(X)$.
\end{remark}

By the exactness of the tensor product of $\mathcal{C}^{\env}$, we have
\begin{equation*}
  A \otimes A
  = \int^{(X, Y) \in \mathcal{C} \times \mathcal{C}} (X \boxtimes {}^*\!X) \otimes (Y \boxtimes {}^*Y)
\end{equation*}
with universal dinatural transformation $\din_X \otimes \din_Y$. By the universal property, there is a unique morphism $m: A \otimes A \to A$ such that the equation
\begin{equation}
  \label{eq:cano-alg-mult}
  m \circ (\din_X \otimes \din_Y) = \din_{X \otimes Y}
\end{equation}
holds for all objects $X, Y \in \mathcal{C}$. The coend $A$ is an algebra in $\mathcal{C}^{\env}$ with multiplication $m$ and unit $\mathtt{u} = \din_{\unitobj}$ ({\it cf}. \cite[Lemma 4.5]{MR3632104}).

\begin{definition}
  We call $A$ the {\em canonical algebra} of $\mathcal{C}$.
\end{definition}

This definition looks different from \cite[Definition 7.9.12]{MR3242743}, however, is equivalent to that; see \cite[Subsection 4.3]{MR3632104} and Appendix \ref{append:subsec:canonical-algebra}.

Now let $\mathcal{M}$ be a finite $\mathcal{C}$-bimodule category. We regard it as a left $\mathcal{C}^{\env}$-module category by defining the action by $(X \boxtimes Y^{\rev}) \catactl M = X \catactl M \catactr Y$. Hence the category of left $A$-modules in $\mathcal{M}$ is defined.

\begin{lemma}
  \label{lem:A-mod-and-center}
  For $\mathcal{M}$ as above, there is an isomorphism ${}_A\mathcal{M} \cong \mathcal{Z}(\mathcal{M})$ of $\bfk$-linear categories commuting with the forgetful functors to $\mathcal{M}$.
\end{lemma}
\begin{proof}
  The Drinfeld center $\mathcal{Z}(\mathcal{C})$ is isomorphic to the category of modules over the monad $Z = A \catactl \id_{\mathcal{C}}$ on $\mathcal{C}$ \cite{MR2869176}.
  This lemma can be proved in a similar way as this fact:
  For $M \in \mathcal{M}$, there are isomorphisms
  \begin{align*}
    \Hom_{\mathcal{M}}(A \catactl M, M)
    & \cong \textstyle \Hom_{\mathcal{M}}(\int^{X \in \mathcal{C}} (X^* \boxtimes X) \catactl M, M) \\
    & \cong \textstyle \int_{X \in \mathcal{C}} \Hom_{\mathcal{M}}(X^* \catactl M \catactr X, M) \\
    & \cong \textstyle \int_{X \in \mathcal{C}} \Hom_{\mathcal{M}}(M \catactr X, X \catactl M) \\
    & \cong \Nat(M \catactr \id_{\mathcal{C}}, \id_{\mathcal{C}} \catactl M).
  \end{align*}
  One can check that a morphism $A \catactl M \to M$ in $\mathcal{M}$ makes $M$ an $A$-module if and only if the corresponding natural transformation $M \catactr \id_{\mathcal{C}} \to \id_{\mathcal{C}} \catactl M$ makes $M$ an object of $\mathcal{Z}(\mathcal{M})$. This establishes an isomorphism ${}_A \mathcal{M} \cong \mathcal{Z}(\mathcal{M})$ that preserves underlying objects. The proof is done.
\end{proof}

It has been known that the Drinfeld center $\mathcal{Z}(\mathcal{C}) = \mathcal{Z}(\id_{\mathcal{C}}, \id_{\mathcal{C}})$ and the twisted Drinfeld center $\mathcal{Z}(\id_{\mathcal{C}}, S^2)$ are finite abelian categories. It is worth noting that the finiteness of the center of a finite $\mathcal{C}$-bimodule category follows from the above lemma and a finiteness criterion mentioned in Subsection~\ref{subsec:FTCs}.

Now let $\mathcal{C}$ and $\mathcal{D}$ be a finite tensor category, and let $F$ and $G$ be tensor functors from $\mathcal{C}$ to $\mathcal{D}$.
Since $G \boxtimes F^{\rev} : \mathcal{C}^{\env} \to \mathcal{D}^{\env}$ is a tensor functor, the object
\begin{equation}
  \label{eq:algebra A-F-G}
  A_{F,G} := (G \boxtimes F^{\rev})(A) = \int^{X \in \mathcal{C}} G(X) \boxtimes F({}^*\!X)
\end{equation}
is an algebra in $\mathcal{D}^{\env}$. For a finite $\mathcal{D}$-bimodule category $\mathcal{M}$, an $A_{F,G}$-module in $\mathcal{M}$ is nothing but an $A$-module in ${}_{\langle G \rangle}\mathcal{M}_{\langle F \rangle}$. Thus, by Lemma~\ref{lem:A-mod-and-center}, we have:

\begin{lemma}
  \label{lem:Z-F-G-monadic}
  For a finite $\mathcal{D}$-bimodule category $\mathcal{M}$, the category of $A_{F,G}$-modules in $\mathcal{M}$ is isomorphic to the center of the $\mathcal{C}$-bimodule category ${}_{\langle G \rangle}\mathcal{M}_{\langle F \rangle}$.
  In particular, $\mathcal{Z}(F, G)$ is identified with the category of $A_{F,G}$-modules in $\mathcal{D}$.
\end{lemma}

\section{Relative modular object}
\label{sec:relative-modulus}

\subsection{Radford isomorphism}

We first recall the categorical analogue of the Radford $S^4$-formula given by Etingof, Nikshych and Ostrik \cite{MR2097289}. Let $\mathcal{C}$ be a finite tensor category. Then $\mathcal{C}^{\env}$ acts on $\mathcal{C}$ from the left by the action
\begin{equation*}
  (X \boxtimes Y^{\rev}) \catactl V = X \otimes V \otimes Y
  \quad (V, X, Y \in \mathcal{C}).
\end{equation*}
We denote by $\iHom : \mathcal{C}^{\op} \times \mathcal{C} \to \mathcal{C}^{\env}$ the internal Hom functor for the left $\mathcal{C}^{\env}$-module category $\mathcal{C}$. There is an equivalence
\begin{equation}
  \label{eq:fund-thm-Hopf-bimod-1}
  \Kappa_{\mathcal{C}}: \mathcal{C} \to (\mathcal{C}^{\env})_{\iHom(\unitobj, \unitobj)}
  \quad V \mapsto \iHom(\unitobj, V)
\end{equation}
of left $\mathcal{C}^{\env}$-module categories \cite[Proposition 2.3]{MR2097289}, which can be thought of as a categorical analogue of the fundamental theorem for Hopf bimodules.

As remarked in \cite[Subsection 4.3]{MR3632104}, the algebra $A := \iHom(\unitobj, \unitobj)$ is identified with the canonical algebra $\int^{X \in \mathcal{C}} X \boxtimes {}^*\!X$ in $\mathcal{C}^{\env}$. Since $\Kappa_{\mathcal{C}}$ is an equivalence of left $\mathcal{C}^{\env}$-module categories, it induces an equivalence
\begin{equation}
  \label{eq:ENO-Hopf-bimod-2}
  \Kappa_{\mathcal{C}}: {}_{A^{**}}\mathcal{C} \to {}_{A^{**}}(\mathcal{C}^{\env})_A,
  \quad V \mapsto \iHom(\unitobj, V)
\end{equation}
between the categories of left $A^{**}$-modules.
We note that $A^{**}$ is identified with the algebra $A_{\id_{\mathcal{C}}^{}, S^4}$, where $S = (-)^*$ is the left duality functor on $\mathcal{C}$ (see Subsection~\ref{subsec:center-as-modules} for notation). Indeed, by Remark~\ref{rem:coend-dual-shift}, we have isomorphisms
\begin{equation*}
  A^{**}
  \cong \int^{X \in \mathcal{C}} (X^{**} \boxtimes X^{*})^{**}
  \cong \int^{X \in \mathcal{C}} S^4(X) \boxtimes {}^*\!X
  = A_{\id_{\mathcal{C}}^{}, S^4}
\end{equation*}
of algebras in $\mathcal{C}^{\env}$.

The object $A^* \in \mathcal{C}^{\env}$ is an $A^{**}$-$A$-bimodule by the actions
\begin{gather}
  \label{eq:cano-alg-m-dagger}
  m^{\ddag} := (\eval_{A^{*}} \otimes \id_{A^*}) \circ (\id_{A^{**}} \otimes m^*) : A^{**} \otimes A^* \to A^*, \\
  \label{eq:cano-alg-m-flat}
  m^{\dagger} := (\id_{A^{*}} \otimes \eval_A) \circ (m^* \otimes \id_{A}) : A^* \otimes A \to A^*.
\end{gather}
The following graphical expressions may be helpful:
\begin{gather*}
  m^* = 
  \begin{tikzpicture}[x = 1pc, y = 1pc, baseline=-.5em]
    \draw let \p1 = (1,0), \p2 = ($(\p1)+(1,0)$),
    \p3 = ($(\p2)+(1,0)$)
    in (\p1) coordinate (T1)
    to [out=-90, in=-90, looseness=2] coordinate[pos = .5](M) (\p2)
    to [out=+90, in=+90, looseness=2] (\p3) coordinate (T3)
    to (\x3, -2) node [below] {$A^{*}$};
    \fill (M) circle (1pt);
    \draw let \p1 = (M), \p2 = ($(T1)-(1,0)$)
    in (\p1) to [out=-90, in=-90, looseness=2] (\x2, \y1)
    to ($(\x2,2)$) node [above] {$A^{*}$};
    \draw let \p1 = (T1), \p2 = (T3)
    in (\p1) to [out=+90, in=+90, looseness=2] ($(\p2)+(1,0)$)
    to ($(\x2,-2)+(1,0)$) node [below] {$A^{*}$};
  \end{tikzpicture}
  \quad
  m^{\ddag} =
  \begin{tikzpicture}[x = 1pc, y = 1pc, baseline=-.5em]
    \draw let \p1 = (1,0), \p2 = ($(\p1)+(1,0)$),
    \p3 = ($(\p2)+(1,0)$), \p4 = ($(\p1)-(2.5,0)$)
    in (\p1) coordinate (T1)
    to [out=-90, in=-90, looseness=2] coordinate[pos = .5](M) (\p2)
    to [out=+90, in=+90, looseness=2] (\p3) coordinate (T3)
    to [out=-90, in=-90, looseness=2] (\p4) coordinate (T4)
    to (\x4, 2) node [above] {$A^{**}$};
    \fill (M) circle (1pt);
    \draw let \p1 = (M), \p2 = ($(T4)+(1.75,0)$)
    in (\p1) to [out=-90, in=-90, looseness=2] (\x2, \y1)
    to ($(\x2,2)$) node [above] {$A^{*}$};
    \draw let \p1 = (T1), \p2 = (T3)
    in (\p1) to [out=+90, in=+90, looseness=2] ($(\p2)+(1,0)$)
    to ($(\x2,-2)+(1,0)$) node [below] {$A^{*}$};
  \end{tikzpicture}
  \quad
  m^{\dagger} =
  \begin{tikzpicture}[x = 1pc, y = 1pc, baseline=-.5em]
    \draw let \p1 = (1,0), \p2 = ($(\p1)+(1,0)$),
    \p3 = ($(\p2)+(1,0)$)
    in (\p1) coordinate (T1)
    to [out=-90, in=-90, looseness=2] coordinate[pos = .5](M) (\p2)
    to [out=+90, in=+90, looseness=2] (\p3) coordinate (T3)
    to (\x3, -2) node [below] {$A^{*}$};
    \fill (M) circle (1pt);
    \draw let \p1 = (M), \p2 = ($(T1)-(1.5,0)$)
    in (\p1) to [out=-90, in=-90, looseness=2] (\x2, \y1)
    to ($(\x2,2)$) node [above] {$A^{*}$};
    \draw let \p1 = (T1) in (\p1) to (\x1, 2) node [above] {$A$};
  \end{tikzpicture}
\end{gather*}
Here,
$\begin{tikzpicture}[x = .5pc, y = .5pc, baseline=0]
  \draw (0, 1.25) to [out=-90, in=-90, looseness=2] coordinate[pos = .5](M) (2, 1.25);
  \fill (M) circle (1pt);
  \draw let \p1 = (M) in (\p1) -- (\x1, -.75);
\end{tikzpicture} : A \otimes A \to A$
is the multiplication of $A$. The evaluation and the coevaluation are expressed by a cup $\cup$ and a cap $\cap$, respectively.

\begin{definition}
  \label{def:rad-obj}
  The {\em Radford object} $\modobjb_{\mathcal{C}} \in \mathcal{Z}(\id_{\mathcal{C}}, S^{4})$ is defined to be the object corresponding to the $A^{**}$-$A$-bimodule $A^{*}$ via the equivalence
  \begin{equation*}
    \renewcommand{\xarrlen}[1]{6em}
    \mathcal{Z}(\id_{\mathcal{C}}, S^{4})
    \xarr{\text{Lemma~\ref{lem:Z-F-G-monadic}}}
    {}_{A^{**}}\mathcal{C}
    \xarr{\text{\eqref{eq:ENO-Hopf-bimod-2}}}
    {}_{A^{**}}(\mathcal{C}^{\env})_A.
  \end{equation*}
  The {\em distinguished invertible object} \cite{MR2097289} of $\mathcal{C}$ is defined to be the object $\modobj_{\mathcal{C}} \in \mathcal{C}$ such that $\iHom(\unitobj, \modobj_{\mathcal{C}}) \cong A^*$ as right $A$-modules.
  Thus there is a natural isomorphism $\delta_X: \modobj_{\mathcal{C}} \otimes X \to X^{****} \otimes \modobj_{\mathcal{C}}$ ($X \in \mathcal{C}$) such that $\modobjb_{\mathcal{C}} = (\modobj_{\mathcal{C}}, \delta)$. We refer to $\delta$ as the {\em Radford isomorphism} of $\mathcal{C}$.
\end{definition}

The proof of the following lemma is postponed to Appendix \ref{sec:def-rad-obj}.

\begin{lemma}
  \label{lem:Radford-iso-ENO}
  The Radford isomorphism $\delta$ defined in the above coincides with that induced by the isomorphism $X^{**} \cong \modobj_{\mathcal{C}} \otimes {}^{**}X \otimes (\modobj_{\mathcal{C}})^*$ of tensor functors given in \cite[Theorem 3.3]{MR2097289}.
\end{lemma}

\subsection{Relative modular object}

A tensor functor is said to be {\em perfect} \cite[Subsection 2.1]{MR3161401} if it has an exact left adjoint or, equivalently, it has an exact right adjoint.
Let $F: \mathcal{B} \to \mathcal{C}$ be a perfect tensor functor between finite tensor categories.
Since $F$ is a $\mathcal{B}$-bimodule functor from $\mathcal{B}$ to ${}_{\langle F \rangle} \mathcal{C}_{\langle F \rangle}$ in an obvious way, its double right adjoint $F^{\rradj} := (F^{\radj})^{\radj}$ is a $\mathcal{B}$-bimodule functor from $\mathcal{B}$ to ${}_{\langle F \rangle} \mathcal{C}_{\langle F \rangle}$.

\begin{definition}
  \label{def:rel-mod}
  Let $F$ be as above. The {\em relative modular object} of $F$ is the object $\relmodb_F \in \mathcal{Z}(F,F)$ corresponding to $F^{\rradj}$ via the category equivalence
  \begin{equation*}
    \Fun_{\mathcal{B} | \mathcal{B}}(\mathcal{B}, {}_{\langle F \rangle} \mathcal{C}_{\langle F \rangle})
    \xrightarrow[\approx]{\quad \text{Lemma~\ref{lem:center-as-bimod-func}} \quad}
    \mathcal{Z}(F, F).
  \end{equation*}
\end{definition}

We write $\relmodb_F = (\relmod_F, \gamma_F)$. As explained in \cite{MR3569179}, the object $\relmod_F$ is a categorical counterpart of the {\em relative modular function} of \cite{MR1401518}.
The main result of \cite{MR3569179} claims that there is an isomorphism $\relmod_F \cong \modobj_{\mathcal{C}}^* \otimes F(\modobj_{\mathcal{B}})$ in $\mathcal{C}$ (remark that $\modobj_{\mathcal{C}}$ and $\relmod_F$ in this paper are $\alpha_{\mathcal{C}}^*$ and $\chi_F^*$ of \cite{MR3569179}, respectively). For the purpose of this paper, we also require the following description of the isomorphism $\gamma_F$.

\begin{theorem}
  \label{thm:relative-modulus}
  There is an isomorphism $u : F(\modobj_{\mathcal{B}}) \to \modobj_{\mathcal{C}} \otimes \relmod_F$ in $\mathcal{B}$ such that the diagram
  \begin{equation*}
    \begin{tikzcd}[column sep = 3em]
      F(\modobj_{\mathcal{B}} \otimes X)
      \arrow[dd, "F(\delta_X)"']
      & F(\modobj_{\mathcal{B}}) \otimes F(X)
      \arrow[l, "F_2"']
      \arrow[r, "u \otimes \id"]
      & \modobj_{\mathcal{C}} \otimes \relmod_F \otimes F(X)
      \arrow[d, "\id \otimes \gamma_F(X)"] \\
      & & \modobj_{\mathcal{C}} \otimes F(X) \otimes \relmod_F
      \arrow[d, "\delta_{F(X)} \otimes \id"] \\
      F(X^{****} \otimes \modobj_{\mathcal{B}})
      \arrow[r, "(F_2)^{-1}"]
      & F(X^{****}) \otimes F(\modobj_{\mathcal{B}})
      \arrow[r, "\xi \otimes u"]
      & F(X)^{****} \otimes \modobj_{\mathcal{C}} \otimes \relmod_F
    \end{tikzcd}
  \end{equation*}
  commutes for all objects $X \in \mathcal{B}$, where $F_2$ is the monoidal structure of $F$ and $\xi$ is the isomorphism obtained by the iterative use of the duality transformation for $F$.
\end{theorem}

In other words, there is an isomorphism $F(\modobjb_{\mathcal{B}}) \cong \modobjb_{\mathcal{C}} \otimes \relmodb_F$ in $\mathcal{Z}(F, F S^4)$ if we view $\modobjb_{\mathcal{C}}$ as an object of $\mathcal{Z}(F, S^4 F)$ in a natural way and identify $F S^4$ with $S^4 F$ through the duality transformation for $F$.

It has been observed in \cite[Remark 4.17]{MR4042867} that one can prove that there is an isomorphism $F(\modobj_{\mathcal{B}}) \cong \modobj_{\mathcal{C}} \otimes \relmod_F$ by a basic property of the Nakayama functor for finite abelian categories. The above theorem will be proved by a slight modification of this idea.

\subsection{Nakayama functor}

For a finite abelian category $\mathcal{M}$, the (left exact) Nakayama functor \cite{MR4042867} is the $\bfk$-linear endofunctor $\Nak^{\ell}_{\mathcal{M}}$ on $\mathcal{M}$ defined by
\begin{equation*}
  \Nak^{\ell}_{\mathcal{M}}(M) = \int_{X \in \mathcal{M}} \Hom_{\mathcal{M}}(X, M) \otimes_{\bfk} X
  \quad (M \in \mathcal{M}).
\end{equation*}
A feature of the Nakayama functor that is important in this paper is the canonical isomorphism \eqref{eq:Nakayama-double-adj} below.
Let $\mathcal{M}$ and $\mathcal{N}$ be finite abelian categories, and let $F : \mathcal{M} \to \mathcal{N}$ be a $\bfk$-linear functor satisfying the following condition:
\begin{equation}
  \label{eq:Nakayama-double-adj-requirement}
  \text{its double left adjoint exists and is left exact}.
\end{equation}
Then we have isomorphisms
\begin{align*}
  \Nak^{\ell}_{\mathcal{N}} F(M)
  & = \textstyle \int_{Y \in \mathcal{N}} \Hom_{\mathcal{N}}(Y, F(M)) \otimes_{\bfk} Y \\
  & \cong \textstyle \int_{Y \in \mathcal{N}} \Hom_{\mathcal{N}}(F^{\ladj}(Y), M) \otimes_{\bfk} Y \\
  & \cong \textstyle \int_{X \in \mathcal{M}} \Hom_{\mathcal{N}}(X, M) \otimes_{\bfk} F^{\lladj}(X) \\
  & \cong \textstyle F^{\lladj}(\int_{X \in \mathcal{M}} \Hom_{\mathcal{N}}(X, M) \otimes_{\bfk} X)
    = F^{\lladj} \Nak^{\ell}_{\mathcal{M}}(M)
\end{align*}
for $M \in \mathcal{M}$, where the first one is the adjunction isomorphism,
the second one is given by Lemma~\ref{lem:coend-adj},
and the third one follows from that $F^{\lladj}$ is $\bfk$-linear and preserves limits. Summarizing, we have an isomorphism
\begin{equation}
  \label{eq:Nakayama-double-adj}
  \varphi_F^{\ell} : \Nak^{\ell}_{\mathcal{N}} \circ F \to F^{\lladj} \circ \Nak^{\ell}_{\mathcal{M}}
\end{equation}
of $\bfk$-linear functors \cite[Theorem 3.18]{MR4042867}. As noted in {\it loc. cit.}, this isomorphism is natural in the variable $F$ and fulfills a certain  `coherent' property.

\begin{remark}
  \label{rem:Nakayama-double-adj-naturality}
  For practical applications, choices of adjoints often matter.
  Suppose that $\bfk$-linear functors $F_i : \mathcal{M} \to \mathcal{N}$ ($i = 1, 2$) satisfy \eqref{eq:Nakayama-double-adj-requirement}.
  For each $i$, we fix a double left adjoint $F^{\lladj}_i$ of $F_i$. The naturality of \eqref{eq:Nakayama-double-adj} means that the diagram
  \begin{equation*}
    \begin{tikzcd}[column sep = 5em]
      \Nak^{\ell}_{\mathcal{N}} \circ F_1
      \arrow[r, "\varphi^{\ell}_{F_1}"]
      \arrow[d, "N^{\ell}_{\mathcal{N}} \circ \xi"']
      & F_1^{\lladj} \circ \Nak^{\ell}_{\mathcal{M}}
      \arrow[d, "\xi^{\lladj} \circ \Nak^{\ell}_{\mathcal{M}}"]
      \\ \Nak^{\ell}_{\mathcal{N}} \circ F_2
      \arrow[r, "\varphi^{\ell}_{F_2}"]
      & F_2^{\lladj} \circ \Nak^{\ell}_{\mathcal{M}}
    \end{tikzcd}
  \end{equation*}
  commutes for all natural transformation $\xi : F_1 \to F_2$. Now we consider the case where $F := F_1 = F_2$ and $\xi = \id$. Although $F_1^{\lladj}$ and $F_2^{\lladj}$ may be different, they are canonically isomorphic. The above commutative diagram says that the canonical isomorphism $\nu : F_1^{\lladj} \to F_2^{\lladj}$ makes the following diagram commute:
  \begin{equation*}
    \begin{tikzcd}[column sep = 5em]
      \Nak^{\ell}_{\mathcal{N}} \circ F_1
      \arrow[r, "\varphi^{\ell}_{F_1}"]
      \arrow[d, equal]
      & F_1^{\lladj} \circ \Nak^{\ell}_{\mathcal{M}}
      \arrow[d, "\nu \circ \Nak^{\ell}_{\mathcal{M}}"]
      \\ \Nak^{\ell}_{\mathcal{N}} \circ F_2
      \arrow[r, "\varphi^{\ell}_{F_2}"]
      & F_2^{\lladj} \circ \Nak^{\ell}_{\mathcal{M}}
    \end{tikzcd}
  \end{equation*}
\end{remark}

\begin{remark}
  Let $F : \mathcal{M} \to \mathcal{N}$ and $G : \mathcal{L} \to \mathcal{M}$ be $\bfk$-linear functors between finite abelian categories $\mathcal{L}$, $\mathcal{M}$ and $\mathcal{N}$. Suppose that $F$ and $G$ satisfy \eqref{eq:Nakayama-double-adj-requirement}. The coherence property of \eqref{eq:Nakayama-double-adj} means that the following diagram commutes:
  \begin{equation*}
    \begin{tikzcd}[column sep = 5em, row sep = 2em]
      \Nak^{\ell}_{\mathcal{N}} \circ F \circ G
      \arrow[d, equal]
      \arrow[rr, "\varphi^{\ell}_{F \circ G}"]
      & & (F \circ G)^{\lladj} \circ \Nak^{\ell}_{\mathcal{L}}
      \arrow[d, "\cong"] \\
      \Nak^{\ell}_{\mathcal{N}} \circ F \circ G
      \arrow[r, "\varphi^{\ell}_{F} \circ \id"]
      & F^{\lladj} \circ \Nak^{\ell}_{\mathcal{M}} \circ G
      \arrow[r, "\id \circ \varphi^{\ell}_{F}"]
      & F^{\lladj} \circ G^{\lladj} \circ \Nak^{\ell}_{\mathcal{L}}
    \end{tikzcd}
  \end{equation*}
\end{remark}

Now let $\mathcal{C}$ be a finite tensor category, and let $\mathcal{M}$ be a finite $\mathcal{C}$-bimodule category. For $X \in \mathcal{C}$, we define two endofunctors $T^{(\ell)}_{X}$ and $T^{(r)}_{X}$ on $\mathcal{M}$ by $T^{(\ell)}_X = X \catactl \id_{\mathcal{M}}$ and $T^{(r)}_X = \id_{\mathcal{M}} \catactr X$, respectively. There are isomorphisms
\begin{equation*}
  T^{(\ell)}_{X^{**}} \circ \Nak^{\ell}_{\mathcal{M}}
  = (T^{(\ell)}_{X})^{\lladj} \circ \Nak^{\ell}_{\mathcal{M}}
  \mathop{\cong}^{\eqref{eq:Nakayama-double-adj}}
  \Nak^{\ell}_{\mathcal{M}} \circ T^{(\ell)}_X
  \quad \text{and} \quad
  T^{(r)}_{{}^{**}\!X} \circ \Nak^{\ell}_{\mathcal{M}}
  \mathop{\cong}^{\eqref{eq:Nakayama-double-adj}}
  \Nak^{\ell}_{\mathcal{M}} \circ T^{(r)}_{X}
\end{equation*}
of endofunctors on $\mathcal{M}$. In other words, there are isomorphisms
\begin{equation*}
  X^{**} \catactl \Nak^{\ell}_{\mathcal{M}}(M) \cong \Nak^{\ell}_{\mathcal{M}}(X \catactl M)
  \quad \text{and} \quad
  \Nak^{\ell}_{\mathcal{M}}(M) \catactr X^{**} \cong \Nak^{\ell}_{\mathcal{M}}(M \catactr X)
\end{equation*}
for $X \in \mathcal{C}$ and $M \in \mathcal{M}$.
By the naturality and the coherence property of \eqref{eq:Nakayama-double-adj}, we see that these isomorphisms make the functor $\Nak^{\ell}_{\mathcal{M}}$ a left $\mathcal{C}$-module functor from $\mathcal{M}$ to ${}_{\langle S \rangle}\mathcal{M}_{\langle S^{-2}\rangle}$, where $S = (-)^*$ and $S^{-1} = {}^*(-)$ \cite[Corollary 4.9]{MR4042867}.

If $\mathcal{M}$ and $\mathcal{N}$ are finite $\mathcal{C}$-bimodule categories and $F : \mathcal{M} \to \mathcal{N}$ is a $\bfk$-linear $\mathcal{C}$-bimodule functor satisfying~\eqref{eq:Nakayama-double-adj-requirement}, then the both sides of \eqref{eq:Nakayama-double-adj} are $\mathcal{C}$-bimodule functors from $\mathcal{M}$ to ${}_{\langle S \rangle}\mathcal{N}_{\langle S^{-2}\rangle}$. The following observation is one of key ingredients of the proof of Theorem~\ref{thm:relative-modulus}.

\begin{lemma}
  \label{lem:Nakayama-double-adj-bimodule}
  For $F: \mathcal{M} \to \mathcal{N}$ as above, the isomorphism \eqref{eq:Nakayama-double-adj} is in fact an isomorphism of $\mathcal{C}$-bimodule functors from $\mathcal{M}$ to ${}_{\langle S^2 \rangle}\mathcal{N}_{\langle S^{-2}\rangle}$.
\end{lemma}
\begin{proof}
  Let $\xi_{X,M} : X \catactl F(M) \to F(X \catactl M)$ ($X \in \mathcal{C}$, $M \in \mathcal{M}$) be the left $\mathcal{C}$-module structure of $F$. We fix an object $X \in \mathcal{C}$ and write $T = T^{(\ell)}_X$ for simplicity. By the naturality and the coherence property of the isomorphism \eqref{eq:Nakayama-double-adj}, we see that the diagram given as Figure~\ref{fig:proof-lemma-Nakayama-adj-bimodule} is commutative.
  On the one hand, the first row of the diagram is the left $\mathcal{C}$-module structure of $\Nak^{\ell}_{\mathcal{N}} \circ F$. On the other hand, by Lemma~\ref{lem:mod-func-double-adj}, the second row is the left $\mathcal{C}$-module structure of $F^{\lladj} \circ \Nak^{\ell}_{\mathcal{M}}$. Thus the commutativity of the diagram implies that $\varphi^{\ell}_{F}$ is a morphism of left $\mathcal{C}$-module functors. It can be shown that $\varphi^{\ell}_{F}$ is also a morphism of right $\mathcal{C}$-module functors in a similar manner. The proof is done.
\end{proof}

\begin{figure}
  \centering
    \begin{equation*}
    \begin{tikzcd}[column sep = 8em, row sep = 2em]
      T^{\lladj} \circ \Nak^{\ell}_{\mathcal{N}} \circ F
      \arrow[r, "T^{\lladj} \circ \varphi^{\ell}_F"]
      \arrow[d, "(\varphi^{\ell}_{T})^{-1}"']
      \arrow[rd, phantom, "\scriptstyle\text{(coherence)}"]
      & T^{\lladj} \circ F^{\lladj} \circ \Nak^{\ell}_{\mathcal{M}}
      \arrow[d, "\cong"] \\
      \Nak^{\ell}_{\mathcal{N}} \circ T \circ F
      \arrow[r, "\varphi^{\ell}_{T \circ F}"]
      \arrow[d, "\Nak^{\ell}_{\mathcal{N}} \circ (\xi_{X,-})"']
      \arrow[rd, phantom, "\scriptstyle\text{(naturality)}"]
      & (T \circ F)^{\lladj} \circ \Nak^{\ell}_{\mathcal{M}}
      \arrow[d, "(\xi_{X,-})^{\lladj} \circ \Nak^{\ell}_{\mathcal{M}}"] \\
      \Nak^{\ell}_{\mathcal{N}} \circ F \circ T
      \arrow[dd, equal]
      \arrow[r, "\varphi^{\ell}_{F \circ T}"]
      \arrow[rdd, phantom, "\scriptstyle\text{(coherence)}"]
      & (F \circ T)^{\lladj} \circ \Nak^{\ell}_{\mathcal{M}}
      \arrow[d, "\cong"] \\
      & F^{\lladj} \circ T^{\lladj} \circ \Nak^{\ell}_{\mathcal{M}}
      \arrow[d, "F^{\lladj} \circ \varphi^{\ell}_{T}"] \\
      \Nak^{\ell}_{\mathcal{N}} \circ F \circ T
      \arrow[r, "\varphi^{\ell}_F \circ T"]
      & F^{\lladj} \circ \Nak^{\ell}_{\mathcal{M}} \circ T
    \end{tikzcd}
  \end{equation*}
  \caption{Proof of Lemma~\ref{lem:Nakayama-double-adj-bimodule}}
  \label{fig:proof-lemma-Nakayama-adj-bimodule}
\end{figure}

\subsection{Nakayama functor and Hopf bimodules}

Let $\mathcal{C}$ be a finite tensor category. We discuss the relation between the Radford isomorphism and the Nakayama functor of $\mathcal{C}$. We first introduce the $\bfk$-linear left exact functor
\begin{equation*}
  \Phi : \mathcal{C}^{\env} \to \Lex(\mathcal{C}) := \Lex(\mathcal{C}, \mathcal{C}),
  \quad V \boxtimes W^{\rev} \mapsto \Hom_{\mathcal{C}}(W^*, -) \otimes_{\bfk} V.
\end{equation*}
As noted in \cite[Remark 2.2]{MR2097289}, this functor is an equivalence (see \cite{MR3632104} and \cite{MR4042867} for the detailed proof and some variants of this equivalence). We make $\Lex(\mathcal{C})$ a $\mathcal{C}^{\env}$-bimodule category in such a way that $\Phi$ is an equivalence of bimodule categories. The resulting actions of $\mathcal{C}^{\env}$ are given by
\begin{gather}
  \label{eq:env-action-Lex-l}
  (X \boxtimes Y^{\rev}) \catactl F = X \otimes F(Y \otimes -), \\
  \label{eq:env-action-Lex-r}
  F \catactr (X \boxtimes Y^{\rev}) = F(- \otimes Y^{**}) \otimes X
\end{gather}
for $X, Y \in \mathcal{C}$ and $F \in \Lex(\mathcal{C})$. Now let $A \in \mathcal{C}^{\env}$ be the canonical algebra.

\begin{lemma}
  \label{lem:right-A-mod-in-Lex}
  The category of right $A$-modules in $\Lex(\mathcal{C})$ is isomorphic to the category of $\bfk$-linear right $\mathcal{C}$-module endofunctors on $\mathcal{C}$.
\end{lemma}
\begin{proof}
  By same way as Lemma~\ref{lem:center-as-bimod-func}, one can prove that the category $\Lex(\mathcal{C})_A$ is isomorphic to the category $\Lex_{\mathcal{C}}(\mathcal{C})$ of $\bfk$-linear left exact right $\mathcal{C}$-module functors endofunctors on $\mathcal{C}$.
  Since every $\bfk$-linear right $\mathcal{C}$-module endofunctors on $\mathcal{C}$ is of the form $V \otimes \id_{\mathcal{C}}$ for some $V \in \mathcal{C}$, the category $\Lex_{\mathcal{C}}(\mathcal{C})$ is actually equal to the category of $\bfk$-linear right $\mathcal{C}$-module endofunctors on $\mathcal{C}$. The proof is done.
\end{proof}

By the above lemma, the functor
\begin{equation}
  \label{eq:equiv-regular-right-module}
  \Lex(\mathcal{C})_A \to \mathcal{C},
  \quad F \mapsto F(\unitobj)
\end{equation}
is an equivalence of categories. This functor preserves the action~\eqref{eq:env-action-Lex-l}, and hence it is in fact an equivalence of left $\mathcal{C}^{\env}$-module categories.

\begin{lemma}
  A quasi-inverse of $\Kappa_{\mathcal{C}}$ is given by the composition
  \begin{equation}
    \label{eq:ENO-Hopf-bimod-3}
    \overline{\Kappa}_{\mathcal{C}}
    := \Big((\mathcal{C}^{\env})_A
      \xrightarrow{\quad \Phi \quad}
      \Lex(\mathcal{C})_A
      \xrightarrow{\quad \eqref{eq:equiv-regular-right-module} \quad}
      \mathcal{C}\Big).
  \end{equation}
\end{lemma}
\begin{proof}
  Since we have already known that $\Kappa_{\mathcal{C}}$ is an equivalence, it is enough to show that there is an isomorphism $\overline{\Kappa}_{\mathcal{C}} \Kappa_{\mathcal{C}} \cong \id_{\mathcal{C}}$ of left $\mathcal{C}^{\env}$-module functors.
  According to \cite[Subsection 4.3]{MR3632104}, a quasi-inverse of $\Phi$ is given by
  \begin{equation*}
    \Phi^{-1}(F) = \int^{X \in \mathcal{C}} F(X) \boxtimes {}^*\!X
    \quad (F \in \Lex(\mathcal{C})).
  \end{equation*}
  In particular, $A = \Phi^{-1}(\id_{\mathcal{C}})$. Since $\Phi$ is a left $\mathcal{C}^{\env}$-module functor, we have
  \begin{equation*}
    \overline{\Kappa}_{\mathcal{C}} \Kappa_{\mathcal{C}}(V)
    = \overline{\Kappa}_{\mathcal{C}}((V \boxtimes \unitobj) \otimes A)
    \cong (V \boxtimes \unitobj) \catactl \overline{\Kappa}_{\mathcal{C}}(A) \cong V
  \end{equation*}
  for $V \in \mathcal{C}$. Thus $\overline{\Kappa}_{\mathcal{C}} \Kappa_{\mathcal{C}} \cong \id_{\mathcal{C}}$. One can check that this isomorphism is in fact a morphism of left $\mathcal{C}^{\env}$-module functors. The proof is done.
\end{proof}

The functors in equation~\eqref{eq:ENO-Hopf-bimod-3} induce equivalences
\begin{equation}
  \label{eq:ENO-Hopf-bimod-4}
  {}_{A^{**}}(\mathcal{C}^{\env})_A
  \xrightarrow{\quad \Phi \quad}
  {}_{A^{**}}\Lex(\mathcal{C})_A
  \xrightarrow{\quad \eqref{eq:equiv-regular-right-module} \quad}
  {}_{A^{**}}\mathcal{C}
\end{equation}
between left the categories of $A^{**}$-modules. By the above lemma, \eqref{eq:ENO-Hopf-bimod-4} is a quasi-inverse of the equivalence \eqref{eq:ENO-Hopf-bimod-2} and thus sends the $A^{**}$-$A$-bimodule $A^*$ to $\modobjb_{\mathcal{C}}$ under the identification ${}_{A^{**}}\mathcal{C} \cong \mathcal{Z}(\id_{\mathcal{C}}, S^4)$ by Lemma~\ref{lem:Z-F-G-monadic}.

It is natural to ask what an object of the category ${}_{A^{**}}\Lex(\mathcal{C})_A$ corresponding to the bimodule $A^*$ is.
By Lemmas~\ref{lem:Z-F-G-monadic} and \ref{lem:right-A-mod-in-Lex}, we have an isomorphism
\begin{equation*}
  {}_{A^{**}}\Lex(\mathcal{C})_A
  \cong \Fun_{\mathcal{C}|\mathcal{C}}(\mathcal{C}, {}_{\langle S^4 \rangle}\mathcal{C})
\end{equation*}
of categories. Now we define $\check{\Nak}: \mathcal{C} \to \mathcal{C}$ by $\check{\Nak}(X) = \Nak^{\ell}_{\mathcal{C}}(X^{**})$ for $X \in \mathcal{C}$. Since the Nakayama functor $\Nak^{\ell}_{\mathcal{C}}$ is a $\mathcal{C}$-bimodule functor from $\mathcal{C}$ to ${}_{\langle S^2 \rangle}\mathcal{C}_{\langle S^{-2} \rangle}$, the functor $\check{\Nak}$ is naturally a $\mathcal{C}$-bimodule functor from $\mathcal{C}$ to ${}_{\langle S^4 \rangle}\mathcal{C}$. The above question is answered as follows:

\begin{lemma}
  $\Phi(A^*) \cong \check{\Nak}$ in $\Fun_{\mathcal{C}|\mathcal{C}}(\mathcal{C}, {}_{\langle S^4 \rangle}\mathcal{C})$.
\end{lemma}
\begin{proof}
  Since the duality functor $(-)^* : \mathcal{C}^{\env} \to \mathcal{C}^{\env}$ is an anti-equivalence, it turns a coend into an end, and thus we have isomorphisms
  \begin{equation}
    \label{eq:canonical-alg-dual-as-end}
    A^*
    \cong \left(\int^{\smash{X \in \mathcal{C}}} {}^*\!X \boxtimes {}^{**}\!X \right)^{\!\!*}
    \cong \int_{X \in \mathcal{C}} ({}^*\!X \boxtimes {}^{**}\!X)^*
    \cong \int_{X \in \mathcal{C}} X \boxtimes {}^{***}\!X
  \end{equation}
  in $\mathcal{C}^{\env}$. Specifically, the object $A^*$ is an end of the above form with the universal dinatural transformation given by
  \begin{equation}
    \pi_X := (\din_{{}^{*}\!X})^*: A^* \to X \boxtimes {}^{***}\!X \quad (X \in \mathcal{C}).
  \end{equation}
  Now we fix an object $V \in \mathcal{C}$ and define $\mathrm{E}_{V} : \Lex(\mathcal{C}) \to \mathcal{C}$ by $F \mapsto F(V)$.
  Since $\Phi$ is an equivalence, and since $\mathrm{E}_{V}$ is left exact (see Subsection~\ref{subsec:fin-ab-cat}), we have
  \begin{align*}
    \Phi(A^*)(V)
    & \cong \textstyle \int_{X \in \mathcal{C}} \mathrm{E}_{V}(\Phi(X \boxtimes {}^{***}\!X)) \\
    & \cong \textstyle \int_{X \in \mathcal{C}} \Hom_{\mathcal{C}}({}^{**}\!X, V) \otimes_{\bfk} X \\
    & \cong \textstyle \int_{X \in \mathcal{C}} \Hom_{\mathcal{C}}(X, V^{**}) \otimes_{\bfk} X = \check{\Nak}(V)
  \end{align*}
  as $\bfk$-linear functors.

  As a next step, we examine the $\mathcal{C}$-bimodule structure of $N := \Phi(A^*)$.
  We fix an object $V \in \mathcal{C}$.
  For simplicity, we write $V^{***} = V^{*3}$, $V^{****} = V^{*4}$, etc.
  We define the morphism $\tau_V^{(\ell)} : (V^{*4}\boxtimes\unitobj) \otimes A^* \to (\unitobj \boxtimes V) \otimes A^*$ by
  \begin{equation}
    \tau^{(\ell)}_V = (\id_{\unitobj \boxtimes V} \otimes m^{\ddag}(\din_{V^{**}}^{**} \otimes \id_{A^*}))
    \circ ((\id_{V^{*4}} \boxtimes \coev_{{}^* V}) \otimes \id_{A^*}),
  \end{equation}
  where $m^{\ddag}$ is the left action of $A^{**}$ on $A^*$ given by \eqref{eq:cano-alg-m-dagger}.
  When we view $\mathcal{M} = \mathcal{C}^{\env}$ as a $\mathcal{C}$-bimodule category by the left action of $\mathcal{C}^{\env}$ on itself, $\tau^{(\ell)}$ is the inverse of the structure morphism of $A^*$ as an object of $\mathcal{Z}({}_{\langle S^4 \rangle}\mathcal{M})$. Thus the ($S^4$-twisted) left $\mathcal{C}$-module structure of $N$ is given by
  \begin{gather*}
    V^{*4} \otimes N(W)
    \mathop{=}^{\eqref{eq:env-action-Lex-l}} \Big((V^{*4} \boxtimes \unitobj) \catactl \Phi(A^*)\Big)(W)
    \cong \Phi\Big( (V^{*4} \boxtimes \unitobj) \otimes A^* \Big)(W) \\
    \xrightarrow{\ \Phi(\tau^{(\ell)}_{V})_W^{} \ }
    \Phi\Big( (\unitobj \boxtimes V) \otimes A^* \Big)(W)
    \cong \Big( (\unitobj \boxtimes V) \catactl \Phi(A^*) \Big)(W)
    \mathop{=}^{\eqref{eq:env-action-Lex-l}} N(V \otimes W)
  \end{gather*}
  for $V, W \in \mathcal{C}$, where $\cong$'s are the left $\mathcal{C}^{\env}$-module structure of $\Phi$. In a similar manner, we see that the right $\mathcal{C}$-module structure of $N$ is given by
  \begin{gather*}
    N(V) \otimes W
    \mathop{=}^{\eqref{eq:env-action-Lex-r}} \Big( \Phi(A^*) \catactr (W \boxtimes \unitobj) \Big)(V)
    \cong \Phi \Big(A^* \otimes (W \boxtimes \unitobj) \Big)(V) \\
    \xrightarrow{\ \Phi(\tau^{(r)}_{W})_V^{} \ }
    \Phi\Big( A^* \otimes (\unitobj \boxtimes {}^{**} W) \Big)(V)
    \cong \Big( \Phi(A^*) \catactr (\unitobj \boxtimes {}^{**} W) \Big)(V)
    \mathop{=}^{\eqref{eq:env-action-Lex-r}} N(V \otimes W)
  \end{gather*}
  for $V, W \in \mathcal{C}$. Here, $\cong$'s are the right $\mathcal{C}^{\env}$-module structure of $\Phi$ and
  \begin{equation}
    \tau^{(r)}_W = (m^{\dagger}(\id_{A^*} \otimes \din_{W}) \otimes \id_{\unitobj \boxtimes {}^{**} W})
    \circ (\id_{A^*} \otimes (\id_W \boxtimes \coev_{{}^{**} W})),
  \end{equation}
  where $m^{\dagger}$ is the right action of $A$ on $A^*$ given by~\eqref{eq:cano-alg-m-flat}.

  By comparing the above description of the $\mathcal{C}$-bimodule structure of $N$ and that of the functor $\check{\Nak}$, we see that, to complete the proof, it suffices to show that the morphisms $\tau^{(\ell)}_V$ and $\tau^{(r)}_W$ are equal to the isomorphisms obtained by Lemma \ref{lem:mod-func-adj} applied to
  $V^{*4} \otimes \id_{\mathcal{C}} \dashv V^{*3} \otimes \id_{\mathcal{C}}$
  and
  $\id_{\mathcal{C}} \otimes W \dashv \id_{\mathcal{C}} \otimes W^*$,
  respectively.

  By the definition of the multiplication $m$, we have
  \begin{align*}
    \pi_{X} \circ m^{\ddag} \circ (\din_{V^{**}}^{**} \otimes \id_{A^*})
    & = (\eval_{A^*} \otimes \pi_{X}) \circ (\pi_{V^{*3}}^* \otimes m^*) \\
    & = (\eval_{V^{*3} \boxtimes V} \otimes \id_{X \boxtimes {}^{***}\!X})
      (\id_{V^{*4} \boxtimes {}^{*}V} \otimes (\pi_{V^{*3}} \otimes \pi_X)m^*) \\
    & = (\eval_{V^{*3} \boxtimes V} \otimes \id_{X \boxtimes {}^{***}\!X})
      (\id_{V^{*4} \boxtimes {}^{*}V} \otimes \pi_{V^{*3} \otimes X})
  \end{align*}
  for $V, X \in \mathcal{C}$. Thus we compute
  \begin{align*}
    (\id_{V \boxtimes \unitobj} \otimes \pi_X) \circ \tau^{(\ell)}_V
    & = \begin{gathered}[t]
      (\id_{\unitobj \boxtimes V} \otimes (\eval_{V^{*3} \boxtimes V} \otimes \id_{X \boxtimes {}^{***}\!X})
      (\id_{V^{*4} \boxtimes {}^{*}V} \otimes \pi_{V^{*3} \otimes X})) \\
      {} \circ (\id_{\unitobj \boxtimes V} \otimes (\id_{V^{****}} \boxtimes \coev_{{}^* V}) \otimes \id_{A^*})
    \end{gathered} \\
    & = \begin{gathered}[t]
      ((\eval_{V^{*3}} \otimes \id_X) \boxtimes
      (\id_{{}^{***}\!X} \otimes (\id_{{}^* \! V} \otimes \eval_{{}^*V})(\coev_{{}^* \! V} \otimes \id_{{}^* \! V}))) \\
      {} \circ (\id_{V^{*4} \boxtimes {}^* \! V} \otimes \pi_{V^{*3} \otimes X})
    \end{gathered} \\
    & = ((\eval_{V^{*3}} \otimes \id_X) \boxtimes \id_{{}^{***}\!X \otimes V})
      \circ (\id_{V^{*4} \boxtimes {}^* \! V} \otimes \pi_{V^{*3} \otimes X}),
  \end{align*}
  where the equation $\eval_{M \boxtimes N} = \eval_M \boxtimes \eval_{{}^* \! N}$ in $\mathcal{C}^{\env}$ is used at the second equality. This shows that $\tau^{(\ell)}_V$ is the isomorphism obtained by applying  Lemma \ref{lem:mod-func-adj} to the adjunction $V^{*4} \otimes \id_{\mathcal{C}} \dashv V^{*3} \otimes \id_{\mathcal{C}}$. Noting the equation
  \begin{equation*}
    \pi_X \circ m^{\dagger} \circ (\id_{A^*} \otimes \din_W)
    = (\id_{X \boxtimes {}^{*3} \! X} \otimes \eval_{W \boxtimes {}^* \! W}) \circ (\pi_{W^* \otimes X} \otimes \id_{W \boxtimes {}^* \! W}),
  \end{equation*}
  one can check that $\tau^{(r)}_W$ is the isomorphism obtained by applying Lemma~\ref{lem:mod-func-adj} to the adjunction $\id_{\mathcal{C}} \otimes W \dashv \id_{\mathcal{C}} \otimes W^*$ in a similar way as above. The proof is done.
\end{proof}

The above lemma gives the following description of the Radford object: Since the bimodule $A^*$ corresponds to $\modobjb_{\mathcal{C}} = (\modobj_{\mathcal{C}}, \delta)$ via the equivalence \eqref{eq:ENO-Hopf-bimod-4}, we may assume that the object $\modobj_{\mathcal{C}}$ is given by
\begin{equation}
  \label{eq:Radford-by-Nakayama-1}
  \modobj_{\mathcal{C}}
  = \check{\Nak}(\unitobj)
  = \Nak_{\mathcal{C}}^{\ell}(\unitobj) = \int_{X \in \mathcal{C}} \Hom_{\mathcal{C}}(X, \unitobj) \otimes_{\bfk} X,
\end{equation}
and the Radford isomorphism is given by the composition
\begin{equation}
  \label{eq:Radford-by-Nakayama-2}
  \delta_{V} = \Big( \modobj_{\mathcal{C}} \otimes V
  \xrightarrow{\quad \xi^{(r)}_{\unitobj, V^{**}} \quad} \Nak^{\ell}_{\mathcal{C}}(V^{**})
  \xrightarrow{\quad (\xi^{(\ell)}_{V^{**}, \unitobj})^{-1} \quad}
  V^{****} \otimes \modobj_{\mathcal{C}} \Big)
\end{equation}
for $V \in \mathcal{C}$, where $\xi^{(\ell)}$ and $\xi^{(r)}$ are the left and the right $\mathcal{C}$-module structure of the Nakayama functor $\Nak^{\ell}_{\mathcal{C}} : \mathcal{C} \to {}_{\langle S^2 \rangle}\mathcal{C}_{\langle S^{-2}\rangle}$, respectively.

\begin{remark}
  By the above discussion, we conclude that the isomorphism $X^{**} \cong \modobj_{\mathcal{C}} \otimes {}^{**}X \otimes (\modobj_{\mathcal{C}})^*$ of tensor functors given in \cite[Corollary 4.12]{MR4042867} coincides with that given in \cite{MR2097289}.
\end{remark}

\subsection{Proof of Theorem~\ref{thm:relative-modulus}}

We give a proof of Theorem~\ref{thm:relative-modulus}.
The setting is as follows: $\mathcal{B}$ and $\mathcal{C}$ are finite tensor categories, $F : \mathcal{B} \to \mathcal{C}$ is a perfect tensor functor, and $\relmodb_F = (\relmod_F, \gamma_F)$ is the relative modular object of $F$.

Let $\kappa_X : F(X)^{**} \to F(X^{**})$ ($X \in \mathcal{B}$) the isomorphism obtained by using the duality transformation. We first remark the following technical lemma:

\begin{lemma}
  \label{lem:relative-modulus-pf-1}
  There is an isomorphism $p : \mu_F^{**} \to \mu_F^{\phantom{*}}$ in $\mathcal{B}$ such that
  \begin{equation}
    \label{eq:relative-modulus-pf-3}
    (\kappa_X \otimes p) \circ \gamma_F(X)^{**} = \gamma_F(X^{**}) \circ (p \otimes \kappa_X)
    \quad (X \in \mathcal{B}).
  \end{equation}
\end{lemma}
\begin{proof}
  Let, in general, $G$ and $H$ be tensor functors from $\mathcal{B}$ to $\mathcal{C}$. If $G \cong H$ as tensor functors, then $\mathcal{Z}(G, G) \cong \mathcal{Z}(H, H)$ as categories. One can check that $\relmodb_{G}$ corresponds to $\relmodb_{H}$ under this isomorphism.

  We consider the case where $G = S^2 F$ and $H = F S^2$. The natural isomorphism $\kappa: G \to H$ is actually an isomorphism of tensor functors. Noting that the double dual functor is an equivalence, one can check that $\relmodb_G$ is given by
  \begin{equation*}
    \relmodb_G = \Big(\mu_F^{**},
      \{ \mu_F^{**} \otimes F(X)^{**} \xrightarrow{\quad \gamma_F(X)^{**} \quad} F(X)^{**} \otimes \mu_F^{**} \}_{X \in \mathcal{B}} \Big),
  \end{equation*}
  while $\relmodb_H$ is given by
  \begin{equation*}
    \relmodb_H = \Big(\mu_F,
      \{ \mu_F \otimes F(X^{**}) \xrightarrow{\quad \gamma_F(X^{**}) \quad} F(X^{**}) \otimes \mu_F \}_{X \in \mathcal{B}} \Big).
  \end{equation*}
  By the above argument, there is an isomorphism $p: \relmodb_G \to \relmodb_H$ in $\mathcal{Z}(H,H)$ if we view $\relmodb_{G}$ as an object of $\mathcal{Z}(H,H)$ through the isomorphism $\kappa$. It follows from the definition that $p$ satisfies \eqref{eq:relative-modulus-pf-3}. The proof is done.
\end{proof}

By the perfectness of $F$ and Lemma~\ref{lem:Nakayama-double-adj-bimodule}, there is an isomorphism
\begin{equation}
  \label{eq:pf-rel-mod-1}
  F \circ \Nak_{\mathcal{B}}^{\ell} \cong \Nak_{\mathcal{C}}^{\ell} \circ F^{\rradj}
\end{equation}
of $\mathcal{B}$-bimodule functors from $\mathcal{B}$ to ${}_{\langle F S^2 \rangle}\mathcal{C}_{\langle F S^{-2} \rangle}$. We may view these functors as $\mathcal{B}$-bimodule functors from $\mathcal{B}$ to ${}_{\langle F S^4 \rangle}\mathcal{C}_{\langle F \rangle}$. The theorem will be proved by computing the objects of $\mathcal{Z}(F, F S^4)$ corresponding to the both sides of \eqref{eq:pf-rel-mod-1}.

By \eqref{eq:Radford-by-Nakayama-1} and \eqref{eq:Radford-by-Nakayama-2}, we easily see that the left-hand side of~\eqref{eq:pf-rel-mod-1} corresponds to the object $F(\modobjb_{\mathcal{B}}) \in \mathcal{Z}(F, F S^4)$. Now we consider the right-hand side. Let $\xi^{(\ell)}$ and $\xi^{(r)}$ denote the left and the right $\mathcal{C}$-module structure of $\Nak := \Nak^{\ell}_{\mathcal{C}}$, respectively. In \eqref{eq:pf-rel-mod-1}, $\Nak$ is, rather, dealt as a $\mathcal{B}$-bimodule functor. We denote by
\begin{equation*}
  \Xi^{(\ell)}_{X,V} : X^{**} \catactl_F \Nak(V) \to \Nak(X \catactl_F V), \quad
  \Xi^{(r)}_{V,X} : \Nak(V) \catactr_F {}^{**}\!X \to \Nak(V \catactr_F X)
\end{equation*}
the $\mathcal{B}$-bimodule structure of $\Nak$.

\begin{lemma}
  \label{lem:relative-modulus-pf-2}
  For $X \in \mathcal{B}$ and $V \in \mathcal{C}$, we have
  \begin{equation*}
    \xi^{(\ell)}_{F(X),V} = \Xi^{(\ell)}_{X,V} \circ (\kappa_X \otimes \id_{\Nak(V)})
    \quad \text{and} \quad
    \xi^{(r)}_{V,F(X)} \circ (\id_{\Nak(V)} \otimes \,{}^{**}\!\kappa_{X^{**}}) = \Xi^{(r)}_{V,X}.
  \end{equation*}
\end{lemma}
\begin{proof}
  We fix an object $X \in \mathcal{B}$ and define three $\bfk$-linear endofunctors $T$, $T_1$ and $T_2$ on $\mathcal{C}$ by the following formulas:
  \begin{equation*}
    T(V) = F(X) \otimes V,
    \quad T_1(V) = F(X)^{**} \otimes V,
    \quad T_2(V) = F(X^{**}) \otimes V.
  \end{equation*}
  Both $T_1$ and $T_2$ are double left adjoints of $T$, and hence there is a canonical isomorphism $\nu : T_1 \to T_2$. By the construction, we have $\nu_V = \kappa_X \otimes \id_V$ for all objects $V \in \mathcal{C}$. Now the first equation of this lemma is proved by applying the argument of Remark~\ref{rem:Nakayama-double-adj-naturality} to $F_i^{\lladj} = T_i$. The latter equation can be proved in a similar way.
\end{proof}

\begin{proof}[Proof of Theorem~\ref{thm:relative-modulus}]
  We consider the isomorphism
  \begin{equation*}
    q := \Big(
    \modobj \otimes \mu
    \xrightarrow{\quad \id \otimes {}^{**}p \quad}
    \Nak(\unitobj) \otimes {}^{**}\!\mu
    \xrightarrow{\quad \xi^{(r)}_{\unitobj,\mu} \quad}
    \Nak(\mu) \Big),
  \end{equation*}
  where $\Nak = \Nak^{\ell}_{\mathcal{C}}$, $\modobj = \modobj_{\mathcal{C}}$ ($=N(\unitobj)$), $\mu = \mu_F$ and $p : \mu \to \mu^{**}$ is the isomorphism given by Lemma~\ref{lem:relative-modulus-pf-1}. For $X \in \mathcal{B}$, we set $\widetilde{\kappa}_X = \kappa_{X^{**}} \circ \kappa_X^{**}$. Then we have the commutative diagram given as Figure~\ref{fig:pf-relative-modulus-1}.
  Let $\tau_X$ and $\sigma_X$ be the morphisms obtained by composing the left and the right column, respectively.
  Thus we have the following commutative diagram:
  \begin{equation*}
    \begin{tikzcd}[column sep = 10em]
      \Nak(\mu) \otimes F(X)
      \arrow[d, "\sigma_X"']
      \arrow[r, "q \otimes \id_{F(X)}"]
      & \modobj \otimes \mu \otimes F(X)
      \arrow[d, "\tau_X"] \\
      F(X^{****}) \otimes \Nak(\mu)
      \arrow[r, "\id_{F(X^{****})} \otimes q"]
      & F(X^{****}) \otimes \modobj \otimes \mu
    \end{tikzcd}
  \end{equation*}
  By Lemma~\ref{lem:relative-modulus-pf-2}, the $\mathcal{B}$-bimodule functor $\Nak \circ F^{\rradj} : \mathcal{B} \to {}_{\langle F S^4 \rangle} \mathcal{C}_{\langle F \rangle}$ corresponds to $(\Nak(\mu), \sigma) \in \mathcal{Z}(F, F S^4)$.
  If we identify $S^4 F$ with $F S^4$ via the isomorphism $\widetilde{\kappa}$, then $\modobjb_{\mathcal{C}} \otimes \relmodb_F = (\modobj \otimes \mu, \tau)$. The above commutative diagram means that $(\Nak(\mu), \sigma)$ and $\modobjb_{\mathcal{C}} \otimes \relmodb_F$ are isomorphic as objects of $\mathcal{Z}(F, F S^4)$. Hence, by \eqref{eq:pf-rel-mod-1}, we have
  \begin{equation*}
    F(\modobjb_{\mathcal{B}}) \cong (\Nak(\mu), \sigma) \cong \modobjb_{\mathcal{C}} \otimes \relmodb_F.
    \qedhere
  \end{equation*}
\end{proof}

\begin{figure}
  \centering
  \begin{tikzcd}[column sep = 5em, row sep = 2.5em]
    \Nak(\mu) \otimes F(X)
    \arrow[d, "\xi^{(r)}_{\mu,F(X)^{**}}"']
    & \modobj \otimes {}^{**}\!\mu \otimes F(X)
    \arrow[l, "\xi^{(r)}_{\unitobj, \mu} \otimes \id"']
    \arrow[dl, "\smash{\xi^{(r)}_{\unitobj, \mu \otimes F(X)^{**}}}\ "']
    & \modobj \otimes \mu \otimes F(X)
    \arrow[l, "\id \otimes {}^{**}p \otimes \id"']
    \arrow[dl, "\smash{\xi^{(r)}_{\unitobj, \mu^{**} \otimes F(X)^{**}}}\ "']
    \arrow[dd, "\id \otimes \gamma_F(X)"]
    \\[+1em] 
    \Nak(\mu \otimes F(X)^{**})
    \arrow[d, "\Nak(\id \otimes \kappa_X)"']
    & \Nak(\mu^{**} \otimes F(X)^{**})
    \arrow[l, "\Nak(p \otimes \id)"]
    \arrow[dd, pos=.3, "\Nak(\gamma_F(X)^{**})"]
    \arrow[ddl, phantom, "\scriptstyle \text{(Lemma~\ref{lem:relative-modulus-pf-1})}"]
    \\ 
    \Nak(\mu \otimes F(X^{**}))
    \arrow[d, "\Nak(\gamma_F(X^{**}))"']
    & & \modobj \otimes F(X) \otimes \mu
    \arrow[dl, "\smash{\xi^{(r)}_{\unitobj,F(X)^{**}\otimes \mu^{**}}}\ "']
    \arrow[d, "\xi^{(r)}_{\unitobj,F(X)^{**}}"]
    \\[+1em] 
    \Nak(F(X^{**}) \otimes \mu)
    \arrow[d, "(\xi^{(\ell)}_{F(X^{**}),\mu})^{-1}"']
    & \Nak(F(X)^{**} \otimes \mu^{**})
    \arrow[l, "\Nak(\kappa_X \otimes p)"']
    \arrow[d, "(\xi^{(\ell)}_{F(X^{**}),\mu^{**}})^{-1}"']
    & \Nak(F(X)^{**}) \otimes \mu
    \arrow[l, "\xi^{(r)}_{F(X)^{**},\mu^{**}}"]
    \arrow[d, "(\xi^{(\ell)}_{F(X)^{**},\unitobj})^{-1}"]
    \\[+0em] 
    F(X^{**})^{**} \otimes \Nak(\mu)
    \arrow[d, "\kappa_{X^{**}}"']
    & F(X)^{****} \otimes \Nak(\mu^{**})
    \arrow[l, "\kappa_X^{**} \otimes \Nak(p)"]
    \arrow[d, "\widetilde{\kappa}_X"]
    & F(X)^{****} \otimes \modobj \otimes \mu
    \arrow[l, "\id \otimes \xi^{(r)}_{\unitobj,\mu^{**}}"]
    \arrow[d, "\widetilde{\kappa}_X \otimes \id \otimes \id"]
    \\ 
    F(X^{****}) \otimes \Nak(\mu)
    & F(X^{****}) \otimes \Nak(\mu^{**})
    \arrow[l, "\id \otimes \Nak(p)"]
    & F(X^{****}) \otimes \modobj \otimes \mu
    \arrow[l, "\id \otimes \xi^{(r)}_{\unitobj,\mu^{**}}"]
  \end{tikzcd}
  \label{fig:pf-relative-modulus-1}
  \caption{Proof of Theorem~\ref{thm:relative-modulus}}
\end{figure}

\section{Ribbon structures of the Drinfeld center}
\label{sec:main-result}

\subsection{Ribbon pivotal structure}

In this section, we classify the ribbon structures of the Drinfeld center of a finite tensor category. Let $\mathcal{B}$ be a braided rigid monoidal category with braiding $\sigma$. We recall that a {\em twist} (or a {\em ribbon structure}) of $\mathcal{B}$ is a natural isomorphism $\theta: \id_{\mathcal{B}} \to \id_{\mathcal{B}}$ such that the equations
\begin{gather}
  \label{eq:twist-def-1}
  \theta_{X \otimes Y} = (\theta_X \otimes \theta_Y) \circ \sigma_{Y,X} \circ \sigma_{X,Y} \\
  \label{eq:twist-def-2}
  (\theta_X)^* = \theta_{X^*}
\end{gather}
hold for all objects $X, Y \in \mathcal{B}$. We denote by $S_{\mathcal{B}}$ the left duality functor of $\mathcal{B}$. The {\em Drinfeld isomorphism} is the natural isomorphism $u: \id_{\mathcal{B}} \to S_{\mathcal{B}}^2$ defined by
\begin{equation}
  \label{eq:Dri-iso-def}
  \renewcommand{\xarrlen}{3.5em}
  u_X = \Big( X \xarr{\id \otimes \coev}
  X \otimes X^* \otimes X^{**}
  \xarr{\sigma \otimes \id}
  X^* \otimes X \otimes X^{**}
  \xarr{\eval \otimes \id} X^{**} \Big)
\end{equation}
for $X \in \mathcal{B}$. It is well-known that $u$ satisfies
\begin{equation}
  \label{eq:Dri-iso-tensor}
  u_{X \otimes Y} = (u_X \otimes u_Y)  \sigma_{X,Y}^{-1} \sigma_{Y,X}^{-1}
\end{equation}
for all objects $X, Y \in \mathcal{B}$ \cite[Proposition 8.9.3]{MR3242743}.

\begin{definition}
  \label{def:ribbon-pivot}
  Let $\mathcal{B}$ and $u$ be as above. A {\em ribbon pivotal structure} of $\mathcal{B}$ is a pivotal structure $j$ of $\mathcal{B}$ such that the equation
  \begin{equation}
    \label{eq:w-iso}
    j_X^{**} \circ j_X^{\phantom{*}} = u_{X}^{**} \circ (u_{\, {}^* \! X}^{*})^{-1}
  \end{equation}
  holds for all $X \in \mathcal{B}$.
\end{definition}

This terminology is justified by the following lemma:

\begin{lemma}
  \label{lem:ribbon-pivot}
  Let $\mathcal{B}$ and $u$ be as above. A natural transformation $\theta: \id_{\mathcal{B}} \to \id_{\mathcal{B}}$ is a twist if and only if $j := u \theta$ is a ribbon pivotal structure.
\end{lemma}
\begin{proof}
  To prove this lemma, we recall from \cite[Appendix A]{MR2095575} that every pivotal structure $j$ of a rigid monoidal category $\mathcal{C}$ satisfies the following equation:
  \begin{equation}
    \label{eq:pivot-inverse}
    j_{X}^* = (j_{X^*}^{\phantom{*}})^{-1} \quad (X \in \mathcal{C}).
  \end{equation}
  Let $\theta: \id_{\mathcal{B}} \to \id_{\mathcal{B}}$ be a natural transformation, and set $j = u \theta$. Suppose that $\theta$ is a twist. Then, by~\eqref{eq:twist-def-1} and~\eqref{eq:Dri-iso-tensor}, we see that $j$ is a pivotal structure of $\mathcal{C}$.
  By~\eqref{eq:twist-def-2} and~\eqref{eq:pivot-inverse}, we compute
  \begin{equation*}
    u^{-1}_X \circ j_X = \theta_X = (\theta_{{}^* \! X})^*
    = (j_{\, {}^* \! X})^* \circ (u_{\, {}^* \! X}^{-1})^*
    = j_X^{-1} \circ (u_{\, {}^* \! X}^*)^{-1}
  \end{equation*}
  for all $X \in \mathcal{B}$. We also note $j_{X^{**}} = (j_X)^{**}$ ($X \in \mathcal{B}$). By the naturality of $j$ and the above computation, we verify \eqref{eq:w-iso} as follows:
  \begin{align*}
    j_{X}^{**} \circ j_X
    & = j_{X}^{**} \circ u_X \circ j_X^{-1} \circ (u_{\, {}^* \! X}^*)^{-1} \\
    & = j_{X}^{**} \circ j_{X^{**}}^{-1} \circ u_X^{**} \circ (u_{\, {}^* \! X}^*)^{-1}
      = u_X^{**} \circ (u_{\, {}^* \! X}^*)^{-1}.
  \end{align*}
  Namely, $j$ is a ribbon pivotal structure. Suppose, conversely, that $j$ is a ribbon pivotal structure of $\mathcal{B}$. Then it is easy to see that $\theta: \id_{\mathcal{B}} \to \id_{\mathcal{B}}$ is a natural isomorphism satisfying \eqref{eq:twist-def-1}. By~\eqref{eq:w-iso} and~\eqref{eq:pivot-inverse}, we have
  \begin{align*}
    \theta_X^*
    = j_{X}^* \circ (u_X^{-1})^*
    & = j_{X}^* \circ (u_{X^*}^{**})^{-1} \circ j_{X^*}^{**} \circ j_{X^*} \\
    & = j_{X}^* \circ (u_{X^*}^{-1})^{**} \circ j_{X^{***}} \circ j_{X^*} \\
    & = j_{X}^* \circ j_{X^*} \circ u_{X^*}^{-1} \circ j_{X^*} =\theta_{X^*}
  \end{align*}
  for all $X \in \mathcal{B}$. Hence $\theta$ is a twist. The proof is done.
\end{proof}

\subsection{Ribbon pivotal structures of the Drinfeld center}

Let $\mathcal{C}$ be a finite tensor category. In view of Lemma~\ref{lem:ribbon-pivot}, we classify the ribbon pivotal structures of $\mathcal{Z}(\mathcal{C})$ instead of the ribbon structures of $\mathcal{Z}(\mathcal{C})$.  To state our classification result, we first recall a result from \cite{2016arXiv160805905S} used to classify the pivotal structures of $\mathcal{Z}(\mathcal{C})$.

Given a tensor autoequivalence $F$ of $\mathcal{C}$, we denote by $\widetilde{F}: \mathcal{Z}(\mathcal{C}) \to \mathcal{Z}(\mathcal{C})$ the braided tensor autoequivalence of $\mathcal{Z}(\mathcal{C})$ induced by $F$. Namely, it is defined by
\begin{equation*}
  \widetilde{F}(\mathbf{V}) = (F(V), \sigma_{F(V)})
\end{equation*}
for $\mathbf{V} = (V, \sigma_V) \in \mathcal{Z}(\mathcal{C})$, where $\sigma_{F(V)}: F(V) \otimes (-) \to (-) \otimes F(V)$ is the half-braiding uniquely determined by the property that the equation
\begin{equation*}
  F_2(X, V) \circ \sigma_{F(V)}(F(X)) = F(\sigma_V(X)) \circ F_2(V, X)
\end{equation*}
holds for all $X \in \mathcal{C}$. By reformulating \cite[Theorem 4.1]{2016arXiv160805905S} in our notation, we obtain the following theorem:

\begin{theorem}
  \label{thm:nat-tensor-bij}
  Let $F$ and $G$ be tensor autoequivalences of $\mathcal{C}$. Given an invertible object $\boldsymbol{\beta} = (\beta, \tau)$ of $\mathcal{Z}(F, G)$, we define $\Theta(\boldsymbol{\beta}): \widetilde{F} \to \widetilde{G}$ by
  \begin{equation*}
    \renewcommand{\xarrlen}{5em}
    \begin{aligned}
      \Theta(\boldsymbol{\beta})_{\mathbf{V}}
      = \Big( F(V)
      & \xarr{\id \otimes \coev_{\beta}}
      F(V) \otimes \beta \otimes \beta^*
        \xarr{\sigma_{F(V)}(\beta) \otimes \id}
      \beta \otimes F(V) \otimes \beta^* \\
      & \xarr{\tau_V \otimes \id} G(V) \otimes \beta \otimes \beta^*
      \xarr{\id \otimes \coev_{\beta}^{-1}} G(V) \Big)
    \end{aligned}
  \end{equation*}
  for $\mathbf{V} = (V, \sigma_V) \in \mathcal{Z}(\mathcal{C})$ with $\widetilde{F}(\mathbf{V}) = (F(V), \sigma_{F(V)})$. Then the map
  \begin{equation*}
    \Theta: \Inv(\mathcal{Z}(F, G)) \to \Nat_{\otimes}(\widetilde{F}, \widetilde{G}),
    \quad [\boldsymbol{\beta}] \mapsto \Theta(\boldsymbol{\beta})
  \end{equation*}
  is a well-defined bijection.
\end{theorem}

Let $\boldsymbol{\beta} = (\beta, \tau)$ be an invertible object of $\mathcal{Z}(F, G)$. Since $\beta$ is an invertible object of $\mathcal{C}$, the natural transformation $\Theta(\boldsymbol{\beta})$ can also be defined by
\begin{equation}
  \label{eq:Theta-beta-another-def}
  \renewcommand{\xarrlen}{4em}
  \Theta(\boldsymbol{\beta})_{\mathbf{V}} \otimes \id_{\beta}
  = \Big( F(V) \otimes \beta
  \xarr{\sigma_{F(V)}(\beta)}
  \beta \otimes F(V)
  \xarr{\tau_V}
  G(B) \otimes \beta
  \Big)
\end{equation}
for $\mathbf{V} = (V, \sigma_V) \in \mathcal{Z}(\mathcal{C})$ with $\widetilde{F}(\mathbf{V}) = (F(V), \sigma_{F(V)})$.

\begin{proof}
  We denote by $\mathscr{J}(F, G)$ the class of all pairs $(\beta, j)$ consisting of an invertible object $\beta \in \mathcal{C}$ and a monoidal natural transformation
  \begin{equation*}
    j_X: \beta \otimes F(X) \otimes \beta^* \to G(X) \quad (X \in \mathcal{C}).
  \end{equation*}
  Given an invertible object $\boldsymbol{\beta} = (\beta, \tau) \in \mathcal{Z}(F, G)$, we define $j^{\boldsymbol{\beta}}$ by
  \begin{equation*}
    j^{\boldsymbol{\beta}}_X
    = \Big( \beta \otimes F(X) \otimes \beta^*
    \xarq{\tau_{X} \otimes \beta^*}
    G(X) \otimes \beta \otimes \beta^*
    \xarq{\id \otimes \coev_{\beta}^{-1}} G(X) \Big)
  \end{equation*}
  for $X \in \mathcal{C}$. The map $(\beta, \tau) \mapsto (\beta, j^{\boldsymbol{\beta}})$ gives a bijection between the class of invertible objects of $\mathcal{Z}(F, G)$ and the class $\mathscr{J}(F, G)$. Moreover, this bijection induces a bijection between the set $\Inv(\mathcal{Z}(F, G))$ and the set $J(F, G) := \mathscr{J}(F, G) / \mathord{\sim}$, where $\mathord{\sim}$ is the equivalence relation on $\mathscr{J}(F, G)$ introduced in \cite{2016arXiv160805905S}. We have constructed a bijection between $J(F, G)$ and $\Nat_{\otimes}(\widetilde{F}, \widetilde{G})$ \cite[Theorem 4.1]{2016arXiv160805905S}. We obtain the bijection $\Theta$ by composing these two bijections.
\end{proof}

Let $S_{\mathcal{C}}$ and $S_{\mathcal{Z}(\mathcal{C})}$ be the left duality functors of $\mathcal{C}$ and $\mathcal{Z}(\mathcal{C})$, respectively. Note that $S_{\mathcal{C}}^2$ is an tensor autoequivalence of $\mathcal{C}$ inducing $S_{\mathcal{Z}(\mathcal{C})}^2$. Now the main result of this paper is stated as follows:

\begin{theorem}
  \label{thm:classification}
  Let $\mathcal{C}$ be a finite tensor category, and let $\modobjb_{\mathcal{C}} \in \mathcal{Z}(\id_{\mathcal{C}}, S_{\mathcal{C}}^4)$ be the Radford object of $\mathcal{C}$. Then the bijection
  \begin{equation*}
    \Theta: \Inv(\mathcal{Z}(\id_{\mathcal{C}}^{}, S_{\mathcal{C}}^2))
    \to \Nat_{\otimes}(\id_{\mathcal{Z}(\mathcal{C})}^{}, S_{\mathcal{Z}(\mathcal{C})}^2)
  \end{equation*}
  given in Theorem~\ref{thm:nat-tensor-bij} restricts to a bijection between the set
  \begin{equation*}
    \Big\{ [\boldsymbol{\beta}] \in \Inv(\mathcal{Z}(\id_{\mathcal{C}}^{}, S_{\mathcal{C}}^2))
    \,\Big|\, S_{\mathcal{C}}^2(\boldsymbol{\beta}) \otimes \boldsymbol{\beta} \cong \modobjb_{\mathcal{C}} \Big\}
  \end{equation*}
  of `square roots' of $\modobjb_{\mathcal{C}}$ and the set of ribbon pivotal structures of $\mathcal{Z}(\mathcal{C})$.
\end{theorem}

The proof of this theorem is given in Subsection~\ref{subsec:proof-classification}.

This theorem means that every ribbon pivotal structure of $\mathcal{Z}(\mathcal{C})$ is obtained from an invertible object $\boldsymbol{\beta} = (\beta, \tau)$ of $\mathcal{Z}(\id_{\mathcal{C}}^{}, S_{\mathcal{C}}^{2})$ such that there is an isomorphism $f : \beta^{**} \otimes \beta \to \modobj_{\mathcal{C}}$ in $\mathcal{C}$ making the diagram
\begin{equation*}
  \begin{tikzcd}[column sep = 5em, row sep = 2em]
    \beta^{**} \otimes \beta \otimes X
    \arrow[d, "f \otimes \id"']
    \arrow[r, "\id \otimes \tau_X"]
    & \beta^{**} \otimes X^{**} \otimes \beta
    \arrow[r, "\tau_X^{**} \otimes \id"]
    & X^{****} \otimes \beta^{**} \otimes \beta
    \arrow[d, "\id \otimes f"] \\
    \modobj_{\mathcal{C}} \otimes X
    \arrow[rr, "\delta_X"]
    & & X^{****} \otimes \modobj_{\mathcal{C}}
  \end{tikzcd}
\end{equation*}
commutes for all objects $X \in \mathcal{C}$. Given such an object $\boldsymbol{\beta} = (\beta, \tau)$, the corresponding ribbon pivotal structure $j$ of $\mathcal{Z}(\mathcal{C})$ is characterized by
\begin{equation*}
  j_{\mathbf{X}} \otimes \id_{\beta}
  = \Big( X \otimes \beta
    \xrightarrow{\quad \sigma_X(\beta) \quad}
    \beta \otimes X
    \xrightarrow{\quad \tau_X \quad}
    X^{**} \otimes \beta
  \Big)
\end{equation*}
for $\mathbf{X} = (X, \sigma_X) \in \mathcal{Z}(\mathcal{C})$. By Lemma~\ref{lem:ribbon-pivot}, the corresponding twist is obtained by composing $j$ with the inverse of the Drinfeld isomorphism.

\subsection{Relative modular object of central tensor functors}

As a first step of the proof of Theorem~\ref{thm:classification}, we give a formula of the relative modular object of a central tensor functor, that is, a tensor functor from a braided finite tensor category that factors through the Drinfeld center. The following technical lemma will be needed:

\begin{lemma}
  \label{lem:adjoint-and-braiding}
  Let $\mathcal{M}$ and $\mathcal{N}$ be bimodule categories over a monoidal category $\mathcal{B}$, and let $R : \mathcal{M} \to \mathcal{N}$ be a lax $\mathcal{B}$-bimodule functor with structure morphisms
  \begin{equation*}
    \xi^{(\ell)}_{X,M} : X \catactl R(M) \to R(X \catactl M)
    \quad \text{and} \quad
    \xi^{(r)}_{M,X} : R(M) \catactr X \to R(M \catactr X).
  \end{equation*}
  Suppose that $R$ has a left adjoint $L$. By Lemma~\ref{lem:mod-func-adj}, the functor $L$ carries a structure of an oplax $\mathcal{B}$-bimodule functor, which we denote by
  \begin{equation*}
    \zeta^{(\ell)}_{X,N} : L(X \catactl N) \to X \catactl L(N)
    \quad \text{and} \quad
    \zeta^{(r)}_{N,X} : L(N \catactr X) \to L(N) \catactr X.
  \end{equation*}
  Suppose moreover that we are given two natural transformations
  \begin{gather*}
    \tau^{\mathcal{M}}_{X, M} : X \catactl M \to M \catactr X \quad (X \in \mathcal{C}, M \in \mathcal{M}), \\
    \tau^{\mathcal{N}}_{X, N} : X \catactl N \to N \catactr X \quad (X \in \mathcal{C}, N \in \mathcal{N}).
  \end{gather*}
  Then the following two assertions are equivalent:
  \begin{enumerate}
  \item For all $X \in \mathcal{C}$ and $M \in \mathcal{M}$, the following equation holds:
    \begin{equation*}
      R(\tau^{\mathcal{M}}_{X,M}) \circ \xi^{(\ell)}_{X,M} = \xi^{(r)}_{M,X} \circ \tau^{\mathcal{N}}_{X,R(M)}.
    \end{equation*}
  \item For all $X \in \mathcal{C}$ and $N \in \mathcal{N}$, the following equation holds:
    \begin{equation*}
       \tau^{\mathcal{M}}_{X,L(N)} \circ \zeta^{(\ell)}_{X,N} =  \zeta^{(r)}_{N,X} \circ L(\tau^{\mathcal{N}}_{X,N}).
    \end{equation*}
  \end{enumerate}
\end{lemma}
\begin{proof}
  We only prove that (1) implies (2), since the other direction can be proved in a similar manner.
  Suppose that (1) holds.
  Let $\eta$ and $\varepsilon$ be the unit and the counit of the adjunction $L \dashv R$.
  We fix objects $X \in \mathcal{B}$ and $N \in \mathcal{N}$ and consider the following diagram:
  \begin{equation*}
    \begin{tikzcd}[column sep = 1em]
      L(X \catactl N)
      \arrow[r, "L(\id \catactl \eta)"]
      \arrow[d, "L(\tau^{\mathcal{N}})"']
      & [3em] L(X \catactl R L(N))
      \arrow[r, "L(\xi^{(\ell)})"]
      \arrow[d, "L(\tau^{\mathcal{N}})"']
      & [3em] L R (X \catactl L(N))
      \arrow[r, "\varepsilon"]
      \arrow[d, "L R(\tau^{\mathcal{M}})"]
      & X \catactl L(N)
      \arrow[d, "\tau^{\mathcal{M}}"] \\
      L(N \catactr X)
      \arrow[r, "L(\eta \catactr \id)"]
      & L(R L(N) \catactr X)
      \arrow[r, "L(\xi^{(r)})"]
      & L R (L(N) \catactr X))
      \arrow[r, "\varepsilon"]
      & L(N) \catactr X
    \end{tikzcd}    
  \end{equation*}
  The central square is commutative by the assumption. The left and the right squares are also commutative by the naturality of $\tau^{\mathcal{N}}$ and $\varepsilon$, respectively. Thus the above diagram commutes. This implies that (2) holds. The proof is done.
\end{proof}

Let $\mathcal{B}$ be a braided finite tensor category with braiding $\sigma$, and let $\mathcal{C}$ be a (not necessarily braided) finite tensor category. We say that a tensor functor $F: \mathcal{B} \to \mathcal{C}$ is {\em central} \cite[Definition 2.3]{MR3022755} if there is a braided tensor functor $\widetilde{F}: \mathcal{B} \to \mathcal{Z}(\mathcal{C})$, called a {\em lift}, such that $F = U \circ \widetilde{F}$, where $U: \mathcal{Z}(\mathcal{C}) \to \mathcal{C}$ is the forgetful functor. We now assume that $F$ is central and fix a lift $\widetilde{F}$. Then we can write
\begin{equation*}
  \widetilde{F}(X) = \Big( F(X), \ \Sigma_X: F(X) \otimes (-) \to (-) \otimes F(X)\Big) \in \mathcal{Z}(\mathcal{C})
\end{equation*}
for some natural isomorphism $\Sigma_X$.
The assumption that $\widetilde{F} : \mathcal{B} \to \mathcal{Z}(\mathcal{C})$ is a braided tensor functor means that the following equation holds:
\begin{equation}
  \label{eq:lift-braided}
  F_2(V,Y) \circ \Sigma_{X}(F(Y)) = F(\sigma_{X,Y}) \circ F_2(X,Y)
  \quad (X, Y \in \mathcal{B}).
\end{equation}

\begin{theorem}
  \label{thm:rel-mod-obj-central}
  Notations are as above. If the central tensor functor $F: \mathcal{B} \to \mathcal{C}$ is perfect, then the relative modular object of $F$ is given by $\relmodb_F = (\mu, \gamma)$, where
  \begin{equation}
    \label{eq:rel-mod-obj-central-1}
    \mu = \modobj_{\mathcal{C}}^* \otimes F(\modobj_{\mathcal{B}}^{})
  \end{equation}
  and the natural transformation $\gamma$ is given by
  \begin{equation}
    \label{eq:rel-mod-obj-central-2}
    \gamma_X: \mu \otimes F(X) \xrightarrow{\quad \Sigma_{X}(\mu)^{-1} \quad} F(X) \otimes \mu
    \quad (X \in \mathcal{B}).
  \end{equation}
\end{theorem}
\begin{proof}
  By \eqref{eq:lift-braided}, we can apply (2) $\Rightarrow$ (1) of Lemma~\ref{lem:adjoint-and-braiding} to $\mathcal{M} = {}_{\langle F \rangle}\mathcal{C}{}_{\langle F \rangle}$, $\mathcal{N} = \mathcal{B}$, $R = F^{\radj}$, $L = F$, $\tau^{\mathcal{M}} = \Sigma$ and $\tau^{\mathcal{N}} = \sigma$. As a consequence, we have
  \begin{equation*}
    F^{\radj}(\Sigma_{X,V}) \circ \xi^{(\ell)}_{X,V} = \xi^{(r)}_{V,X} \circ \sigma_{X, F^{\radj}(V)}
  \end{equation*}
  for all objects $X \in \mathcal{B}$ and $V \in \mathcal{C}$, where $\xi^{(\ell)}$ and $\xi^{(r)}$ are the structure morphism of $F^{\radj}$ as a lax $\mathcal{B}$-bimodule functor. By inverting its structure morphisms, we regard $F^{\radj}$ as an oplax $\mathcal{B}$-bimodule functor. Then, again by (2) $\Rightarrow$ (1) of Lemma~\ref{lem:adjoint-and-braiding}, we see that the equation
  \begin{equation*}
    \omega^{(r)}_{Y,X} \circ \Sigma_{X}(F^{\rradj}(Y)) = F^{\rradj}(\sigma_{X,Y}) \circ \omega^{(\ell)}_{X,Y}
  \end{equation*}
  holds for all objects $X, Y \in \mathcal{B}$, where $\omega^{(\ell)}$ and $\omega^{(r)}$ are the structure morphisms of the lax $\mathcal{B}$-bimodule functor $F^{\rradj}$. Thus, by the definition $\relmodb_F = (\mu, \gamma)$, we have
  \begin{gather*}
    \gamma_X
    = (\omega_{X, \unitobj}^{(\ell)})^{-1} \circ \omega_{\unitobj, X}^{(r)}
    = (\omega_{X, \unitobj}^{(\ell)})^{-1} \circ F^{\rradj}(\sigma_{X,\unitobj})^{-1} \circ \omega_{\unitobj,X}^{(r)}
    = \Sigma_X(\mu)^{-1}
  \end{gather*}
  for all $X \in \mathcal{C}$. By Theorem~\ref{thm:relative-modulus}, we may assume that $\mu$ is given by~\eqref{eq:rel-mod-obj-central-1}. Then the above formula of $\gamma_X$ coincides with \eqref{eq:rel-mod-obj-central-2}. The proof is done.
\end{proof}

The forgetful functor $U: \mathcal{Z}(\mathcal{C}) \to \mathcal{C}$ is a perfect tensor functor \cite[Corollary 4.9]{MR3632104}. We note that $\mathcal{Z}(\mathcal{C})$ is {\em unimodular}, that is, its distinguished invertible object $\modobj_{\mathcal{Z}(\mathcal{C})}$ is isomorphic to the unit object \cite[Proposition 8.10.10]{MR3242743}. Applying the above theorem to $U$, we obtain:

\begin{corollary}
  \label{cor:rel-mod-obj-Dri}
  The relative modular object of $U$ is given by $\relmodb_U = (\modobj^*, \gamma)$, where $\modobj$ is the distinguished invertible object of $\mathcal{C}$ and $\gamma$ is given by
  \begin{equation*}
    \gamma_{\mathbf{X}} = \Big( \modobj^* \otimes U(\mathbf{X})
    = \modobj^* \otimes X
    \xrightarrow{\quad \sigma_X(\modobj^*)^{-1} \quad} 
    X \otimes \modobj^*
    = U(\mathbf{X}) \otimes \modobj^* \Big)
  \end{equation*}
  for $\mathbf{X} = (X, \sigma_X) \in \mathcal{Z}(\mathcal{C})$.
\end{corollary}

\subsection{The Radford isomorphism of the Drinfeld center}

Let $\mathcal{C}$ be a finite tensor category, and let $U: \mathcal{Z}(\mathcal{C}) \to \mathcal{C}$ be the forgetful functor.
Theorem~\ref{thm:relative-modulus} and Corollary~\ref{cor:rel-mod-obj-Dri} give two different descriptions of the relative modular object of $U$. Using these results, we can determine the Radford isomorphism of $\mathcal{Z}(\mathcal{C})$ as follows:

\begin{theorem}
  \label{thm:Rad-obj-Dri-cen}
  $\modobjb_{\mathcal{Z}(\mathcal{C})} = (\unitobj_{\mathcal{Z}(\mathcal{C})}, \Theta(\modobjb_{\mathcal{C}}))$.
\end{theorem}
\begin{proof}
  Since $\mathcal{Z}(\mathcal{C})$ is unimodular, there is a natural isomorphism
  $\Delta_{\mathbf{X}}: \mathbf{X} \to \mathbf{X}^{****}$ ($\mathbf{X} \in \mathcal{Z}(\mathcal{C})$)
  such that $\modobjb_{\mathcal{Z}(\mathcal{C})} = (\unitobj_{\mathcal{Z}(\mathcal{C})}, \Delta)$. By~\eqref{eq:Theta-beta-another-def}, the claim of this theorem is equivalent to that the natural isomorphism $\Delta$ satisfies
  \begin{equation}
    \label{eq:Radford-iso-Dri}
    \Delta_{\mathbf{X}} \otimes \id_{\modobj}
    = \Big( X \otimes \modobj
    \xarq{\sigma_{X}(\modobj)} \modobj \otimes X
    \xarq{\delta_X} X^{****} \otimes \modobj \Big)
  \end{equation}
  for $\mathbf{X} = (X, \sigma_X) \in \mathcal{Z}(\mathcal{C})$, where $\modobjb_{\mathcal{C}} = (\modobj, \delta)$.
  By Theorem~\ref{thm:relative-modulus} and Corollary~\ref{cor:rel-mod-obj-Dri}, there is an isomorphism $p : \unitobj \to D \otimes D^*$ in $\mathcal{C}$ such that the diagram
  \begin{equation*}
    \begin{tikzcd}
      \unitobj \otimes X
      \arrow[rr, "\Delta_X"]
      \arrow[d, "p \otimes \id"']
      & [5em] & [2em]
      X^{****} \otimes \unitobj
      \arrow[d, "\id \otimes p"] \\
      \modobj \otimes \modobj^* \otimes X
      \arrow[r, "\id \otimes \sigma_{X}(D^*)^{-1}"]
      & \modobj \otimes X \otimes \modobj^*
      \arrow[r, "\delta_X \otimes \id"]
      & X^{****} \otimes \modobj \otimes \modobj^*
    \end{tikzcd}
  \end{equation*}
  commutes for $\mathbf{X} = (X, \sigma_X) \in \mathcal{Z}(\mathcal{C})$. Since $\Hom_{\mathcal{C}}(\unitobj, \modobj \otimes \modobj^*)$ is a one-dimensional space spanned by $\coev_{\modobj}$, we may assume that $p = \coev_{\modobj}$. Now \eqref{eq:Radford-iso-Dri} follows from the above commutative diagram and the general formula
  \begin{equation*}
    \sigma_{X}(V^*)^{-1} =
    \begin{gathered}[t]
      (\eval_{V} \otimes \id_{X} \otimes \id_{V^*}) \circ (\id_{V^*} \otimes \sigma_X(V) \otimes \id_{V^*}) \\
      \circ (\id_{V^*} \otimes \id_{X} \otimes \coev_V)
    \end{gathered}
  \end{equation*}
  for $\mathbf{X} = (X, \sigma_X) \in \mathcal{Z}(\mathcal{C})$ and $V \in \mathcal{C}$. The proof is done.
\end{proof}

\subsection{Proof of Theorem~\ref{thm:classification}}
\label{subsec:proof-classification}

Let $\mathcal{B}$ be a braided finite tensor category, and let $(\modobj, \delta)$ be the Radford object $\mathcal{B}$. For $X \in \mathcal{B}$, we define $w_X: X \to X^{****}$ by the right-hand side of~\eqref{eq:w-iso}. Then the equation
\begin{equation*}
  \renewcommand{\xarrlen}{4em}
  \delta_X = \Big(X \otimes \modobj
  \xarr{\sigma_{X,\modobj}} \modobj \otimes X
  \xarr{\id \otimes w_X} \modobj \otimes X^{****} \Big)
\end{equation*}
holds for all $X \in \mathcal{B}$ \cite[Theorem 8.10.7]{MR3242743}. Since the object $\modobj$ is invertible, the above formula of the Radford isomorphism implies:

\begin{lemma}
  \label{lem:ribbon-pivot-finite}
  A pivotal structure $j$ of $\mathcal{B}$ is ribbon if and only if the equation
  \begin{equation*}
    \delta_X = (\id_{\modobj} \otimes j_{X}^{**}j_X^{\phantom{*}}) \circ \sigma_{X, \modobj}^{\phantom{*}}
  \end{equation*}
  holds for all objects $X \in \mathcal{B}$. 
\end{lemma}

\begin{proof}[Proof of Theorem~\ref{thm:classification}]
  Suppose that we have three tensor autoequivalences $F$, $G$ and $H$ of $\mathcal{C}$.
  It is straightforward to show that the equation
  \begin{equation}
    \label{eq:Theta-functorial}
    \Theta(\boldsymbol{\beta}_1 \otimes \boldsymbol{\beta}_2)_{\mathbf{X}}^{}
    = \Theta(\boldsymbol{\beta}_1)_{\mathbf{X}}^{} \circ \Theta(\boldsymbol{\beta}_2)_{\mathbf{X}}^{}
  \end{equation}
  holds for all invertible objects $\boldsymbol{\beta}_1 \in \mathcal{Z}(G, H)$ and $\boldsymbol{\beta}_2 \in \mathcal{Z}(F, G)$.

  Now let $\boldsymbol{\beta} \in \mathcal{Z}(\id_{\mathcal{C}}^{}, S_{\mathcal{C}}^2)$ be an invertible object.
  We apply the above lemma to $\mathcal{B} = \mathcal{Z}(\mathcal{C})$ and $j = \Theta(\boldsymbol{\beta})$. 
  By Theorem~\ref{thm:Rad-obj-Dri-cen} and the fact that $\mathcal{Z}(\mathcal{C})$ is unimodular, the pivotal structure $\Theta(\boldsymbol{\beta})$ of $\mathcal{Z}(\mathcal{C})$ is ribbon if and only if the equation
  \begin{equation*}
    \Theta(\boldsymbol{\beta})_{\mathbf{X}}^{**} \circ \Theta(\boldsymbol{\beta})_{\mathbf{X}}^{} = \Theta(\modobjb_{\mathcal{C}})_{\mathbf{X}}^{}
  \end{equation*}
  holds for all $\mathbf{X} \in \mathcal{Z}(\mathcal{C})$. By \eqref{eq:Theta-functorial}, this is equivalent to the equation
  \begin{equation*}
    \Theta(S_{\mathcal{C}}^2(\boldsymbol{\beta}) \otimes \boldsymbol{\beta}) = \Theta(\modobjb_{\mathcal{C}})
  \end{equation*}
  in $\Nat(\id_{\mathcal{Z}(\mathcal{C})}, S_{\mathcal{Z}(\mathcal{C})}^4)$. Now the claim follows from the bijectivity of $\Theta$.
\end{proof}

\subsection{Applications to non-semisimple modular tensor categories}

We give applications of Theorem~\ref{thm:classification} to not necessarily semisimple modular tensor categories in the sense of Lyubashenko \cite{MR1862634}. We briefly recall the definition: If $\mathcal{B}$ is a braided finite tensor category, then the coend $\mathcal{L} = \int^{X \in \mathcal{B}} X \otimes X^*$ has a natural structure of a Hopf algebra in $\mathcal{B}$. The Hopf algebra $\mathcal{L}$ has a canonical Hopf paring $\omega: \mathcal{L} \otimes \mathcal{L} \to \unitobj$. We say that $\mathcal{B}$ is {\em non-degenerate} if the pairing $\omega$ is. A {\em modular tensor category} is a non-degenerate braided finite tensor category equipped with a ribbon structure.

Now let $\mathcal{C}$ be a finite tensor category. We note that the Drinfeld center $\mathcal{Z}(\mathcal{C})$ is always non-degenerate by \cite{MR3996323} and \cite[Proposition 8.6.3]{MR3242743}, but it does not have a ribbon structure in general. Our theorem implies:

\begin{theorem}
  $\mathcal{Z}(\mathcal{C})$ is a modular tensor category if and only if
  \begin{equation*}
    \Big\{ [\boldsymbol{\beta}] \in \Inv(\mathcal{Z}(\id_{\mathcal{C}}^{}, S_{\mathcal{C}}^2))
    \,\Big|\, S_{\mathcal{C}}^2(\boldsymbol{\beta}) \otimes \boldsymbol{\beta} \cong \modobjb_{\mathcal{C}} \Big\} \ne \emptyset.
  \end{equation*}
\end{theorem}

A spherical pivotal structure \cite{MR1686423} is a pivotal structure such that the associated left trace and the right trace coincide. Although spherical fusion categories are an important class of tensor categories, such a kind of trace condition is often meaningless in the non-semisimple setting. From the viewpoint of topological quantum field theory, Douglas, Schommer-Pries and Snyder \cite[Definition 4.5.2]{2013arXiv1312.7188D} introduced an alternative notion of the {\em sphericity} of finite tensor categories. In our notation, a {\em spherical finite tensor category} in the sense of \cite{2013arXiv1312.7188D} is just a finite tensor category $\mathcal{C}$ equipped with a pivotal structure $j$ such that
\begin{equation*}
  S_{\mathcal{C}}^2(\unitobj, j) \otimes (\unitobj, j) \cong \modobjb_{\mathcal{C}}
\end{equation*}
in $\mathcal{Z}(\id_{\mathcal{C}}, S_{\mathcal{C}}^4)$. By the above theorem, we have:

\begin{theorem}
  If $\mathcal{C}$ is a spherical finite tensor category in the sense of \cite{2013arXiv1312.7188D}, then $\mathcal{Z}(\mathcal{C})$ is a modular tensor category.
\end{theorem}

This answers Problem (7) of \cite[Section 6]{MR2681261}.

\appendix
\section{Remarks on the Radford isomorphism}
\label{sec:def-rad-obj}

\newcommand{\deltaENO}{\delta^{\mathrm{ENO}}}

\subsection{Original definition}

Let $\mathcal{C}$ be a finite tensor category.
For simplicity, we write $X^{***} = X^{*3}$, $X^{****} = X^{*4}$, etc., for $X \in \mathcal{C}$. In Section~\ref{sec:relative-modulus}, we have introduced an invertible object $\modobj = \modobj_{\mathcal{C}}$ of $\mathcal{C}$ and a natural isomorphism
\begin{equation*}
  \delta_X: \modobj \otimes X \to X^{*4} \otimes \modobj \quad (X \in \mathcal{C}),
\end{equation*}
which we call the {\em Radford isomorphism}. The first aim of this appendix is to show that $\delta_X$ is identical to the isomorphism introduced in \cite{MR2097289}.

We first recall the definition given in \cite{MR2097289}. We make $\mathcal{C}$ as a finite left $\mathcal{C}^{\env}$-module category by $(X \boxtimes Y) \catactl V = X \otimes V \otimes Y$. Let $\iHom$ be the internal Hom functor of the $\mathcal{C}^{\env}$-module category $\mathcal{C}$. For the algebra $A := \iHom(\unitobj, \unitobj)$ in $\mathcal{C}^{\env}$, there is an equivalence
\begin{equation*}
  \Kappa: \mathcal{C} \to (\mathcal{C}^{\env})_A,
  \quad V \mapsto \iHom(\unitobj, V)
  \mathop{\cong}^{\eqref{eq:int-Hom-iso-1}} (V \boxtimes \unitobj) \otimes A
\end{equation*}
of left $\mathcal{C}^{\env}$-module categories.
The {\em distinguished invertible object} of $\mathcal{C}$ \cite[Definition 3.1]{MR2097289} is the object corresponding to the right $A$-module $A^*$ via the above equivalence. Let $\modobj$ be the distinguished invertible object. By definition, there is an isomorphism
$\phi: (\modobj \boxtimes \unitobj) \otimes A \to A^*$
of right $A$-modules in $\mathcal{C}^{\env}$. This induces an isomorphism
\begin{equation*}
  \psi := \Big( A^{**} \xarq{\phi^*}
  A^* \otimes (\modobj^* \boxtimes \unitobj)
  \xarq{\phi^{-1} \otimes \id}
  (\modobj \boxtimes \unitobj) \otimes A \otimes (\modobj \boxtimes \unitobj)^* \Big)
\end{equation*}
of algebras in $\mathcal{C}^{\env}$. There is a natural isomorphism
\begin{align*}
  \tau_X := \Big( (\unitobj \boxtimes X) \otimes A
  \xrightarrow{\ \eqref{eq:int-Hom-iso-1} \ }
  & \iHom(\unitobj, (\unitobj \boxtimes X) \catactl \unitobj) \\[-7pt]
  =
  & \iHom(\unitobj, (X \boxtimes \unitobj) \catactl \unitobj)
  \xrightarrow{\ \eqref{eq:int-Hom-iso-1} \ }
  (X \boxtimes \unitobj) \otimes A \Big)
\end{align*}
obtained from the structure of $\iHom(\unitobj, -)$ as a $\mathcal{C}^{\env}$-module functor. We now define the natural isomorphism
\begin{equation*}
  \tilde{\tau}_X:
  (\unitobj \boxtimes X) \otimes (\modobj \boxtimes \unitobj) \otimes A
  \to (X^{*4} \boxtimes \unitobj) \otimes (\modobj \boxtimes \unitobj) \otimes A
  \quad (X \in \mathcal{C})
\end{equation*}
to be the unique morphism such that the diagram
\begin{equation*}
  \begin{tikzcd}
    (\unitobj \boxtimes X) \otimes A^{**}
    \arrow[r, "\id \otimes \varphi"]
    \arrow[d, "(\tau_{X^{**}})^{**}"]
    & (\unitobj \boxtimes X) \otimes (\modobj \boxtimes \unitobj) \otimes A \otimes (\modobj^* \boxtimes \unitobj)
    \arrow[d, "\tilde{\tau}_X \otimes \id"]
    \\
    (X^{*4} \boxtimes \unitobj) \otimes A^{**}
    \arrow[r, "\id \otimes \varphi"]
    & 
    (X^{**} \boxtimes \unitobj) \otimes (\modobj \boxtimes \unitobj) \otimes A \otimes (\modobj^* \boxtimes \unitobj)
  \end{tikzcd}
\end{equation*}
commutes.

\begin{definition}
  We define $\deltaENO_X: \modobj \otimes X \to X^{*4} \otimes \modobj$ ($X \in \mathcal{C}$) to be the natural isomorphism characterized by the following formula:
  \begin{equation*}
    \renewcommand{\xarrlen}{2.5em}
    \begin{aligned}
      \Kappa(\deltaENO_X) = \Big(
      \Kappa(\modobj \otimes X) \xarr{\eqref{eq:int-Hom-iso-1}}
      & \, (\modobj \boxtimes \unitobj) \otimes (X \boxtimes \unitobj) \otimes A \\
      \xarr{\id \otimes \tau_X^{-1}}
      & \, (\modobj \boxtimes \unitobj) \otimes (\unitobj \boxtimes X) \otimes A \\
      = & \, (\unitobj \boxtimes X) \otimes (\modobj \boxtimes \unitobj) \otimes A \\
      \xarr{\tilde{\tau}_X}
      & \, (X^{*4} \boxtimes \unitobj) \otimes (\modobj \boxtimes \unitobj) \otimes A
      \xarr{\eqref{eq:int-Hom-iso-1}} \Kappa(X^{*4} \otimes \modobj) \Big).
    \end{aligned}
  \end{equation*}
\end{definition}

Noting that the isomorphisms $\tau_X$ and $\tilde{\tau}_X$ are the inverses of the natural isomorphisms written as $\rho_X$ and $\tilde{\rho}_X$ in \cite{MR2097289}, respectively, one can check that the isomorphism $\modobj \otimes {}^{**}\!X \otimes \modobj^* \cong X^{**}$ of tensor functors given in \cite[Theorem 3.3]{MR2097289} corresponds to $\deltaENO_{{}^{**}\!X}$ via the canonical isomorphism
\begin{equation*}
  \Hom_{\mathcal{C}}(\modobj \otimes {}^{**}\!X, X^{**} \otimes \modobj)
  \cong \Hom_{\mathcal{C}}(\modobj \otimes {}^{**}\!X \otimes \modobj^*, X^{**}).
\end{equation*}

\subsection{The algebra $A$ as a coend}
\label{append:subsec:canonical-algebra}

Instead of the algebra $A = \iHom(\unitobj, \unitobj)$, the canonical algebra $A' := \int^{X \in \mathcal{C}} X \boxtimes {}^* \! X$ is used in this paper.
According to \cite{MR3632104}, we can identify the algebra $A$ with $A'$.
We explain how to do so:
Let $\din_X: X \boxtimes {}^* \! X \to A'$ be the universal dinatural transformation of the coend $A'$. We define the morphism $\epsilon: A' \catactl \unitobj \to \unitobj$ in $\mathcal{C}$ to be the unique morphism in $\mathcal{C}$ such that the equation
\begin{equation*}
  \epsilon \circ (\din_X \catactl \id_{\unitobj}) = \eval_{{}^*\!X}
\end{equation*}
holds for all objects $X \in \mathcal{C}$. For $M \in \mathcal{C}^{\env}$ and $V \in \mathcal{C}$, we have the map
\begin{equation}
  \label{eq:coend-adj-1}
  \Hom_{\mathcal{C}^{\env}}(M, (V \boxtimes \unitobj) \otimes A')
  \to \Hom_{\mathcal{C}}(M \catactl \unitobj, V)
\end{equation}
sending a morphism $f: M \to (V \boxtimes \unitobj) \otimes A'$ in $\mathcal{C}^{\env}$ to
\begin{equation}
  \label{eq:coend-adj-2}
  \renewcommand{\xarrlen}{4em}
  M \catactl \unitobj \xarr{f \catactl \id_{\unitobj}}
  ((V \boxtimes \unitobj) \otimes A') \catactl \unitobj
  = V \otimes (A' \catactl \unitobj)
  \xarr{\id_V \otimes \epsilon} V.
\end{equation}
By the discussion of \cite[Subsection 4.3]{MR3632104}, the map~\eqref{eq:coend-adj-1} is bijective. Thus, by the definition of the internal Hom functor, we may identify
\begin{equation}
  \label{eq:int-Hom-coend}
  \iHom(\unitobj, V) = (V \boxtimes \unitobj) \otimes A'
\end{equation}
for all $V \in \mathcal{C}$. It is moreover shown in \cite{MR3632104} that $A$ is identified with $A'$ as an algebra if we endow the coend $A'$ with an algebra structure as in Subsection~\ref{subsec:center-as-modules} of this paper.

In what follows, we identify $\iHom(\unitobj, V)$ with $(V \boxtimes \unitobj) \otimes A'$ as above.
Since the isomorphism $\deltaENO$ is defined in terms of the internal Hom functor, it is convenient if morphisms related to $\iHom$ are written in terms of the universal dinatural transformation $\din_X$. We first express the unit and the counit of the adjunction \eqref{eq:coend-adj-1} as follows:

\begin{lemma}
  The unit of the adjunction \eqref{eq:coend-adj-1}, which we denote by
  \begin{equation*}
    \eta_M: M \to \iHom(\unitobj, M \catactl \unitobj)
    \quad (M \in \mathcal{C}^{\env}),    
  \end{equation*}
  is a unique natural transformation such that
  \begin{equation*}
    \renewcommand{\xarrlen}{5em}
    \begin{aligned}
      \eta_{X \boxtimes Y}
      = \Big( X \boxtimes Y
      \xarr{(\id \otimes \coev) \boxtimes \id} & \,
      (X \otimes Y \otimes Y^*) \boxtimes Y
      = ((X \otimes Y) \boxtimes \unitobj) \otimes (Y^* \boxtimes Y) \\
      \xarr{\id \otimes i_{Y^*}} & \, ((X \otimes Y) \boxtimes \unitobj) \otimes A
      = \iHom(\unitobj, (X \boxtimes Y) \catactl \unitobj) \Big)
    \end{aligned}
  \end{equation*}
  for all objects $X, Y \in \mathcal{C}$. The counit of \eqref{eq:coend-adj-1}, which we denote by
  \begin{equation*}
    \varepsilon_V: \iHom(\unitobj, V) \catactl \unitobj \to V \quad (V \in \mathcal{C}),
  \end{equation*}
  is given by
  \begin{equation*}
    \renewcommand{\xarrlen}{5em}
    \varepsilon_V = \Big(
    \iHom(\unitobj, V) \catactl \unitobj
    = V \otimes (A \catactl \unitobj)
    \xarr{\id_V \otimes \epsilon}
    V \otimes \unitobj = V \Big).
  \end{equation*}
\end{lemma}
\begin{proof}
  It is obvious from \eqref{eq:coend-adj-2} that the counit is given as stated. We note that the unit of the adjunction is the morphism corresponding to the identity via
  \begin{equation*}
    \Hom_{\mathcal{C}^{\env}}(M, \iHom(\unitobj, M \catactl \unitobj))
    \xrightarrow{\text{\quad \eqref{eq:coend-adj-1} \quad}}
    \Hom_{\mathcal{C}}(M \catactl \unitobj, M \catactl \unitobj).
  \end{equation*}
  Thus the description of the unit follows from the equation
  \begin{equation*}
    \id_{X \otimes Y} = \Big( X \otimes Y
    \xrightarrow{\quad \eta_{X \boxtimes Y} \catactl \unitobj \quad}
    X \otimes Y \otimes (A \catactl \unitobj)
    \xrightarrow{\quad \id \otimes \epsilon \quad} X \otimes Y \Big),
  \end{equation*}
  which is easily verified.
\end{proof}

By Lemma~\ref{lem:mod-func-adj}, the left $\mathcal{C}^{\env}$-module structure of $\iHom(\unitobj, -)$ is given by
\begin{equation*}
  \iHom(\unitobj, \id_M \catactl \varepsilon_V)
  \circ \eta_{M \otimes \iHom(\unitobj, V)}:
  M \otimes \iHom(\unitobj, V) \to \iHom(\unitobj, M \catactl V)
\end{equation*}
for $M \in \mathcal{C}^{\env}$ and $V \in \mathcal{C}$. The morphism $\tau_X$ for $X \in \mathcal{C}$ is the case where $V = \unitobj$ and $M = \unitobj \boxtimes X$. Thus, by a straightforward computation, we have:

\begin{lemma}
  \label{append:lem:tau}
  The morphism $\tau_X$ is characterized to be the unique morphism in $\mathcal{C}^{\env}$ such that the following equation holds:
  \begin{equation}
    \label{eq:append-tau}
    \tau_X (\id_{\unitobj \boxtimes X} \otimes \din_Y)
    = (\id_{X \boxtimes \unitobj} \otimes \din_{X^* \otimes Y})
    (\coev_X \boxtimes \id_{X}) \otimes \id_{Y \boxtimes {}^* Y}
    \quad (Y \in \mathcal{C}).
  \end{equation}
\end{lemma}

This means that $\tau_X$ is equal to the isomorphism
\begin{equation*}
  (\unitobj \boxtimes X) \otimes A
  = \int^{Y \in \mathcal{C}} Y \boxtimes {}^* (X^* \otimes Y)
  \cong
  \int^{Y \in \mathcal{C}} (X \otimes Y) \boxtimes {}^* Y
  = (X \boxtimes \unitobj) \otimes A
\end{equation*}
obtained by applying Lemma~\ref{lem:coend-adj} to the adjunction $X^* \otimes \id_{\mathcal{C}} \dashv X \otimes \id_{\mathcal{C}}$. By remarks succeeding that lemma, the inverse of $\tau_X$ is characterized by the equation
\begin{equation}
  \label{eq:append-tau-inverse}
  \tau_X^{-1} (\id_{X \boxtimes \unitobj} \otimes \din_Y)
  = (\id_{\unitobj \boxtimes X} \otimes \din_{X \otimes Y})
  (\id_{X} \boxtimes \coev_{{}^*\!X})
  \id_{Y \boxtimes {}^* Y}
  \quad (Y \in \mathcal{C}).
\end{equation}

\subsection{Proof of Lemma~\ref{lem:Radford-iso-ENO}}

We make $(\modobj \boxtimes \unitobj) \otimes A$ a left $A^{**}$-module by
\begin{equation}
  \label{append:eq:rho-tilde}
  \widetilde{\rho} = \Big( A^{**} \otimes (\modobj \boxtimes \unitobj) \otimes A
  \xrightarrow{\ \id \otimes \phi \ } A^{**} \otimes A^*
  \xrightarrow{\ m^{\ddag} \ } A^*
  \xrightarrow{\ \phi^{-1} \ } (\modobj \boxtimes \unitobj) \otimes A \Big)
\end{equation}
so that $\phi: (\modobj \boxtimes \unitobj) \otimes A \to A^*$ is an isomorphism of $A^{**}$-$A$-bimodules.
Since the functor $\Kappa: \mathcal{C} \to (\mathcal{C}^{\env})_A$ is an equivalence of left $\mathcal{C}^{\env}$-module categories, and since $\modobj$ is defined by $\Kappa(\modobj) \cong A^*$, the object $\modobj$ is an $A^{**}$-module in $\mathcal{C}$.
Let $\rho$ be the action of $A^{**}$ on $\modobj$. If we fix a quasi-inverse $\overline{\Kappa}$ of $\Kappa$, then $\rho$ is given by
\begin{equation*}
  \rho = \Big(A^{**} \catactl \modobj
  \xrightarrow{\ \cong\ } A^{**} \catactl \overline{\Kappa} \Kappa(\modobj)
  \xrightarrow{\ \cong\ } \overline{\Kappa}(A^{**} \otimes \Kappa(\modobj))
  \xrightarrow{\ \overline{\Kappa}(\widetilde{\rho}) \ } \overline{\Kappa}\Kappa(\modobj)
  \xrightarrow{\ \cong\ } \modobj \Big).
\end{equation*}
The Radford isomorphism $\delta$ is given by
\begin{equation*}
  \renewcommand{\xarrlen}{6em}
  \begin{aligned}
    \delta_X = \Big(\modobj \otimes X
    \xarr{\coev \otimes \id \otimes \id} \,
    & X^{*4} \otimes X^{*5} \otimes \modobj \otimes X \\[-3pt]
    = \, & X^{*4} \otimes ((X^{***} \boxtimes X^{**})^{**} \catactl \modobj) \\
    \xarr{\id \otimes (\din_{X^{***}}^{**} \catactl \modobj)} \,
    & X^{*4} \otimes (A^{**} \catactl \modobj)
    \xarr{\id \otimes \rho} X^{*4} \otimes \modobj \Big).
  \end{aligned}
\end{equation*}
for $X \in \mathcal{C}$. We now prove Lemma~\ref{lem:Radford-iso-ENO}, which states $\delta = \delta^{\mathrm{ENO}}$ in our notation.

\begin{proof}[Proof of Lemma~\ref{lem:Radford-iso-ENO}]
  For $X \in \mathcal{C}$, we set
  \begin{equation}
    \label{append:eq:zeta-def}
    \zeta_X = (\id_{X \boxtimes \unitobj} \otimes \din_{X^*}) \circ (\coev_{X} \boxtimes \id_X).
  \end{equation}
  Then, by \eqref{eq:append-tau}, we have
  \begin{equation}
    \label{append:eq:tau-tilde-2}
    \tau_X = m \circ (\zeta_X \otimes \id_A)
  \end{equation}
  We consider the following diagram:
  \begin{equation*}
    \begin{tikzcd}[column sep = 8em]
      (\unitobj \boxtimes X^{**}) \otimes A^{**}
      \arrow[d, "\zeta_X^{**} \otimes \id"']
      \arrow[r, "\id \otimes \phi^*"]
      & (\unitobj \boxtimes X^{**}) \otimes A^* \otimes (\modobj^* \boxtimes \unitobj)
      \arrow[d, "\zeta_X^{**} \otimes \id"] \\
      (X^{**} \boxtimes \unitobj) \otimes A^{**} \otimes A^{**}
      \arrow[d, "\id \otimes m^{**}"']
      \arrow[r, "\id \otimes \id \otimes \phi^*"]
      & (X^{**} \boxtimes \unitobj) \otimes A^{**} \otimes A^* \otimes (\modobj^* \boxtimes \unitobj)
      \arrow[d, "\id \otimes m^{\ddag} \otimes \id"] \\
      (X^{**} \boxtimes \unitobj) \otimes A^{**}
      \arrow[r, "\id \otimes \phi^*"]
      & (X^{**} \boxtimes \unitobj) \otimes A^* \otimes (\modobj^* \boxtimes \unitobj).
    \end{tikzcd}
  \end{equation*}
  Since $\phi: (\modobj \boxtimes \unitobj) \otimes A \to A^*$ is a morphism of right $A$-modules by definition, $\phi^*$ is a morphism of left $A^{**}$-modules. Hence we have
  \begin{equation}
    \label{eq:phi-A-linear}
    \phi^* \circ m^{**} = (m^{\ddag} \otimes \id_{\modobj^* \boxtimes \unitobj}) \circ (\id_{A^{**}} \otimes \phi^*),
  \end{equation}
  which implies that the above diagram commutes. By~\eqref{append:eq:tau-tilde-2}, the composition along the first row of the diagram is equal to $\tau_X^{**}$. Thus,
  \begin{align*}
    \widetilde{\tau}_X \otimes \id_{\modobj^* \boxtimes \unitobj}
    & = \begin{aligned}[t]
      & (\id_{X^{*4} \boxtimes \unitobj} \otimes \phi^{-1} \otimes \id_{\modobj^* \boxtimes \unitobj})
      \circ (\id_{X^{*4} \boxtimes \unitobj} \otimes \phi^{*}) \\
      & \qquad \circ (\tau_{X^{**}})^{**}
      \circ (\id_{\unitobj \boxtimes X} \otimes (\phi^{*})^{-1})
      \circ (\id_{\unitobj \boxtimes X} \otimes \phi \otimes \id_{\modobj^* \boxtimes \unitobj}) \\
    \end{aligned} \\
    & = (\id_{X^{*4} \boxtimes \unitobj} \otimes \phi^{-1} m^{\ddag} \otimes \id_{\modobj^* \boxtimes \unitobj})
      \circ (\zeta_{X^{**}}^{**} \otimes \phi \otimes \id_{\modobj^* \boxtimes \unitobj}).
  \end{align*}
  Since $\modobj^* \boxtimes \unitobj \in \mathcal{C}^{\env}$ is an invertible object, we have
  \begin{equation}
    \label{append:eq:tau-tilde-3}
    \begin{aligned}
    \widetilde{\tau}_X
    & = (\id_{X^{*4} \boxtimes \unitobj} \otimes \phi^{-1} m^{\ddag})
    \circ (\zeta_{X^{**}}^{**} \otimes \phi) \\
    & = (\id_{X^{*4} \boxtimes \unitobj} \otimes \widetilde{\rho})
    \circ (\zeta_{X^{**}}^{**} \otimes \id_{\modobj \boxtimes \unitobj} \otimes \id_A).      
    \end{aligned}
  \end{equation}
  Now we consider the commutative diagram
  \begin{equation}
    \label{append:eq:proof-thm-1}
    \begin{tikzcd}[column sep = 4em]
      (\unitobj \boxtimes X) \otimes \Kappa(\modobj)
      \arrow[d, "(\coev \boxtimes \id) \otimes \id"']
      \arrow[r, "\xi_{\unitobj \boxtimes X, \modobj}"]
      & \Kappa(\modobj \otimes X)
      \arrow[d, "\Kappa(\coev \otimes \id \otimes \id)"] \\
      (X^{*4} \boxtimes \unitobj) \otimes (X^{*5} \boxtimes X) \otimes \Kappa(\modobj)
      \arrow[d, "\id \otimes i_{X^{***}}^{**} \otimes \id"']
      \arrow[r, "\xi"]
      & \Kappa(X^{*4} \otimes X^{*5} \otimes \modobj \otimes X)
      \arrow[d, "\Kappa(\id \otimes (i_{X^{***}}^{**} \catactl \id)"] \\
      (X^{*4} \boxtimes \unitobj) \otimes A^{**} \otimes \Kappa(\modobj)
      \arrow[d, "\id \otimes \widetilde{\rho}"']
      \arrow[r, "\xi"]
      & \Kappa(X^{*4} \otimes (A^{**} \catactl \modobj))
      \arrow[d, "\Kappa(\id \otimes \rho)"] \\
      (X^{*4} \boxtimes \unitobj) \otimes \Kappa(\modobj)
      \arrow[r, "\xi_{(X^{*4} \boxtimes \unitobj), \modobj}"]
      & \Kappa(X^{*4} \otimes \modobj)
    \end{tikzcd}
  \end{equation}
  where $\xi_{M,V} : M \otimes \Kappa(V) \to \Kappa(M \catactl V)$ ($M \in \mathcal{C}^{\env}$, $V \in \mathcal{C}$) is the structure morphisms of $\Kappa$ as a left $\mathcal{C}^{\env}$-module functor. Note, by definition,
  \begin{equation*}
    \xi_{X \boxtimes \unitobj, Y} = \id_{\Kappa(X \otimes Y)}, \quad
    \xi_{\unitobj \boxtimes X, Y} = \id_{Y \boxtimes \unitobj} \otimes \tau_X
    \quad (X, Y \in \mathcal{C}).
  \end{equation*}
  The first row of the diagram is equal to $\widetilde{\tau}_X$ by \eqref{append:eq:zeta-def} and \eqref{append:eq:tau-tilde-3}, while the second row is $\Kappa(\delta_X)$ by definition. Hence we have
  \begin{equation*}
    \Kappa(\delta_X)
    = \xi_{X^{*4} \boxtimes \unitobj, D} \circ \widetilde{\tau}_X \circ (\xi_{\unitobj \boxtimes X, D})^{-1}
    = \Kappa(\deltaENO_X)
  \end{equation*}
  for all $X \in \mathcal{C}$. Since $\Kappa$ is an equivalence, $\delta = \deltaENO$. The proof is done.
\end{proof}

\subsection{The braided case}
\label{append:subsec:braided-case}

Let $\mathcal{C}$ be a braided finite tensor category with braiding $\sigma$. For $X \in \mathcal{C}$, we define the isomorphism $w_X: X \to X^{****}$ by
\begin{equation}
  \label{eq:append-def-w}
  w_X = u_X^{**} \circ (u_{\, {}^* \! X}^{*})^{-1}
  = (u_{X^*}^{*})^{-1} \circ u_X,
\end{equation}
where $u$ is the Drinfeld isomorphism given in \eqref{eq:Dri-iso-def}. Then the Radford isomorphism has the following expression:

\begin{theorem}
  \label{append:thm:Rad-iso-br}
  For all $X \in \mathcal{C}$, we have
  \begin{equation*}
    \renewcommand{\xarrlen}{3em}
    \delta_X = \Big( \modobj \otimes X
    \xarr{\sigma_{\modobj, X}} X \otimes \modobj
    \xarr{w_X \otimes \id} X^{*4} \otimes \modobj \Big),
  \end{equation*}
  where $\modobj = \modobj_{\mathcal{C}}$ is the distinguished invertible object of $\mathcal{C}$.
\end{theorem}

This theorem is a generalization of Radford's result on the distinguished grouplike elements of finite-dimensional quasitriangular Hopf algebras \cite{MR1182011}, and has been proved in \cite{MR2097289} in the unimodular case and in \cite[Theorem 8.10.7]{MR3242743} in the general case. In view of the importance of this formula, we give an alternative proof from the viewpoint of the equivalence ${}_{A^{**}}(\mathcal{C}^{\env})_A \approx \mathcal{Z}(\id_{\mathcal{C}}, S^4)$ used to define the Radford object.

At first, we work in a general setting. Let $\mathcal{B}$ be a braided rigid monoidal category with braiding $\sigma$, and let $A$ be an algebra in $\mathcal{B}$ with multiplication $m$. Then $A^*$ is an $A^{**}$-$A^*$-bimodule in $\mathcal{B}$. As in Section~\ref{sec:relative-modulus}, we denote by $m^{\ddag} : A^{**} \otimes A^* \to A^*$ and $m^{\dagger} : A^* \otimes A \to A^*$ the left and the right action on $A^{*}$, respectively. We suppose that $A$ is commutative, that is, the equation $m \circ \sigma_{A,A} = m$ holds. Then there is the following relation between $m^{\ddag}$ and $m^{\dagger}$.

\begin{lemma}
  \label{lem:append-braided-m-dagger-lemma}
  With above notation, we have $m^{\ddag} \circ (u_A \otimes \id_A) = m^{\dagger} \circ \sigma_{A,A^*}$.
\end{lemma}
\begin{proof}
  The commutativity of $A$ implies
  \begin{equation}
    \label{append:eq:commutative}
    \sigma_{A^*,A^*} \circ m^* = (\sigma_{A,A})^* \circ m^* = (m \circ \sigma_{A,A})^* = m^*.
  \end{equation}
  The claim is verified with the help of string diagrams as follows:
  \begin{equation*}
    \tikzset{cross/.style={preaction={-,draw=white,line width=6pt}}}
    m^{\ddag} \circ (u_A \otimes \id_{A^*})
    =
    \begin{tikzpicture}[x = 1pc, y = 1pc, baseline=0]
      \coordinate (S1) at (-2,2); \node at (S1) [above] {$A$};
      \node (B1) at (0,1.25) [draw, below] {\makebox[2em]{$\phantom{m^*}$}};
      \draw let \p1 = (B1.north) in (\p1) -- (\x1, 2)
      node [above] {$A^*$};
      \draw let \p1 = (S1), \p2 = ([shift={(-.5,0)}]B1.south)
      in (\p1) -- (\x1, \y2) coordinate (Q1)
      to [out=-90, in=90, looseness=1] ($(\x2, \y2)-(0,1.5)$) coordinate (Q2);
      \draw [cross] let \p1 = (Q1), \p2 = (Q2)
      in (Q2) to [out=-90, in=-90, looseness=1.5] (\x1, \y2)
      to [out=90, in=-90] ([shift={(-.5,0)}]B1.south);
      \draw let \p1 = ([shift={(+.5,0)}]B1.south)
      in (\p1) -- (\x1, -2) node [below] {$A^{*}$};
      \node (B1) at (0,1.25) [draw, below] {\makebox[2em]{$m^*$}};
    \end{tikzpicture}
    =
    \begin{tikzpicture}[x = 1pc, y = 1pc, baseline=0]
      \coordinate (S1) at (0,2); \node at (S1) [above] {$A$};
      \coordinate (S2) at (1.75,2); \node at (S2) [above] {$A^*$};
      \node (B1) at (0,.5) [below] {\makebox[2em]{$\phantom{m^*}$}};
      \draw let \p1 = (S1), \p2 = (S2), \p3 = (B1.north), \p4 = (B1.south)
      in (\p1) to [out=-90, in=90, looseness=1] (\x2, \y3) -- (\x2, \y4) coordinate (Q1);
      \draw [cross] (S2) to [out=-90, in=90, looseness=1] (B1.north);
      \draw ([shift={(-.5,0)}]B1.south) coordinate (Q2)
      to [out=-90, in=-90, looseness=1.75] (Q1);
      \draw [cross] let \p1 = ([shift={(+.5,0)}]B1.south), \p2 = (Q2)
      in (\p1) to [out=-90, in=90] ($(\x2, -2)+(-.25,0)$) node [below] {$A^*$};
      \node (B1) at (0,.5) [draw, below] {\makebox[2em]{$m^*$}};
    \end{tikzpicture}
    \!\!\! \mathop{=}^{\eqref{append:eq:commutative}}
    \begin{tikzpicture}[x = 1pc, y = 1pc, baseline=0]
      \coordinate (S1) at (0,2); \node at (S1) [above] {$A$};
      \coordinate (S2) at (1.75,2); \node at (S2) [above] {$A^*$};
      \node (B1) at (0,.5) [below] {\makebox[2em]{$\phantom{m^*}$}};
      \draw let \p1 = (S1), \p2 = (S2), \p3 = (B1.north), \p4 = (B1.south)
      in (\p1) to [out=-90, in=90, looseness=1] (\x2, \y3) -- (\x2, \y4) coordinate (Q1);
      \draw [cross] (S2) to [out=-90, in=90, looseness=1] (B1.north);
      \draw ([shift={(.5,0)}]B1.south) coordinate (Q2)
      to [out=-90, in=-90, looseness=2] (Q1);
      \draw let \p1 = ([shift={(-.5,0)}]B1.south)
      in (\p1) to [out=-90, in=90] (\x1, -2) node [below] {$A^*$};
      \node (B1) at (0,.5) [draw, below] {\makebox[2em]{$m^*$}};
    \end{tikzpicture}
    = m^{\dagger} \circ \sigma_{A,A^*}. \qedhere
  \end{equation*}
\end{proof}

Suppose that there is an object $X \in \mathcal{B}$ and an isomorphism $\phi : X \otimes A \to A^*$ of right $A$-modules in $\mathcal{B}$. We make the object $X \otimes A$ an $A^{**}$-$A$-bimodule in $\mathcal{B}$ in such a way that $\phi$ is an isomorphism of left $A^{**}$-modules.

\begin{lemma}
  \label{lem:append-rho-tilde}
  The left action $\lambda : A^{**} \otimes X \otimes A \to X \otimes A$ is given by
  \begin{equation*}
    \lambda = (\id_{X} \otimes m) \circ (\sigma_{A, X} \otimes \id_A) \circ (u^{-1}_A \otimes \id_{X} \otimes \id_A).    
  \end{equation*}
\end{lemma}
\begin{proof}
  By definition, we have $\lambda = \phi^{-1} \circ m^{\ddag} \circ (\id_{A^{**}} \otimes \phi)$. Thus,
  \begin{align*}
    \lambda \circ (u_A \otimes \id_{A^*})
    & = \phi^{-1} \circ m^{\dagger} \circ \sigma_{A,A^*} \circ (\id_{A} \otimes \phi) \\
    & = (\id_{X} \otimes m) \circ (\phi^{-1} \otimes \id_A) \circ \sigma_{A,A^*} \circ (\id_{A} \otimes \phi) \\
    & = (\id_{X} \otimes m) \circ \sigma_{A, X \otimes A}
      \circ (\id_{A} \otimes \phi^{-1}) \circ (\id_{A} \otimes \phi) \\
    & = (\id_{X} \otimes m) \circ (\id_{X} \otimes \sigma_{A,A})
      \circ (\sigma_{A, X} \otimes \id_A) \\
    & = (\id_{X} \otimes m) \circ (\sigma_{A, X} \otimes \id_A).
  \end{align*}
  Here,
  the first equality follows from Lemma~\ref{lem:append-braided-m-dagger-lemma},
  the second from that $\phi : A^* \to X \otimes A$ is a morphism of right $A$-modules,
  and the last from the commutativity of $A$. The proof is done.
\end{proof}

We apply the above result for algebras in a braided rigid monoidal category to the canonical algebra $A \in \mathcal{C}^{\env}$ as follows:
We first endow $\mathcal{C}^{\env}$ with the braiding $\widetilde{\sigma}$ determined by
$\widetilde{\sigma}_{X \boxtimes X', Y \boxtimes Y'} = \sigma_{X,Y} \boxtimes \sigma_{X',Y'}^{-1}$ ($X, Y, X', Y' \in \mathcal{C}$).
Then the canonical algebra $A \in \mathcal{C}^{\env}$ is commutative. Indeed, we have
\begin{gather*}
  m \widetilde{\sigma}_{A,A} (\din_X \otimes \din_Y)
  = m (\din_Y \otimes \din_X) \widetilde{\sigma}_{X \boxtimes {}^* \! X, Y \boxtimes {}^* \! Y}
  \mathop{=}^{\eqref{eq:cano-alg-mult}} \din_{Y \otimes X}  (\sigma_{X,Y} \boxtimes {}^*(\sigma_{X, Y}^{-1})) \\
  = \din_{X \otimes Y}  (\sigma_{X,Y}^{-1} \sigma_{X,Y} \boxtimes \id_{\, {}^*(X \otimes Y)})
  = \din_{X \otimes Y} \mathop{=}^{\eqref{eq:cano-alg-mult}} m (\din_X \otimes \din_Y)
\end{gather*}
for all objects $X, Y \in \mathcal{C}$.
We denote by $\widetilde{u}$ the Drinfeld isomorphism of $\mathcal{C}^{\env}$.
Then, by the above lemma, the action $\widetilde{\rho} : A^{**} \otimes \Kappa(\modobj) \to \Kappa(\modobj)$ is given by
\begin{equation*}
  \widetilde{\rho} = (\id_{\modobj \boxtimes \unitobj} \otimes m) \circ (\sigma_{A, \modobj \boxtimes \unitobj} \otimes \id_A) \circ (\widetilde{u}^{-1}_A \otimes \id_{\modobj \boxtimes \unitobj} \otimes \id_A).    
\end{equation*}
The Radford isomorphism is induced by $\widetilde{\rho}$.
Theorem~\ref{append:thm:Rad-iso-br} is now proved by computing the Radford isomorphism with the expression of $\widetilde{\rho}$ given in the above, although a bit lengthy computation will be required for checking it.

\begin{proof}[Proof of Theorem~\ref{append:thm:Rad-iso-br}]
  It is straightforward to verify
  \begin{equation}
    \label{eq:append-u-tilde}
    \widetilde{u}_{X \boxtimes Y} = u_X \boxtimes {}^{**}(u_{Y}^{-1})
    \quad (X, Y \in \mathcal{C}).
  \end{equation}
  By the (di)naturality of $\widetilde{u}$ and $\din$, we have
  \begin{gather*}
    \widetilde{u}^{-1}_A \circ \din_{V^{***}}^{**}
    = \din_{V^{***}} \circ \widetilde{u}^{-1}_{V^{***} \boxtimes V^{**}}
    \mathop{=}^{\eqref{eq:append-u-tilde}} \din_{V^{***}} \circ (u_{V^{***}}^{-1} \boxtimes {}^{**}u_{V^{**}}^{}) \\
    = \din_{V^{*}} \circ ({}^* u_{V^{**}}^{} u_{V^{***}}^{-1} \boxtimes \id_{V})
    \mathop{=}^{\eqref{eq:append-def-w}} \din_{V^{*}} \circ (w_V^* \boxtimes \id_{V})
  \end{gather*}
  for all $V \in \mathcal{C}$. Hence, by Lemma \ref{lem:append-rho-tilde}, we have
  \begin{align*}
    & \widetilde{\rho} \circ (\din_{V^{***}}^{**} \otimes \id_{\modobj \boxtimes \unitobj} \otimes \din_W) \\
    & = (\id_{\modobj \boxtimes \unitobj} \otimes m) \circ (\sigma_{A, \modobj \boxtimes \unitobj} \otimes \id_A)
      \circ (\widetilde{u}^{-1}_A \din_{V^{***}}^{**} \otimes \id_{\modobj \boxtimes \unitobj} \otimes \din_W) \\
    & = (\id_{\modobj \boxtimes \unitobj} \otimes m) \circ (\sigma_{A, \modobj \boxtimes \unitobj} \otimes \id_A)
      \circ (\din_{V^{*}} (w_V^* \boxtimes \id_V) \otimes \id_{\modobj \boxtimes \unitobj} \otimes \din_W) \\
    & = (\id_{\modobj \boxtimes \unitobj} \otimes m (\din_{V^*} \otimes \din_W))
      \circ (\sigma_{V^* \boxtimes V, \modobj \boxtimes \unitobj}((w_V^* \boxtimes \id_V) \otimes \id_{\modobj \boxtimes \unitobj})
      \otimes \id_{W \boxtimes {}^*W}) \\
    & = (\id_{\modobj \boxtimes \unitobj} \otimes \din_{V^* \otimes W})
      \circ ((\sigma_{V^*, \modobj}(w_V^* \otimes \id_{\modobj}) \boxtimes \id_V) \otimes \id_{W \boxtimes {}^*W})
  \end{align*}
  for all $V, W \in \mathcal{C}$.
  Thus the diagram given as Figure~\ref{fig:proof-Rad-iso-br} commutes. By the dinaturality of $\din$ and $\coev_{{}^*\!X} = {}^*(\eval_X)$, we have
  \begin{equation*}
    \din_{X^* \otimes X \otimes Y}
    \circ (\id_{X^* \otimes X \otimes Y} \otimes (\id_{{}^*Y} \otimes \coev_{{}^*\!X}))
    = \din_{Y}
    \circ (\eval_X \otimes \id_{Y}) \otimes \id_{{}^*Y}).
  \end{equation*}
  Thus, by the commutative diagram, we see that $\Kappa(\delta_X)$ is equal to
  \begin{equation*}
    \Kappa\Big((\id_{X^{*4}} \otimes \eval_X) (\id_{X^{*4}} \otimes \sigma_{X^*,\modobj}) (\id_{X^{*4}} \otimes w_X^* \otimes \id_{\modobj})(\coev_{X^{*4}} \otimes \id_{\modobj})\Big).
  \end{equation*}
  for all $X \in \mathcal{C}$. Since the functor $\Kappa$ is an equivalence, we have
  \begin{gather*}
    \delta_X
    = (\id_{X^{*4}} \otimes \eval_X) (\id_{X^{*4}} \otimes \sigma_{X^*,\modobj}) (\id_{X^*4} \otimes w_X^* \otimes \id_{\modobj})(\coev_{X^{*4}} \otimes \id_{\modobj}) \\
    = (w_X \otimes \eval_X) (\id_{X} \otimes \sigma_{X^*,\modobj}) (\coev_{X} \otimes \id_{\modobj})
    = (w_X \otimes \id_{\modobj}) \circ \sigma_{\modobj, X}^{-1}.
  \end{gather*}
  This formula is different from the claim of this theorem. Nevertheless, by using this formula, one can prove $\sigma_{X, \modobj} = \sigma_{\modobj, X}^{-1}$ ($X \in \mathcal{C}$) in the same way as \cite[Corollary 8.10.8]{MR3242743}. The proof is done.
\end{proof}

\begin{figure}
  \centering
  \makebox[\textwidth][c]{
    \begin{tikzcd}[ampersand replacement=\&, column sep = 0em, row sep = 2em]
      \& [-2em] 
      (\modobj \boxtimes \unitobj) \otimes (X \boxtimes \unitobj) \otimes (Y \boxtimes {}^*Y)
      \arrow[rd, phantom, "\scriptstyle \text{($\circlearrowright$ by \eqref{eq:append-tau-inverse})}"]
      \arrow[r, equal]
      \arrow[d, "\id \otimes \id \otimes \din_Y"]
      \& [1em] 
      (\modobj \otimes X \otimes Y) \boxtimes {}^*Y
      \arrow[dd, "\id_{\modobj \otimes X \otimes Y} \boxtimes (\id_{{}^*Y} \otimes \coev_{{}^*\!X})"] \\
      \Kappa(\modobj \otimes X) \arrow[r, equal]
      \arrow[rd, phantom, "\scriptstyle \text{($\circlearrowright$ by \eqref{append:eq:proof-thm-1})}"]
      \arrow[dddd, pos = .25, "\Kappa(\delta_X)"]
      \& (\modobj \boxtimes \unitobj) \otimes (X \boxtimes \unitobj) \otimes A
      \arrow[d, "\id_{\modobj \boxtimes \unitobj} \otimes \tau_X^{-1}"]
      \& \mbox{} \\[-1.5em]
      \& (\unitobj \boxtimes X) \otimes (\modobj \boxtimes \unitobj) \otimes A
      \arrow[d, "(\coev_{X^{*4}} \boxtimes \id) \otimes \id \otimes \id"]
      \& \begin{gathered}[b]
        (\modobj \otimes X \otimes Y) \boxtimes ({}^*Y \otimes {}^* \!X \otimes X) = \\
        (\modobj \boxtimes X) \otimes ((X \otimes Y) \boxtimes {}^*(X \otimes Y))
      \end{gathered}
      \arrow[d, "((\coev_{X^{*4}} \otimes \id) \boxtimes \id) \otimes \id"]
      \arrow[l, "\id_{D \boxtimes X} \otimes \din_{X \otimes Y}"'] \\
      \& (X^{*4} \boxtimes \unitobj) \otimes (X^{*5} \boxtimes X) \otimes (D \boxtimes \unitobj) \otimes A
      \arrow[d, "\id \otimes \din_{X^{***}}^{**} \otimes \id \otimes \id"]
      \& ((X^{*4} \otimes X^{*5} \otimes \modobj) \boxtimes X) \otimes ((X \otimes Y) \boxtimes {}^*(X \otimes Y))
      \arrow[l, "\id \otimes \din_{X \otimes Y}\strut"]
      \arrow[dd, "((\id \otimes \sigma_{X^*,D}(w_X^* \otimes \id_D)) \boxtimes \id) \otimes \id"] \\
      \& (X^{*4} \boxtimes \unitobj) \otimes A^{**} \otimes (D \boxtimes \unitobj) \otimes A
      \arrow[d, "\id \otimes \widetilde{\rho}"] \\[-1.5em]
      \Kappa(X^{*4} \otimes D)
      \arrow[r, equal]
      \& (X^{*4} \boxtimes \unitobj) \otimes (D \boxtimes \unitobj) \otimes A
      \& \begin{gathered}[b]
        ((X^{*4} \otimes \modobj \otimes X^*) \boxtimes X) \otimes ((X \otimes Y) \boxtimes {}^*(X \otimes Y)) = \\
        ((X^{*4} \otimes \modobj) \boxtimes \unitobj)
        \otimes ((X^* \otimes X \otimes Y) \boxtimes {}^*(X^* \otimes X \otimes Y))
      \end{gathered}
      \arrow[l, "\id \otimes \din_{X^* \otimes X \otimes Y}\strut"]
    \end{tikzcd}}
  \label{fig:proof-Rad-iso-br}
  \caption{Proof of Theorem~\ref{append:thm:Rad-iso-br}}
\end{figure}


\begin{thebibliography}{EGNO15}

\bibitem[BN14]{MR3161401}
Alain Brugui{\`e}res and Sonia Natale.
\newblock Central exact sequences of tensor categories, equivariantization and
  applications.
\newblock {\em J. Math. Soc. Japan}, 66(1):257--287, 2014.

\bibitem[BV12]{MR2869176}
Alain Brugui{\`e}res and Alexis Virelizier.
\newblock Quantum double of {H}opf monads and categorical centers.
\newblock {\em Trans. Amer. Math. Soc.}, 364(3):1225--1279, 2012.

\bibitem[BW99]{MR1686423}
John~W. Barrett and Bruce~W. Westbury.
\newblock Spherical categories.
\newblock {\em Adv. Math.}, 143(2):357--375, 1999.

\bibitem[DNO13]{MR3022755}
Alexei Davydov, Dmitri Nikshych, and Victor Ostrik.
\newblock On the structure of the {W}itt group of braided fusion categories.
\newblock {\em Selecta Math. (N.S.)}, 19(1):237--269, 2013.

\bibitem[DSPS19]{MR3934626}
Christopher~L. Douglas, Christopher Schommer-Pries, and Noah Snyder.
\newblock The balanced tensor product of module categories.
\newblock {\em Kyoto J. Math.}, 59(1):167--179, 2019.

\bibitem[DSS13]{2013arXiv1312.7188D}
Christopher~L. {Douglas}, Christopher {Schommer-Pries}, and Noah {Snyder}.
\newblock {Dualizable tensor categories}.
\newblock {\em arXiv e-prints}, page arXiv:1312.7188, December 2013.

\bibitem[EGNO15]{MR3242743}
Pavel Etingof, Shlomo Gelaki, Dmitri Nikshych, and Victor Ostrik.
\newblock {\em Tensor categories}, volume 205 of {\em Mathematical Surveys and
  Monographs}.
\newblock American Mathematical Society, Providence, RI, 2015.

\bibitem[ENO04]{MR2097289}
Pavel Etingof, Dmitri Nikshych, and Viktor Ostrik.
\newblock An analogue of {R}adford's {$S^4$} formula for finite tensor
  categories.
\newblock {\em Int. Math. Res. Not.}, (54):2915--2933, 2004.

\bibitem[EO04]{MR2119143}
Pavel Etingof and Viktor Ostrik.
\newblock Finite tensor categories.
\newblock {\em Mosc. Math. J.}, 4(3):627--654, 782--783, 2004.

\bibitem[FMS97]{MR1401518}
D.~Fischman, S.~Montgomery, and H.-J. Schneider.
\newblock Frobenius extensions of subalgebras of {H}opf algebras.
\newblock {\em Trans. Amer. Math. Soc.}, 349(12):4857--4895, 1997.

\bibitem[FSS17]{MR3638361}
J\"urgen Fuchs, Gregor Schaumann, and Christoph Schweigert.
\newblock A trace for bimodule categories.
\newblock {\em Appl. Categ. Structures}, 25(2):227--268, 2017.

\bibitem[FSS20]{MR4042867}
J\"{u}rgen Fuchs, Gregor Schaumann, and Christoph Schweigert.
\newblock Eilenberg-{W}atts calculus for finite categories and a bimodule
  {R}adford {$S^4$} theorem.
\newblock {\em Trans. Amer. Math. Soc.}, 373(1):1--40, 2020.

\bibitem[KL01]{MR1862634}
T. Kerler and V.~V. Lyubashenko.
\newblock {\em Non-semisimple topological quantum field theories for
  3-manifolds with corners}, volume 1765 of {\em Lecture Notes in Mathematics}.
\newblock Springer-Verlag, Berlin, 2001.

\bibitem[KR93]{MR1231205}
L.~H. Kauffman and D.~E. Radford.
\newblock A necessary and sufficient condition for a finite-dimensional
  {D}rinfel\cprime d double to be a ribbon {H}opf algebra.
\newblock {\em J. Algebra}, 159(1):98--114, 1993.

\bibitem[Lyu95a]{MR1324034}
V.~Lyubashenko.
\newblock Modular transformations for tensor categories.
\newblock {\em J. Pure Appl. Algebra}, 98(3):279--327, 1995.

\bibitem[Lyu95b]{MR1352517}
V. Lyubashenko.
\newblock Modular properties of ribbon abelian categories.
\newblock In {\em Proceedings of the 2nd {G}auss {S}ymposium. {C}onference {A}:
  {M}athematics and {T}heoretical {P}hysics ({M}unich, 1993)}, Sympos.
  Gaussiana, pages 529--579, Berlin, 1995. de Gruyter.

\bibitem[Lyu95c]{MR1354257}
V. Lyubashenko.
\newblock Invariants of {$3$}-manifolds and projective representations of
  mapping class groups via quantum groups at roots of unity.
\newblock {\em Comm. Math. Phys.}, 172(3):467--516, 1995.

\bibitem[Maj91]{MR1151906}
S. Majid.
\newblock Representations, duals and quantum doubles of monoidal categories.
\newblock In {\em Proceedings of the {W}inter {S}chool on {G}eometry and
  {P}hysics ({S}rn\'\i, 1990)}, number~26, pages 197--206, 1991.

\bibitem[ML98]{MR1712872}
S. Mac~Lane.
\newblock {\em Categories for the working mathematician}, volume~5 of {\em
  Graduate Texts in Mathematics}.
\newblock Springer-Verlag, New York, second edition, 1998.

\bibitem[M{\"u}g10]{MR2681261}
M. M{\"u}ger.
\newblock Tensor categories: a selective guided tour.
\newblock {\em Rev. Un. Mat. Argentina}, 51(1):95--163, 2010.

\bibitem[NS07]{MR2381536}
S.-H. Ng and P. Schauenburg.
\newblock Higher {F}robenius-{S}chur indicators for pivotal categories.
\newblock In {\em Hopf algebras and generalizations}, volume 441 of {\em
  Contemp. Math.}, pages 63--90. Amer. Math. Soc., Providence, RI, 2007.

\bibitem[Rad92]{MR1182011}
D.~E. Radford.
\newblock On the antipode of a quasitriangular {H}opf algebra.
\newblock {\em J. Algebra}, 151(1):1--11, 1992.

\bibitem[Sch04]{MR2095575}
P.~Schauenburg.
\newblock On the {F}robenius-{S}chur indicators for quasi-{H}opf algebras.
\newblock {\em J. Algebra}, 282(1):129--139, 2004.

\bibitem[Shi15]{MR3314297}
K.~Shimizu.
\newblock The pivotal cover and {F}robenius-{S}chur indicators.
\newblock {\em J. Algebra}, 428:357--402, 2015.

\bibitem[{Shi}16]{2016arXiv160805905S}
K.~{Shimizu}.
\newblock {Pivotal structures of the Drinfeld center of a finite tensor
  category}.
\newblock arXiv:1608.05905, 2016, submitted.

\bibitem[Shi17a]{MR3632104}
K.~Shimizu.
\newblock On unimodular finite tensor categories.
\newblock {\em Int. Math. Res. Not. IMRN}, (1):277--322, 2017.

\bibitem[Shi17b]{MR3569179}
K.~Shimizu.
\newblock The relative modular object and {F}robenius extensions of finite
  {H}opf algebras.
\newblock {\em J. Algebra}, 471:75--112, 2017.

\bibitem[Shi19]{MR3996323}
K.~Shimizu.
\newblock Non-degeneracy conditions for braided finite tensor categories.
\newblock {\em Adv. Math.}, 355:106778, 36, 2019.

\end{thebibliography}
\def\cprime{$'$}

\end{document}